\documentclass[american]{article}
\usepackage[T1]{fontenc}
\usepackage[latin9]{inputenc}
\usepackage{longtable}
\usepackage{amsmath}
\usepackage{amsthm}
\usepackage{amssymb}
\usepackage{wasysym}

\makeatletter

\providecommand{\tabularnewline}{\\}

\numberwithin{equation}{section}
  \theoremstyle{plain}
  \newtheorem*{thm*}{\protect\theoremname}
\theoremstyle{plain}
\newtheorem{thm}{\protect\theoremname}[section]
  \theoremstyle{remark}
  \newtheorem{rem}[thm]{\protect\remarkname}
  \theoremstyle{plain}
  \newtheorem{lem}[thm]{\protect\lemmaname}
  \theoremstyle{definition}
  \newtheorem{defn}[thm]{\protect\definitionname}
  \theoremstyle{plain}
  \newtheorem{prop}[thm]{\protect\propositionname}
  \theoremstyle{plain}
  \newtheorem{cor}[thm]{\protect\corollaryname}
  \theoremstyle{remark}
  \newtheorem*{rem*}{\protect\remarkname}

\newcommand{\hide}[1]{}

\DeclareMathOperator{\id}{id}

\DeclareMathOperator*{\im}{image}

\DeclareMathOperator{\lip}{Lip}

\DeclareMathOperator{\rank}{rank}

\DeclareMathOperator{\diam}{diam}
\DeclareMathOperator{\supp}{supp}

\DeclareMathOperator{\bdim}{dim_\textrm{B}}

\DeclareMathOperator{\edim}{dim_\textrm{e}}
\DeclareMathOperator{\ledim}{\underline{\dim_e}}
\DeclareMathOperator{\uedim}{\overline{\dim_e}}

\def\conv{\mbox{\LARGE{$.$}}}

\DeclareMathOperator{\spn}{span}
\DeclareMathOperator{\diag}{diag}

\newcommand{\PR}{\mathbb{RP}^1}

\def\conv{\mbox{\LARGE{$.$}}}
\renewcommand{\textendash}{--}

\date{\today}

\makeatother

\usepackage{babel}
  \providecommand{\corollaryname}{Corollary}
  \providecommand{\definitionname}{Definition}
  \providecommand{\lemmaname}{Lemma}
  \providecommand{\propositionname}{Proposition}
  \providecommand{\remarkname}{Remark}
  \providecommand{\theoremname}{Theorem}
\providecommand{\theoremname}{Theorem}

\begin{document}

\title{\date{}Hausdorff Dimension of Planar Self-Affine Sets and Measures
with Overlaps}

\author{Michael Hochman and Ariel Rapaport}
\maketitle
\begin{abstract}
We prove that if $\mu$ is a self-affine measure in the plane whose
defining IFS acts totally irreducibly on $\PR$ and satisfies an exponential
separation condition, then its dimension is equal to its Lyapunov
dimension. We also treat a class of reducible systems. This extends
our previous work on the subject with B\'ar\'any to the overlapping
case.
\end{abstract}
\tableofcontents{}

\section{\label{sec:Introduction}Introduction}

\subsection{\label{subsec:Statement-of-results}Statement of results}

Let $X=\bigcup_{i\in\Lambda}\varphi_{i}X\subseteq\mathbb{R}^{2}$
be a planar self-affine set, and let $\mu=\sum_{i\in\Lambda}p_{i}\cdot\varphi_{i}\mu\in\mathcal{P}(\mathbb{R}^{2})$
be a planar self-affine measure, generated by a finite system $\Phi=\{\varphi_{i}\}_{i\in\Lambda}$
of invertible affine contractions of $\mathbb{R}^{2}$ and a probability
vector $p=(p_{i})_{i\in\Lambda}$. To avoid trivial cases we assume
throughout this paper (and without further mention) that 
\begin{itemize}
\item The maps $\varphi_{i}$ do not have a common fixed point;
\item $p_{i}>0$ for all $i\in\Lambda$. 
\end{itemize}
We write $\varphi_{i}(x)=A_{i}x+b_{i}$ where $A_{i}$ is a $2\times2$
matrix and $b_{i}\in\mathbb{R}^{2}$, and for a general affine map
$\varphi$ of $\mathbb{R}^{2}$ we similarly write $\varphi(x)=A_{\varphi}x+b_{\varphi}$.

It has been a longstanding problem to compute the dimensions $\dim X$
and $\dim\mu$. General upper bounds have been known for some time:
the affinity dimension $\dim_{a}X$ bounds the dimension of $X$ \cite{Falconer1988},
and the Lyapunov dimension $\dim_{L}\mu$ bounds the dimension of
$\mu$ \cite{JordanPollicottSimon2007}.\footnote{Strictly speaking, the affinity and Lyapunov dimensions depend on
$\Phi$ and $p$, not on $X$ and $\mu$, but we suppress this in
our notation.} Another, trivial, upper bound is the dimension $2$ of the ambient
space $\mathbb{R}^{2}$; thus we obtain the general bound
\begin{align}
\dim X & \leq\min\{2,\dim_{a}X\}\label{eq:X-bound}\\
\dim\mu & \leq\min\{2,\dim_{L}\mu\}\;.\label{eq:mu-bound}
\end{align}
It is a natural question to ask when $X$ and $\mu$ are ``as spread
out as possible'', that is, when these bounds are achieved. Equality
turns out to be the situation for ``typical'' $\Phi$, as has been
established in many instances over the past few decades, most often
as the generic behavior in various parametric families of systems,
and in some special cases of concrete systems, see e.g. \cite{Falconer1992,HueterLalley1995,Solomyak1998,BaranyKaenmakiKoivusalo2018}.
This behavior is not universal, and some counterexamples are known,
but they are rather special, consisting either of systems in which,
in suitable coordinates, the matrices $A_{i}$ are all diagonal \cite{Bedford1989-box-dimension-of-repellers,McMullen1984},
see also \cite{GatzourasLalley1992-dim-of-certain-self-affine-fractals};
or of systems with many ``overlaps'', that is to say, systems in
which there are many algebraic relations in the semigroup generated
by $\Phi$.

Over the past few years results have emerged that apply to specific
instances of systems \cite{Barany2015,FalconerKempton2016,Rapaport2016,MorrisShmerkin2016},
under some separation assumption and assumptions on the dimension
of the associated Furstenberg measure. Most recently, in joint work
with B. B\'{a}r\'{a}ny, we removed the last assumption and proved
the following general result:
\begin{thm*}
[\cite{BHR}] Suppose that $\Phi=\{\varphi_{i}\}_{i\in\Lambda}$ is
a finite system of invertible affine contractions in $\mathbb{R}^{2}$
and satisfies the following conditions:
\begin{itemize}
\item \textbf{Non-conformality}: There is no coordinate system in which
all the maps $\varphi_{i}$ are similarities.
\item \textbf{Total irreducibility}: There is no finite set of lines $\{\ell_{1},\ldots,\ell_{n}\}$
in $\mathbb{R}^{2}$ which is invariant under all of the matrices
$A_{i}$.
\item \textbf{Strong~open~set~condition}: There is a bounded open set
$U\subseteq\mathbb{R}^{2}$ such that $U\cap X\neq\emptyset$, $\varphi_{i}U\subseteq U$
for all $i\in\Lambda$, and $\varphi_{i}U\cap\varphi_{j}U=\emptyset$
for distinct $i,j\in\Lambda$.
\end{itemize}
Then equality holds in (\ref{eq:X-bound}) and (\ref{eq:mu-bound}).
\end{thm*}
The first assumption, non-conformality, is not actually necessary
for the conclusion to hold, because under the separation assumption
given, the conformal (or self-similar) case is easily dealt with using
classical methods. It was stated here and in our earlier paper because
the methods in the conformal and non-conformal settings turn out to
be quite different. 

The second assumption, total irreducibility, can be replaced with
weaker assumptions for some systems of triangular matrices \cite[Proposition 6.6]{BHR},
but cannot be eliminated entirely, as is shown by carpet-like examples.

The purpose of the present paper is to replace the third assumption,
the strong open set condition, with a substantially weaker one, analogous
to the state-of-the-art in the conformal case \cite{Ho,HO2}. This
is of intrinsic interest, as it is a step towards eliminating the
separation assumption entirely (a possibility which, at present, is
only conjectural). As further motivation, we anticipate that understanding
the overlapping two-dimensional case will be an important step towards
treating the separated case in higher dimensions; we will explain
this point in more detail below. Finally, although our previous work
concerned the same non-conformal class of fractals as here, in fact
the proof there reduced to dealing with a family of conformal-like
fractals on the line. The present work requires genuinely non-conformal
techniques, which we introduce here. These are of independent interest.

To state our main result we fix a left-invariant metric $d$, derived
from a Riemannian metric, on the group $A_{2,2}$ of invertible affine
maps $\mathbb{R}^{2}\rightarrow\mathbb{R}^{2}$. We say that the system
$\{\varphi_{i}\}$ satisfies exponential separation if there exists
a constant $c>0$ such that for every $n\in\mathbb{N}$ and for every
pair of sequences $i_{1}\ldots i_{n}\neq j_{1}\ldots j_{n}$ in $\Lambda^{n}$,
we have
\begin{equation}
d(\varphi_{i_{1}}\ldots\varphi_{i_{n}},\varphi_{j_{1}}\ldots\varphi_{j_{n}})>c^{n}\;.\label{eq:exp-sep}
\end{equation}
Note that the constant $c$ will depend on the choice of metric, but
the existence of such a constant is independent of the metric. Other
metrics would also serve for this purpose, e.g. the norm metric when
the affine maps are viewed as $3\times3$ matrices in the standard
way.
\begin{thm}
\label{thm:main}Let $\Phi=\{\varphi_{i}\}_{i\in\Lambda}$ be a finite
system of invertible affine contractions of $\mathbb{R}^{2}$, and
suppose that $\Phi$ has no common fixed point, satisfies the non-conformality
and total irreducibility assumptions, and is exponentially separated.
Then, writing $X$ for the attractor, we have 
\[
\dim X=\min\{2,\dim_{a}X\}\;.
\]
Furthermore, for any positive probability vector $p$, the associated
self-affine measure $\mu=\sum p_{i}\cdot\varphi_{i}\mu$ satisfies
\begin{equation}
\dim\mu=\min\{2,\dim_{L}\mu\}\;.\label{eq:Lyap-dim}
\end{equation}
\end{thm}

The first statement follows from the second using a variational principle
due to Morris and Shmerkin \cite{MorrisShmerkin2016}. We therefore
focus on calculating the dimension of $\mu$. 

For Theorem \ref{thm:main} and other theorems below which assume
exponential separation, it is enough to assume the weaker property
that there exists a $c>0$ for which, for infinitely many $n$, (\ref{eq:exp-sep})
holds over all distinct choices $\mathbf{i},\mathbf{j}\in\Lambda^{n}$.
This is true also for the results in \cite{BHR} and several other
recent works on the subject. The proof requires almost no modification,
see \cite{Ho} where it is given on the line. We continue to state
our results in the case of exponential separation because this has
become customary and holds in many important cases, but one should
remember that it can be weakened, and can be significant, see e.g.
\cite{VARJU-BC-PAPER}.

A version of Theorem \ref{thm:main} holds also in terms of random
walk entropy. Specifically, suppose that (\ref{eq:exp-sep}) holds
for all $n$ (or for arbitrarily large $n$) for all pairs $\mathbf{i},\mathbf{j}\in\Lambda^{n}$
such that $\varphi_{\mathbf{i}}\neq\varphi_{\mathbf{j}}$. Then (\ref{eq:Lyap-dim})
holds, but we must define the Lyapunov dimension not with respect
to the entropy $H(p)$ of $p$, but rather with respect to the random
walk entropy $H_{RW}(\Phi,p)$ of the random walk on the affine group
generated by $\Phi$ and $p$. The proof of this requires only minor
modifications (specifically, to Proposition \ref{prop:decomp of desc of slices},
although not to its statement), and is by now well understood, so
we omit the details.

\subsection{\label{subsec:Discussion-and-reduction}Discussion and reduction}

A central tool in this theory is the Ledrappier-Young formula, which
in the setting of self-affine measures is due to K\"aenm\"aki and B\'ar\'ani
\cite{BK,Feng2019}, and which we now recall (see also Section \ref{sub:furstenberg}).
Let $\eta^{*}$ denote the Furstenberg measure of the i.i.d. random
matrix product $\zeta_{n}\cdot\zeta_{n-1}\cdot\ldots\cdot\zeta_{1}$
where $\zeta_{i}$ takes the value $A_{i}^{*}$ with probability $p_{i}$.
This is the unique measure on the projective line $\PR$ satisfying
the relation $\eta^{*}=\sum p_{i}\cdot A_{i}^{*}\eta^{*}$, where
we let matrices act on the space of lines, and on measures on this
space, in the natural way. Also, let $-\infty<\chi_{2}<\chi_{1}<0$
denote the Lyapunov exponents of this random matrix product, which
are negative because the matrices contract (this accounts for the
absolute values later on), and are distinct if we assume total irreducibility
and non-conformality. For a linear subspace $W\leq\mathbb{R}^{2}$,
let $\pi_{W}$ denote the orthogonal projection to $W$, and write
$\mu_{x}^{W}$ for the conditional measure of $\mu$ on $x+W$, which
is $\mu$-a.e. well defined. Write $H(p)$ for the Shannon entropy
$H(p)=-\sum_{i\in\Lambda}p_{i}\log p_{i}$. Let $\Pi:\Lambda^{\mathbb{N}}\rightarrow X$
denote the natural coding map of the attractor $X$, let $\mathcal{B}$
denote the Borel $\sigma$-algebra of $\mathbb{R}^{2}$, and let $\mathcal{P}_{1}$
denote the partition of $\Lambda^{\mathbb{N}}$ according to the first
coordinate. 
\begin{thm}
[Ledrappier-Young formula \cite{BK}] \label{thm:LY}Let $\mu$ be
a self-affine measure in $\mathbb{R}^{2}$, and, in the notation above,
assume $\chi_{2}<\chi_{1}$. Then the real number $H(p)$ splits as
a sum 
\[
H(p)=H_{1}+H_{2}+H_{3}
\]
such that
\begin{itemize}
\item $0\leq H_{1}/|\chi_{1}|\leq1$ and $\dim\pi_{W}\mu=H_{1}/|\chi_{1}|$
for $\eta^{*}$-a.e. $W$.
\item $0\leq H_{2}/|\chi_{2}|\leq1$ and $\dim\mu_{x}^{W^{\perp}}=H_{2}/|\chi_{2}|$
for $\eta^{*}$-a.e. $W$ and $\mu$-a.e. $x$.
\item $\dim\mu=H_{1}/|\chi_{1}|+H_{2}/|\chi_{2}|$.
\item $H_{3}=H_{p^{\mathbb{N}}}(\mathcal{P}_{1}|\Pi^{-1}\mathcal{B})$ (in
particular $H_{3}\geq0$).
\end{itemize}
\end{thm}

The Ledrappier-Young theorem does not by itself determine $\dim\mu$,
because the expression $\dim\mu=H_{1}/|\chi_{1}|+H_{2}/|\chi_{2}|$
for the dimension is constrained primarily by the identity $H(p)=H_{1}+H_{2}+H_{3}$,
and this leaves two degrees of freedom.\footnote{There is an explicit description of $H_{1},H_{2}$ in terms of a conditional
entropy, but computing them is no easier than computing the dimension
directly, so we did not present it here.} But the theorem also gives bounds for the $H_{i}$, placing them
in a certain compact convex set. Regarding these parameters as free
variables, we may proceed to maximize the linear expression $H_{1}/|\chi_{1}|+H_{2}/|\chi_{2}|$
on this compact domain; its maximal value is essentially the Lyapunov
dimension $\dim_{L}\mu$, and by the Ledrappier-Young formula it is
automatically an upper bound on the dimension, $\dim\mu\leq\dim_{L}\mu$.
In order to compute this maximal value, one relies on two observations:
\begin{itemize}
\item If $H_{1}<|\chi_{1}|$ and if one of the other parameters $H_{j}$
is positive, then the target function $H_{1}/|\chi_{1}|+H_{2}/|\chi_{2}|$
can be increased by increasing $H_{1}$ and decreasing $H_{j}$, while
keeping $H_{1}+H_{2}+H_{3}$ constant.\footnote{Transferring from $H_{2}$ to $H_{1}$ increases the target function
because, due to our assumption $\chi_{2}<\chi_{1}<0$, the coefficient
$1/|\chi_{1}|$ of $H_{1}$ is larger than the coefficient $1/|\chi_{2}|$
of $H_{2}$.} 
\item If $H_{2}<|\chi_{2}|$ and $H_{3}>0$ then the target function $H_{1}/|\chi_{1}|+H_{2}/|\chi_{2}|$
can be increased by increasing $H_{2}$ and decreasing $H_{3}$, while
keeping $H_{1}+H_{2}+H_{3}$ constant.
\end{itemize}
In other words, the maximum is achieved if $H_{1}$ is maximal relative
to the constraints, and $H_{2}$ is maximal given the constraints
and $H_{1}$. From this, one easily derives the formula for $\dim_{L}$
in the cases\footnote{In the third case, $H(p)>|\chi_{1}|+|\chi_{2}|$, the formula for
the Lyapunov dimension is not explained by the Ledrappier-Young formula,
but is motivated by considerations involving the affinity dimension.
In this case the Lyapunov dimension is greater than $2$, and since
we take the minimum with $2$ in Theorem \ref{thm:main}, the details
of this case do not interest us here.} $H(p)\leq|\chi_{1}|+|\chi_{2}|$,
\[
\dim_{L}\mu=\left\{ \begin{array}{ll}
\frac{H(p)}{|\chi_{1}|} & \text{if }H(p)\leq|\chi_{1}|\\
1+\frac{H(p)-|\chi_{1}|}{|\chi_{2}|} & |\chi_{1}|\leq H(p)\leq|\chi_{1}|+|\chi_{2}|\\
2\cdot\frac{H(p)}{|\chi_{1}|+|\chi_{2}|} & |\chi_{1}|+|\chi_{2}|<H(p)
\end{array}\right..
\]

In our previous work \cite{BHR}, we proved the following result under
the same assumptions as Theorem \ref{thm:main}:
\begin{thm}
[\cite{BHR}] \label{thm:BHR-projections}Under the assumptions of
Theorem \ref{thm:main} and with the notation in the Ledrappier-Young
theorem, it holds that 
\begin{equation}
\dim\pi_{W}\mu=\min\{1,\frac{H(p)}{|\chi_{1}|}\}\qquad\text{for }\eta^{*}\text{-a.e. }W\;.\label{eq:BHR-projections}
\end{equation}
\end{thm}

It should be noted that Theorem \ref{thm:BHR-projections} hinges
on computing $\dim\pi_{W}\mu$, which is the dimension of a fractal
measure on $\mathbb{R}$. In this sense, it does not confront the
non-conformality of $\Phi$ and $\mu$ directly. Nevertheless, it
implies Theorem \ref{thm:main} in two important cases:
\begin{enumerate}
\item If $H_{3}=0$, and, in particular, under the strong open set condition.\footnote{The SOSC implies $H_{3}=0$, see \cite[Corollary 2.8]{BK}}
In this case we saw that $\dim\mu=\dim_{L}\mu$ provided that $H_{1}$
takes its maximal value given the constraints, i.e. provided that
either $H_{1}=H(p)$ (if $H(p)\leq|\chi_{1}|$) or $H_{1}=|\chi_{1}|$
(if $H(p)>|\chi_{1}|$). This holds because Theorems \ref{thm:LY}
and \ref{thm:BHR-projections} together imply 
\[
\frac{H_{1}}{|\chi_{1}|}=\dim\pi_{W}\mu=\min\{1,\frac{H(p)}{|\chi_{1}|}\}\qquad\text{for }\eta^{*}\text{-a.e. }W\;.
\]
\item If $\dim\mu<1$. In this case, since projections are Lipschitz maps
and cannot increase dimension, we know that 
\[
\dim\pi_{W}\leq\dim\mu<1\qquad\text{for all }W\in\PR
\]
By Theorems \ref{thm:LY} and \ref{thm:BHR-projections} we obtain
\[
\frac{H_{1}}{|\chi_{1}|}=\dim\pi_{W}\mu=\frac{H(p)}{|\chi_{1}|}\qquad\text{for }\eta^{*}\text{-a.e. }W\;,
\]
hence $H_{1}=H(p)<|\chi_{1}|$, so $\dim\mu=\frac{H(p)}{|\chi_{1}|}=\dim_{L}\mu$.
\end{enumerate}
Thus, in order to prove Theorem \ref{thm:main}, we need to prove
$\dim\mu=\dim_{L}\mu$ for the cases not covered above, which is the
following statement:
\begin{thm}
\label{thm:main-reformulated}Under the assumptions of Theorem \ref{thm:main}
and with the notation in the Ledrappier-Young theorem, if $H_{3}>0$
and $\dim\mu\geq1$, then $\dim\mu=2$.
\end{thm}

The bulk of this paper is devoted to proving this last result, but
many of the intermediate steps are valid \textendash{} and interesting
\textendash{} under weaker assumptions than those above, and so we
prove them under the minimal assumptions necessary. The reader should
take note of the exact assumptions made on $\Phi$ in each of the
sections of the paper; these are stated at the start of each section
and in the main theorems, but, for the sake of readability, not in
all the lemmas and propositions. 

\subsection{\label{subsec:Outline-of-the-argument}Overview of the argument}

In the following paragraphs, we sketch the main ingredients of the
proof of Theorem \ref{thm:main-reformulated}, and the main auxiliary
results that go into it. We shall present it as an argument by contradiction.
Thus, for most of the following discussion, we assume that $\mu$
is a self-affine measure generated by $\Phi$, and that
\begin{itemize}
\item $\Phi$ is non-conformal, totally irreducible, and satisfies exponential
separation.
\item $H_{3}=H_{p^{\mathbb{N}}}(\mathcal{P}_{1}|\Pi^{-1}\mathcal{B})>0$.
\item $1\leq\dim\mu<2$.
\end{itemize}
The proof will depend heavily on the analysis of entropy of measures
at a variety of different scales (for the basic definitions see Section
\ref{sec:Notation-and-background}). In this introduction we are purposely
vague about how we measure entropy, but during~this exposition we
use the convention that when measuring entropy at some small scale
$2^{-m}$, we normalize the entropy by dividing by $m$, so that after
normalization the entropy is comparable to the dimension for well
behaved measures. Then non-negligible entropy means that (after dividing
by $m$) the entropy is bounded away from $0$, perhaps by a very
small constant; entropy of order one means that before normalization
the entropy was of order $m$; etc. 

Denote by $*$ the convolution operation between measures on a group,
usually $\mathbb{R}^{2}$ or the affine group; and for a measure $\theta$
on the affine group and a measure $\nu$ on $\mathbb{R}^{2}$, denote
by $\theta\conv\nu$ the push-forward of $\theta\times\nu$ by the
action map $(\varphi,x)\mapsto\varphi x$; we also sometimes write
$\theta\conv x=\theta\conv\delta_{x}$. The starting point of the
analysis is the basic convolution structure of $\mu$ as a self-affine
measure. By slight abuse of notation, write $p=\sum_{i\in\Lambda}p_{i}\cdot\delta_{\varphi_{i}}$
for the measure on the affine group corresponding to $\Phi$ (with
weights $(p_{i})$), so that 
\[
\mu=p\conv\mu=(p*p)\conv\mu\ldots=p^{*n}\conv\mu
\]
for all $n$. The overall structure of the proof is similar to other
recent results in the area:
\begin{description}
\item [{Decomposing~$p^{*n}$}] Express $p^{*n}$ as an average of measures
$\theta$ which are supported on sets of diameter $O(1)$ in the affine
group (with respect to the left-invariant metric $d$), and such that
a positive fraction of the $\theta$ have non-negligible entropy at
scale $Cn$ for some $C>0$.

This step is where $H_{3}>0$ and exponential separation are used.
\item [{Normalizing~in~the~affine~group}] For each piece $\theta$
of $p^{*n}$, fix an affine map $\varphi\in\supp\theta$ and replace
$\theta$ by its translate $\varphi^{-1}\theta$ in the affine group,
which is supported on an $O(1)$-neighborhood of the identity (by
the left-invariance of the metric).

\medskip This step is meant to deal with some of the problems arising
from the non-conformality of the maps, since $\varphi^{-1}\theta$
is now supported on maps with bounded distortion.
\item [{Entropy~growth}] Apply an entropy-growth result to the convolution
$(\varphi^{-1}\theta)\conv\mu$, and conclude that, for a positive
fraction of the pieces $\theta$ of $p^{*n}$, the entropy of $(\varphi^{-1}\theta)\conv\mu$
is substantially larger than that of $\mu$. 

\medskip We establish the entropy growth result more generally for
convolutions of the form $\theta\conv\mu$, assuming $\theta$ is
a measure near the identity of the affine group having non-negligible
entropy at a small scale. We do not require exponential separation
of $\mu$ for this result.
\item [{Returning~to~the~distorted~setting}] Re-interpreting this for
the convolution $\theta\conv\mu=\varphi((\varphi^{-1}\theta)\conv\mu)$,
we find that for a positive fraction of the pieces $\theta$ of $p^{*n}$,
the entropy of $\theta\conv\mu$, \emph{when measured in the correct
way}, is substantially larger than that of $\mu$. 

\medskip Here one must measure the entropy of $\varphi(\varphi^{-1}\theta\conv\mu)$
using partitions whose cells are adapted to $\varphi$; roughly speaking
they will be like the images under $\varphi$ of square cells. We
shall loosely call this a non-conformal partition. 
\item [{Interpolation}] We show that the entropy increase observed for
the non-conformal partitions implies an increase with respect to appropriately
chosen conformal partitions.

\medskip We do this by interpolating between the non-conformal and
conformal partitions. We must show this interpolation has a neutral
effect on the entropy. This is done with the aide of fine information
provided by the Ledrappier-Young formula and a careful analysis of
projections and slices of $\mu$. This step is the main place where
we use the assumption $\dim\mu\geq1$ (although it also simplifies
some of the other arguments). This step also uses exponential separation
and total irreducibility.
\item [{Total~entropy~change}] Observing that $p^{*n}\conv\mu$ is an
average (over the choice of the piece $\theta$) of the convolutions
of the form $\theta\conv\mu$, we show that the extra entropy from
the last step accumulates to imply that the entropy of $p^{*n}\conv\mu$
is substantially larger than that of $\mu$, which in view of the
identity $p^{*n}\conv\mu=\mu$, is the desired contradiction. 
\end{description}

\subsection{Some more details}

We now discuss some of these steps in more detail, and the new ingredients
in them. 

\subsubsection*{Analyzing the function $L$ and the orientation of cylinders }

One interesting new feature in our proof, which holds without assuming
exponential separation or $\dim\mu\geq1$, is an observation about
the orientation of cylinder measures in $\mu$. A cylinder measure
of generation $n$ is a measure of the form $\varphi_{i_{1}}\ldots\varphi_{i_{n}}\mu$,
and because the affine map $\varphi_{i_{1}}\ldots\varphi_{i_{n}}$
is highly non-conformal, the cylinder measure is supported very close
to a line whose direction $L(A_{i_{1}}\ldots A_{i_{n}})$ is the direction
of the major axis of the image of the unit ball under the matrix product
$A_{i_{1}}\ldots A_{i_{n}}$. It is a basic result in the theory of
random matrix products that this direction converges, for a $p^{\mathbb{N}}$-typical
sequence $\mathbf{i}\in\Lambda^{\mathbb{N}}$ and as $n\rightarrow\infty$,
to a direction $L(\mathbf{i})$; and the distribution $\eta$ of this
direction, as a function of the $p^{\mathbb{N}}$-random sequence
$\mathbf{i}$, is the associated Furstenberg measure. Note that we
are now multiplying the original matrices $A_{i}$ and not, as we
did earlier, their transposes, so $\eta\neq\eta^{*}$ in general;
See Section \ref{sub:furstenberg} for more details.

We are assuming that the symbolic coding $\Pi:\Lambda^{\mathbb{N}}\rightarrow X$
is far from being injective (since $H_{3}>0$), so for a typical point
$x\in X$ with respect to the measure $\mu=\Pi(p^{\mathbb{N}})$,
the function $L$ potentially can take many values on the fiber $\Pi^{-1}(x)$.
However, under our assumptions, it turns out that $L$ \emph{does
}factor through $X$:
\begin{thm}
\label{thm:L-descends-to-mu}Let $\mu$ be a self-affine measure in
$\mathbb{R}^{2}$ of dimension $<2$ generated by a system $\Phi$
that is totally irreducible and non-conformal. Then $L$ is measurable
with respect to $\Pi^{-1}\mathcal{B}$ (up to a $p^{\mathbb{N}}$-nullset).
\end{thm}

Note that this theorem does not require exponential separation or
$\dim\mu\geq1$. 

The intuition behind the proof is simple. For simplicity assume for
the moment exponential separation and $\dim\mu\geq1$. Then, if $L$
were not constant on typical $\Pi$-fibers, it would mean that there
is a set $E\subseteq X$ of positive $\mu$-measure such that for
$x\in E$, the cylinder sets which $x$ belongs to ``point'' in
substantially different directions. Now, these cylinder measures are
very nearly concentrated on a line segment and, heuristically, Theorem
\ref{thm:BHR-projections} implies that their projection to this line
has dimension $1$ (the rigorous version of this is given in Section
\ref{subsec:Entropy-of-thickened-slices}). It follows that the measure
$\mu|_{E}$ looks, at small scales, like a collection of uniform measures
on parallel line segments, but that this holds simultaneously for
two different directions. It then follows by a Fubini type argument
that the dimension of $\mu|_{E}$ should be $2$. 

This argument works also without exponential separation, and when
$\dim\mu<1$. Then we do not know that the projections of $\mu$ to
lines have dimension $1$, but using a projection theorem due to Bourgain,
and the fact that $\dim\eta^{*}>0$, one can show that there is a
$\delta>0$ such that for $\eta^{*}$-a.e. $W$ we have $\dim\pi_{W}\mu\geq\frac{1}{2}\dim\mu+\delta$,
and this is enough to carry out the argument. 

In summary, under the assumptions of theorem \ref{thm:main-reformulated},
the function $L:\Lambda^{\mathbb{N}}\rightarrow\PR$ descends to a
$\mu$-a.e. defined measurable function $L:X\rightarrow\PR$. 

For details see Section \ref{sec:the function L}. 

\subsubsection*{Decomposing $p^{*n}$ }

Under the assumptions of Theorem \ref{thm:main-reformulated}, we
wish to decompose $p^{*n}$ into ``smaller'' measures $\theta$
whose supports have diameter $O(1)$ but which still possess non-negligible
entropy. One should first note that $p^{*n}$ itself does not have
this property; it is a very spread out measure that is supported on
exponentially many atoms, describing a set of exponential diameter. 

In this paper, the measures $\theta$ are obtained by first covering
the fibers $\Pi^{-1}(x)$ of the symbolic coding map by cylinders
of a given length $n$, interpreting the name of each cylinder as
a composition of affine maps in the group, and assigning it the weight
that the cylinder has under the conditional measure of $p^{\mathbb{N}}$
on $\Pi^{-1}(x)$. The assumption that $H_{3}=H_{p^{\mathbb{N}}}(\mathcal{P}_{1}|\Pi^{-1}\mathcal{B})>0$
means that these fiber-measures have positive dimension, and so require
exponentially many cylinders to cover them. This leads to $\theta$
having positive entropy as a \emph{discrete }measure, and by exponential
separation, it also has positive entropy at scale $Cn$ for some $C\gg1$.

This construction does not give the necessary bound on the diameter
of the support of $\theta$, and, in fact, $\theta$ can still be
very spread out. The measure $\theta$ arising as above consists of
atoms at affine maps $\varphi_{i_{1}}\ldots\varphi_{i_{n}}$ which
correspond to cylinder sets containing $x$, and if the directions
$L(\varphi_{i_{1}}\ldots\varphi_{i_{n}})$ of these cylinders vary
enough, then the measure $\theta$ will be supported on a very large
set. We would like to further decompose $\theta$ into smaller measures
$\theta'$ which are supported on sets of diameter $O(1)$, but if
we needed to partition it into exponentially many such sets, then
there is the risk that the entropy of each small piece would be negligible,
and that the entropy of $\theta$ originally came from the variation
in directions. 

Luckily, the orientation of the cylinder at a point $x$ are controlled
by the value $L(x)$: the $n$-th cylinder's orientation converges
to $L(x)$ as $n\rightarrow\infty$, and there is some control of
the rates (this is a feature of standard proofs of the Oseledets ergodic
theorem, and a result of the (eventually) contractive nature of the
action of matrix products on the flag space). Using this, we can ensure
that, in order to decompose $\theta$ into pieces of support $O(1)$,
we need only a sub-exponential number of pieces, and therefore a positive
proportion of the pieces will still have substantial entropy.

For details see Section \ref{sec:decomposition-of-p*n}.

\subsubsection*{Entropy growth }

For the entropy growth part of the proof we establish another general
result which does not require the assumption of exponential separation
or $\dim\mu\geq1$. In the following statement, $\mathcal{D}_{n}$
denotes a dyadic-like partition of the affine group into cells of
diameter approximately $2^{-n}$, see Section \ref{sub:dyadic-partitions}
for details.
\begin{thm}
\label{thm:entropy growth-near-identity}Let $\mu$ be a self-affine
measure in $\mathbb{R}^{2}$ defined by a non-conformal,\footnote{In fact, the conformal case is also true, but the proof is different,
and we do not pursue this here. } totally irreducible system $\Phi$ and satisfying $\dim\mu<2$. Then
for every $\varepsilon,R>0$ there is a $\delta=\delta(\mu,\varepsilon,R)>0$
such that for every $n>N(\mu,\varepsilon,R)$, the following holds:

If $\theta$ is a probability measure on the affine group supported
within distance $R$ of the identity, then 
\[
\frac{1}{n}H(\theta,\mathcal{D}_{n})>\varepsilon\qquad\implies\qquad\frac{1}{n}H(\theta\conv\mu,\mathcal{D}_{n})>\frac{1}{n}H(\mu,\mathcal{D}_{n})+\delta\;.
\]
\end{thm}

The proof is given in Section \ref{sec:Proof-of ent growth}. It has
some features in common with results in the literature, but also requires
many new ideas. These are explained in the following summary of the
main steps.

\paragraph*{(i) Linearization }

This step is similar to previous work. In order to study the entropy
of $\theta\conv\mu$, where $\theta$ is a measure in a bounded neighborhood
of the identity in the affine group, we first decompose both $\theta$
and $\mu$ into pieces $\theta'$ and $\mu'$ respectively, so that
$\theta\conv\mu$ is the convex combination of $\theta'\conv\mu'$;
and we choose the pieces so that they are supported on sets of small
diameter. 

Next, we use the fact that on small balls (e.g. the supports of $\theta',\mu'$),
the action $(\varphi,x)\mapsto\varphi x$ is essentially linear. Thus
we can approximate the action-convolution $\theta'\conv\mu'$ by a
Euclidean convolution $(\theta'\conv x)*(\varphi\mu')$, for some
(any) choice of $x\in\supp\mu'$ and $\varphi\in\supp\theta'$. 

Gathering all the pieces together, and using the fact that entropy
is concave, we conclude that the entropy of $\theta\conv\mu$ is at
least the average entropies of $\theta'\conv\mu'$ (the average being
over the pieces), and if the pieces are small enough this is essentially
the same as the average of $(\theta'\conv x)*(\varphi\mu')$, with
$x,\varphi$ as above.

This step is explained in more detail in Section \ref{subsec:Linearization}.

\paragraph*{(ii) Applying the multidimensional inverse theorem}

The inverse theorem in $\mathbb{R}^{d}$ from \cite{HO2} says that
in order for a convolution $\tau*\nu$ of measures in $\mathbb{R}^{2}$
to have entropy that is essentially the same as that of $\nu$ alone,
it must be the case that, at most scales $\delta$, there is a linear
subspace $V=V_{\delta}\leq\mathbb{R}^{2}$ such that at $\tau$-most
points $x$ the restriction of $\tau$ to the ball $B_{\delta}(x)$
is concentrated near a translate of $V$, and for $\nu$-most points
$y$, the measure $\nu$ on $B_{\delta}(y)$ looks like a combination
of uniform measures on translates of $V$. If $\tau$ has positive
entropy then we know that $V_{\delta}$ cannot be the trivial subspace
$\{0\}$ at too many scales, and if $V_{\delta}$ had dimension $2$
at a substantial number of scales this is also to our advantage, since
this would mean that on many small balls $\nu$ looks like $2$-dimensional
Lebesgue measure. Thus, to ensure entropy growth, we want to rule
out the possibility that $\dim V_{\delta}=1$ at more than a fraction
of all scales.

Now, in our case, with $\tau=\theta'\conv x$ and $\nu=\varphi\mu'$,
we aim to show that $\varphi\mu'$ does not look like a combination
of uniform measures on line segments in direction $V_{\delta}$; but,
unfortunately, it is very likely that this is precisely what it looks
like in some direction. Indeed, $\mu'$ is a piece of $\mu$, and
$\mu$ is a combination of cylinder measures $\varphi_{i_{1}}\ldots\varphi_{i_{n}}\mu$,
which, as we already noted, look like copies of $\mu$ squeezed onto
a line segment in direction $L(\varphi_{i_{1}}\ldots\varphi_{i_{n}})\approx L(x)$;
these look like the orthogonal projection of $\mu$ to a line, and
when $\dim\mu\geq1$ it is entirely possible (even likely) that this
projection has dimension $1$. Thus the fractal structure of $\mu'$
actually supports the possibility that it's structure is ``bad''
from the point of view of applying the inverse theorem, since it looks
like uniform measure on translates of $L(x)$ (so $\varphi\mu'$ looks
like the uniform measure on lines parallel to $\varphi L(x)$).

\paragraph*{(iii) Identification of the direction $L(x)$ and using total irreducibility}

Summarizing, if there is no entropy growth in the convolution $(\theta'\conv x)*(\varphi\mu')$,
then, at scale $\delta$, on the one hand $\varphi\mu'$ is uniform
when conditioned on translates of the $1$-dimensional subspace $V_{\delta}$;
on the other hand, it is uniform when conditioned on translates of
lines in direction $\varphi L(x)$. If these subspaces are transverse,
this would lead to $\mu'$ having entropy $2$, which would eventually
lead to $\mu$ having dimension $2$, contrary to our assumptions.
So we conclude that $V_{\delta}$ must agree with $\varphi L(x)$. 

Now fix $\theta'$ and let $\mu'$ vary, so also $\varphi\in\supp\theta'$
is fixed, but $x\in\supp\mu'$ varies. Then, under the assumption
that there is no entropy growth, we have found that the measure $\theta'\conv x$
is essentially supported on a translate of an affine line in direction
$\varphi L(x)$. Equivalently, the measure $\varphi^{-1}\theta'\conv x$
is essentially supported on a translate of an affine line in direction
$L(x)$, and this holds for $\mu$-most $x$. We then show that in
this situation, $L(x)$ must be an affine function of $x$; that is,
there exists an affine function $\psi:\mathbb{R}^{2}\rightarrow\mathbb{R}^{2}$
such that $\mu$-a.e. the value $L(x)$ is the direction of the line
$\mathbb{R}\psi(x)$.

Finally, we show that if $L$ is affine in the sense above, then $\mu$
(and the attractor $X$) must be supported on a quadratic curve in
$\mathbb{R}^{2}$. This, in turn, can be shown to contradict the total
irreducibility of $\Phi$, completing the entropy growth part of the
proof.

\subsection{\label{subsec:Triangular-matrces}Triangular matrices}

Systems in which the matrices $A_{i}$ act reducibly on $\mathbb{R}^{2}$
present additional challenges, and our results for them are less complete.
An extreme instance occurs when the matrices $A_{i}$ are jointly
diagonalizable, in which case some unusual behaviors can occur, e.g.
Hausdorff and box dimensions may not agree. This situation has been
extensively studied over several decades, beginning with the work
of Bedford \cite{Bedford1989-box-dimension-of-repellers} and McMullen
\cite{McMullen1984}, and we do not discuss it \,here.

Our focus will be on the intermediate case, in which the $A_{i}$
have a single common eigendirection. Then, in some coordinate system,
the $A_{i}$ are given by triangular matrices of the same kind (upper
or lower), and we assume such coordinates have been chosen. For concreteness
we consider the lower-triangular case (the upper triangular case being
similar), and write systems $\Phi=\{\varphi_{i}\}_{i\in\Lambda}$
as
\begin{equation}
\varphi_{i}(x)=\left(\begin{array}{cc}
a_{i} & 0\\
b_{i} & c_{i}
\end{array}\right)x+v_{i}\;.\label{eq:triangular-system}
\end{equation}
As before, we assume that the maps $\{\varphi_{i}\}$ are invertible,
i.e. that $a_{i},c_{i}\ne0$ for each $i\in\Lambda$. Write $\overline{e}_{1},\overline{e}_{2}$
for the horizontal and vertical lines through the origin, respectively.
Then $\overline{e}_{2}$ is the common eigendirection of the matrices
above, and $\overline{e}_{1}$ is the common eigendirection of their
transposes. We are assuming that the matrices are not jointly diagonalizable,
so there is no other jointly invariant direction. Let us now note
some of the differences between this case and the totally irreducible
one:
\begin{itemize}
\item Without total irreducibility, we shall need additional assumptions
to ensure\footnote{If the Lyapunov exponents agree, one can apply the methods from the
self-similar case more directly. } that the Lyapunov exponents are distinct (previously this followed
from non-conformality and total irreducibility). 
\item Assuming that the Lyapunov exponents are distinct, one of the random
walks driven by $\{A_{i}\}$ or $\{A_{i}^{*}\}$ admits a unique stationary
distribution equal to $\delta_{\overline{e}_{2}}$ or $\delta_{\overline{e}_{1}}$,
respectively; and the other random walk admits two ergodic stationary
measures, one of which has positive dimension, and one again being
$\delta_{\overline{e}_{2}}$ or $\delta_{\overline{e}_{1}}$, respectively
(which of these occurs is determined by whether the expansion rate
of the $\{A_{i}\}$ acting on the invariant space $\overline{e}_{2}$
is $2^{\chi_{1}}$ or $2^{\chi_{2}}$). Either way, this breaks parts
of our argument which relied on the uniform convergence of the random
walks to their stationary distribution, or on the stationary measures
$\eta,\eta^{*}$ having positive dimension or being non-atomic. 

Crucially, when the $\{A_{i}^{*}\}$-walk is attracted to $\delta_{\overline{e}_{1}}$,
Theorem \ref{thm:BHR-projections} is not valid, and we get no good
bound on the dimension of $\eta^{*}$-typical projections; and when
$\{A_{i}^{*}\}$ is attracted to a measure of positive dimension,
but non-uniformly and not from all initial lines, then the information
we get about projections of $\mu$ is also non-uniform.
\item Due to the behavior of the random walks, the projection $\pi_{1}=\pi_{\overline{e}_{1}}$
onto $\overline{e}_{1}$ plays a distinguished role in the analysis.
Because the foliation of $\mathbb{R}^{2}$ by lines parallel to $\overline{e}_{2}$
is invariant under the $\varphi_{i}$, there is an induced system
$\overline{\Phi}=\{\overline{\varphi}_{i}\}_{i\in\Lambda}$ of affine
maps on $\mathbb{R}$, given by
\[
\overline{\varphi}_{i}(x)=a_{i}x+\pi_{1}(v_{i})\;,
\]
and satisfying
\begin{equation}
\overline{\varphi}_{i}\pi_{1}=\pi_{1}\varphi_{i}\;.\label{eq:projected-maps}
\end{equation}
The projection $\pi_{1}\mu$ is then a self-similar measure of the
system $\overline{\Phi}$. One should note, however, that exponential
separation of $\Phi$ does not imply the same for $\overline{\Phi}$,
so computing $\dim\pi_{1}\mu$ is not always possible with current
methods.
\item In contrast to the totally irreducible case, in the triangular case,
it actually is possible that $X$ and $\mu$ lie in a quadratic curve.\footnote{We remark that by work of Feng and K\"aenm\"aki \cite{FK}, quadratic
curves and, in trivial cases, lines, are the only algebraic curves
which can support a self-affine measure.} Such examples were first given by Bandt and Kravchenko \cite{BandtKravchenko2011-fractal-curves},
and in fact they show that there is a 1-parameter family of affine
maps (with triangular linear parts) preserving a given parabola. It
is then an easy matter to choose an exponentially separated sub-family
with an arbitrarily large number of maps. In this way we can obtain
a system $\Phi$ whose attractor has dimension $1$, but whose affinity
dimension (or Lyapunov dimension for e.g. the uniform choice of weights)
is larger than $2$. This shows that being ``trapped'' in a quadratic
curve is a real, rather than just hypothetical, obstruction to achieving
the Lyapunov dimension.
\end{itemize}
Due to these many issues, our arguments do not work in the triangular
case in general, and we are able to handle only one of the scenarios
above, namely, when $\eta$ has positive dimension and $\eta^{*}=\delta_{\overline{e}_{1}}$:
\begin{thm}
\label{thm:main-triangular}Let $\mu$ be a self-affine measure defined
by a system $\Phi=\{\varphi_{i}(x)=A_{i}x+v_{i}\}_{i\in\Lambda}$
as in (\ref{eq:triangular-system}), i.e. $\{A_{i}\}$ are invertible
and lower-triangular. Suppose that,
\begin{itemize}
\item $\{A_{i}\}$ are not simultaneously conjugated to a diagonal system;
\item $\Phi$ satisfies exponential separation;
\item The Lyapunov exponents are distinct: $-\infty<\chi_{2}<\chi_{1}<0$,
and $\overline{e}_{2}$ is contracted at rate $2^{\chi_{2}}$ (for
example, this holds if $|c_{i}|<|a_{i}|$ for all $i\in\Lambda$);
\item $\mu$ is not supported on a quadratic curve;
\item The projection $\pi_{1}\mu$ has the maximal possible dimension, i.e.
\begin{equation}
\dim\pi_{1}\mu=\min\{1,\dim\mu\}\;.\label{eq:pi-1-lower-bound}
\end{equation}
\end{itemize}
Then
\[
\dim\mu=\min\{2,\dim_{L}\mu\}\;.
\]
\end{thm}

\begin{rem}
The case covered by Theorem \ref{thm:main-triangular} is complementary
to the one analyzed in \cite[Proposition 6.6]{BHR}. Because Theorem
\ref{thm:BHR-projections} cannot be applied, we have been forced
to add an explicit assumption about $\dim\pi_{1}\mu$ (where $\pi_{1}$
is in fact the projection to a $\eta^{*}$-typical line). The case
which the theorem above does not cover is when $\chi_{2}<\chi_{1}<0$
but $\overline{e}_{2}$ is contracted at rate $2^{\chi_{1}}$; then
Theorem \ref{thm:BHR-projections} does hold, but we are unable to
carry out the rest of the argument, and are still not able to go beyond
the case when $H_{3}=0$, which already follows from \cite{BHR}.
\end{rem}

The situation in the theorem here is reminiscent of that of self-similar
measures in the plane generated by homotheties, and carpet fractals.
In all these cases one gets information about $\mu$ (or $X$) only
if one can show that certain projections are large (or that the corresponding
slices are small). This is unsatisfactory, but examples show that
it reflects the true state of affairs for self-similar and carpet
measures, and it is likely that the same is true in our setting. 

There are currently two main ways to try to verify hypothesis (\ref{eq:pi-1-lower-bound}).
First, if the induced system $\overline{\Phi}$ satisfies exponential
separation, then we will have $\dim\pi_{1}\mu=\min\{1,\dim_{L}\pi_{1}\mu\}$,
in which case (\ref{eq:pi-1-lower-bound}) clearly holds. Second,
by the Ledrappier-Young formula, a ``dimension conservation'' phenomenon
holds: 
\begin{equation}
\dim\mu=\dim\pi_{1}\mu+\dim\mu_{x}^{\overline{e}_{2}}\qquad\text{for }\mu\text{-a.e. }x\;,\label{eq:dim-conservation}
\end{equation}
where $\mu_{x}^{\overline{e}_{2}}$ denotes the conditional measure
on $\overline{e}_{2}+x$. If we can show that all vertical slices
$X\cap(x+\overline{e}_{2})$ of the attractor $X$ satisfy $\dim(X\cap(x+\overline{e}_{2}))\le\max\{\dim\mu-1,0\}$,
we would get similar bounds for $\dim\mu_{x}^{\overline{e}_{2}}$,
and (\ref{eq:pi-1-lower-bound}) follows from (\ref{eq:dim-conservation}).

\subsection{\label{subsec:Higher-dimensions}Higher dimensions}

The study of the overlapping case for planar self-affine measures
is motivated not only by its general interest, but because it is closely
related to the higher-dimensional setting. In this section we very
briefly explain this connection. 

One can see the connection already in our work on separated self-affine
measures in the plane \cite{BHR}. There the key ingredient of the
analysis was the computation of the dimension of projections, which
are complicated for two reasons: first, they are not self-affine,
but nevertheless they do have some convolutions structure, which helps
in the analysis; but, second, although $\mu$ was separated, its projections
to lines are generally not separated. This makes it necessary to analyze
overlapping fractals in the line in order to study separated planar
ones.

A similar situation holds in higher dimensions. As a demonstration,
suppose that one wants to study the \emph{separated} case of self-affine
measures in $\mathbb{R}^{3}$. Let $\mu=\sum p_{i}\cdot\varphi_{i}\mu$
be such a measure. Assume that there are distinct Lyapunov exponents
$\chi_{3}<\chi_{2}<\chi_{1}<0$, meaning that the normalized logarithms
of the singular values of the random products $A_{i_{n}}\ldots A_{i_{1}}$
converge to these constants a.s. The Furstenberg measure $\eta^{*}$
is also a more complicated object: it is a measure on pairs $(V,W)$
where $V\leq\mathbb{R}^{3}$ is a line and $W\leq\mathbb{R}^{3}$
is a $2$-dimensional subspace containing $V$ (this is the so-called
flag space). The projections $\eta_{1}^{*},\eta_{2}^{*}$ to the first
and second components now describe the asymptotic distribution of
the random walks $A_{i_{n}}^{*}\ldots A_{i_{1}}^{*}V$ on lines and
$A_{i_{n}}^{*}\ldots A_{i_{1}}^{*}W$ on planes. 

The Ledrappier-Young formula in this case says that the entropy $H(p)$
decomposes as a non-negative sum\footnote{If we did not assume separation, there would be a fourth term $H_{4}=H(\xi,\mathcal{P}_{1}|\Pi^{-1}\mathcal{B})$.}
$H(p)=H_{1}+H_{2}+H_{3}$, where 
\begin{itemize}
\item $\dim\pi_{V}\mu=H_{1}/|\chi_{1}|$ for $\eta_{1}^{*}$-a.e. line $V$;
\item $\dim\pi_{W}\mu=H_{1}/|\chi_{1}|+H_{2}/|\chi_{2}|$ for $\eta_{2}^{*}$-a.e.
plane $W$;
\item $\dim\mu=H_{1}/|\chi_{1}|+H_{2}/|\chi_{2}|+H_{3}/|\chi_{3}|$.
\end{itemize}
Now, our results from \cite{BHR} can be adapted to show that $H_{1}$
must be maximal, i.e. that $\dim\pi_{V}\mu=\min\{1,H(p)/|\chi_{1}|\}$
for $\eta_{1}^{*}$-a.e. $V$. However, that still leaves one degree
of freedom to determine $H_{2},H_{3}$. To prove that the dimension
is maximal subject to the constraints, it is then necessary to show
that $\pi_{W}\mu$ is maximal. 

Now, $\pi_{W}\mu$ is a measure in a plane $W$ and is not, strictly
speaking, self-affine, but it shares some of that structure of a self-affine
measure, in the sense that it can be written as 
\[
\pi_{W}\mu=\sum p_{i}\cdot\pi_{W}\varphi_{i}\mu
\]
(note that the right hand side does not consist of affine images of
the left hand side, but when this identity is iterated the distribution
of the measures on the right hand side becomes consistent across scales). 

Therefore, one may hope to analyze $\pi_{W}\mu$ using the methods
we have developed for self-affine measures in the plane. However,
although $\mu$ is a separated self-affine measure in $\mathbb{R}^{3}$,
its projections $\pi_{W}\mu$ on a plane $W$ in general is not separated.
Nevertheless it is likely to be exponentially separated for $\eta_{2}^{*}$-typical
choices of $W$. One therefore hopes that the methods from this paper
can be applied there. 

We anticipate that in this way one can, by a suitable induction on
the dimension of the ambient space, compute the dimension of exponentially
separated self-affine measures in general, at least under the assumption
of total irreducibility and, possibly, simple Lyapunov spectrum. We
hope to return to this in a future paper.

\subsection{\label{subsec:Organization-of-the}Organization of the paper}

\noindent In the next section (Section \ref{sec:Notation-and-background})
we develop notation and background, such as basic results on entropy,
the Oseledets theorem, Furstenberg measure and related material. Section
\ref{sec:Entropy} establishes many technical results about the entropy
of projections and slices of $\mu$ as well as those of the cylinder
measures of $\mu$ and its components (restrictions to dyadic cells).
In Section \ref{sec:the function L} we study the function $L$ describing
the orientation of cylinders and show that it is well-defined $\mu$-a.e.
(Theorem \ref{thm:L-descends-to-mu}). In Section \ref{sec:Concentration-near-lines}
we give some algebraic results showing among other things that $L$
is not affine. Section \ref{sec:Proof-of ent growth} establishes
the entropy growth theorem (Theorem \ref{thm:entropy growth-near-identity}).
Section \ref{sec:non-conformal-partitions} analyzes the entropy of
non-conformal partitions\textbf{. }In Section \ref{sec:decomposition-of-p*n}
we construct the decomposition of $p^{*n}$ into high-entropy measures
supported on sets of diameter $O(1)$. Finally, Section \ref{sec:Proof-of-main-result}
contains the proof of the main theorem, Theorem \ref{thm:main}. 

We include a summary of our main notation:

\begin{longtable}{|l|l|}
\hline 
$A_{k,m}$ & Space of maximal-rank affine maps $\mathbb{R}^{k}\rightarrow\mathbb{R}^{m}$.\tabularnewline
\hline 
$A_{k,m}^{vec}$ & Vector space of all affine maps $\mathbb{R}^{k}\rightarrow\mathbb{R}^{m}$.\tabularnewline
\hline 
$A_{\varphi},b_{\varphi}$ & For $\varphi\in A_{2,2}$ with $\varphi(x)=A_{\varphi}x+b_{\varphi}$.\tabularnewline
\hline 
$\pi_{W}$ & Orthogonal projection onto $W$.\tabularnewline
\hline 
$T_{c}$, $S_{a}$ & Scaling $x\rightarrow cx$ and translation $x\rightarrow x+a$.\tabularnewline
\hline 
$\Phi=\{\varphi_{i}\}_{i\in\Lambda}$ & Affine invertible contractions of $\mathbb{R}^{2}$, no common fixed
point.\tabularnewline
\hline 
$p=(p_{i})_{i\in\Lambda}$ & positive prob. vector; identify with $\sum p_{i}\cdot\delta_{\varphi_{i}}\in\mathcal{P}(A_{2,2})$.\tabularnewline
\hline 
$X$ & Self-affine set.\tabularnewline
\hline 
$\mu$ & Self-affine measure, $\mu=\sum_{i\in\Lambda}p_{i}\varphi_{i}\mu$.\tabularnewline
\hline 
$\mu_{x}^{W}$ & Conditional measure on the line $x+W$.\tabularnewline
\hline 
$\alpha,\beta,\gamma$ & Dimension of $\mu$, its projections and slices: Section \ref{sub:self-affine-measures}.\tabularnewline
\hline 
$\chi_{2}<\chi_{1}\leq0$ & Lyapunov exponents, Section \ref{sub:furstenberg}.\tabularnewline
\hline 
$\eta,\eta^{*}$ & Furstenberg measure of products of $A_{i}$ and $A_{i}^{*},$ resp.\tabularnewline
\hline 
$\varphi_{i_{1}\ldots i_{n}},A_{i_{1}\ldots i_{n}}$ & Composition $\varphi_{i_{1}}\circ\ldots\circ\varphi_{i_{n}}$ etc.\tabularnewline
\hline 
$[a]$$\subseteq\Lambda^{\mathbb{N}}$ & Cylinder set corresponding to $a\in\Lambda^{n}$.\tabularnewline
\hline 
$S$ & Shift map on $\Lambda^{\mathbb{N}}$.\tabularnewline
\hline 
$\Pi$ & Coding map $\Lambda^{\mathbb{N}}\rightarrow X$.\tabularnewline
\hline 
$\xi=p^{\mathbb{N}}$ & Product measure on $\Lambda^{\mathbb{N}}$.\tabularnewline
\hline 
$\xi_{\omega}$ & Conditional measure on $\Pi^{-1}(\Pi(\omega))$.\tabularnewline
\hline 
$\mu_{x}^{V}$ & Conditional measure on $x+V$ for line $V\leq\mathbb{R}^{2}$.\tabularnewline
\hline 
$\PR$ & Projective space (space of lines in $\mathbb{R}^{2}$)\tabularnewline
\hline 
$\overline{x}\in\PR$ & element of $\PR$ (sometimes associated to $x\in\mathbb{R}^{2}\setminus\{0\}$)\tabularnewline
\hline 
$\alpha_{1}(A)\geq\alpha_{2}(A)$ & Singular values of a matrix $A$.\tabularnewline
\hline 
$L(A),L(\omega)\in\PR$ & Major axis/asymptotic version (Sections \ref{subsec:The-function-L},
\ref{sub:furstenberg}, \ref{sec:the function L}).\tabularnewline
\hline 
$\mathcal{D}_{n}$ & Partition into level-$n$ dyadic cells or equivalent (Section \ref{sub:dyadic-partitions}).\tabularnewline
\hline 
$\mathcal{D}_{n}^{W\oplus W^{\perp}}$ & Dyadic partition in coordinates $W\oplus W^{\perp}$.\tabularnewline
\hline 
$\nu_{x,n},\nu^{x,n}$  & Dyadic components, Section \ref{q-adic-components}.\tabularnewline
\hline 
$\Psi_{n},\Upsilon_{n}\subseteq\Lambda^{*}$  & See Section \ref{sub:random-cylinder-measures}.\tabularnewline
\hline 
$\mathbf{I}(n)$,$\mathbf{K}(n)$ & See Section \ref{sub:random-cylinder-measures}.\tabularnewline
\hline 
$d$ & Left-invariant metric on $A_{2,2}$.\tabularnewline
\hline 
$d_{\PR}$ & Metric on $\PR$: $d_{\PR}(V,W)=\left\Vert \pi_{V}-\pi_{W}\right\Vert $.\tabularnewline
\hline 
$d_{TV}$ & Total variation metric on measures.\tabularnewline
\hline 
$H(\nu,\mathcal{C}),H(\nu,\mathcal{C}|\mathcal{E})$ & Entropy (resp. conditional).\tabularnewline
\hline 
$\nu_{1}*\nu_{2}$ & Convolution in $\mathbb{R}^{2}$ or $A_{2,2}$.\tabularnewline
\hline 
$\theta\conv\nu$ & Convolution of $\theta\in\mathcal{P}(A_{2,2})$ and $\nu\in\mathcal{P}(\mathbb{R}^{2})$.\tabularnewline
\hline 
\end{longtable}

\subsection*{Acknowledgment}

The authors would like to thank Bal\'azs B\'ar\'any for many useful
discussions, and Emmanuel Breuillard for his comments on an early
version of the results. M.H. supported by ERC grant 306494 and ISF
grant 1702/17, and National Science Foundation Grant No. DMS-1638352.
A.R. supported by ERC grant 306494 and the Herchel Smith Fund at the
University of Cambridge.

\section{\label{sec:Notation-and-background}Preparations}

\subsection{\label{subsec:Measures-on-integer-intervals}\label{subsec:Basic-notation}Conventions}

We equip $\mathbb{R}^{d}$ with the Euclidean norm. Spaces of matrices
and linear maps are given the operator norm. In a metric space, $B_{r}(x)$
is the closed ball of radius $r$ around $x$, and $E^{(r)}$ is the
open $r$-neighborhood of $E$, that is, all point of distance $<r$
from $E$. We write $\mathcal{P}(X)$ for the space of Borel probability
measures on $X$. All measures are Borel measures unless otherwise
stated and all functions are assumed measurable even if not mentioned
explicitly. Convergence of measures in $\mathcal{P}(X)$ is by default
understood to be weak convergence, although we will sometimes also
consider the total variation metric on $\mathcal{P}(X)$, which we
denote $d_{TV}.$ We use standard big-$O$ and little-$o$ notation.

\subsection{\label{sub:self-affine-measures}Self-affine sets and measures}

Throughout the paper, $\Phi=\{\varphi_{i}(x)=A_{i}x+a_{i}\}_{i\in\Lambda}$
is a system of invertible affine contractions of $\mathbb{R}^{2}$
without a common fixed point, and $X\neq\emptyset$ is the associated
compact attractor, defined uniquely by the relation 
\[
X=\bigcup_{i\in\Lambda}\varphi_{i}(X)\;.
\]
We also fix a strictly positive probability vector $p=(p_{i})_{i\in\Lambda}$,
and let $\mu$ denote the associated self-affine measure, defined
uniquely by the relation 
\[
\mu=\sum_{i\in\Lambda}p_{i}\cdot\varphi_{i}\mu\;.
\]
We write $\Lambda^{*}$ for the set of all finite words over $\Lambda$.
For a word $\mathbf{i}=i_{1}\ldots i_{n}\in\Lambda^{*}$, let 
\[
\varphi_{\mathbf{i}}=\varphi_{i_{1}}\ldots\varphi_{i_{n}}\;,
\]
and similarly write $A_{\mathbf{i}}=A_{i_{1}}\cdots A_{i_{n}},$ $p_{\mathbf{i}}=p_{i_{1}}\ldots p_{i_{n}}$,
etc.

We define the coding map, $\Pi:\Lambda^{\mathbb{N}}\to X$, by 
\[
\Pi(\mathbf{i})=\lim_{n\to\infty}\varphi_{i_{1}\ldots,i_{n}}(0)\;,
\]
where the limit exists by contraction. Then $X=\im\Pi$. We write
\[
\xi=p^{\mathbb{N}}
\]
for the product measure on $\Lambda^{\mathbb{N}}$ with marginal $p$,
so that 
\[
\mu=\Pi\xi\;.
\]

For $\mathbf{i}\in\Lambda^{n}$ we refer to the measure $\varphi_{\mathbf{i}}\mu$
as a (generation $n$) cylinder measure. We also define the generation-$n$
cylinder set $[\mathbf{i}]\subseteq\Lambda^{\mathbb{N}}$ by
\[
[\mathbf{i}]=\{\mathbf{j}\in\Lambda^{\mathbb{N}}\,:\,j_{1}\ldots j_{n}=i_{1}\ldots i_{n}\}\;,
\]
which is closed and open in the product topology. The corresponding
generation-$n$ cylinder measure of $\xi$ is defined by $\xi_{[\mathbf{i}]}=\xi([\mathbf{i}])^{-1}\cdot\xi|_{[\mathbf{i}]}$,
and we have
\[
\varphi_{\mathbf{i}}\mu=\Pi\xi_{[\mathbf{i}]}\;,
\]
so that the generation-$n$ cylinder measures of $\mu$ are the images
under $\Pi$ of generation-$n$ cylinder measures of $\xi$.

Throughout the paper, we write
\[
\alpha=\dim\mu\;,
\]
and, when assuming non-conformality and total irreducibility, we let
$\beta$ denote the $\eta^{*}$-almost-sure value of orthogonal projections,
\[
\beta=\dim\pi_{W}\mu\qquad\text{for }\eta^{*}\text{-a.e. }W
\]
(which exists by Theorem \ref{thm:LY}, for $\eta^{*}$ see that theorem
or Section \ref{sub:furstenberg} below). Note that if exponential
separation is assumed, then $\beta=\min\{1,H(p)/|\chi_{1}|\}$ by
Theorem \ref{thm:BHR-projections}. Also set
\[
\gamma=\alpha-\beta\;.
\]
It is another consequence of the Ledrappier-Young theory that $\gamma$
is the a.s. dimension of the conditional measures of $\mu$ on translates
of lines perpendicular to $\eta^{*}$-typical directions. For details
see Theorem \ref{thm:LY} above.

\subsection{\label{sub:linear-maps}Affine maps, projections, dilations, translation}

We write $A_{k,m}$ for the space of maximal-rank affine maps $\mathbb{R}^{k}\rightarrow\mathbb{R}^{m}$,
and $A_{k,m}^{vec}$ for the vector space of all affine maps $\mathbb{R}^{k}\rightarrow\mathbb{R}^{m}$,
so that $A_{2,2}\subseteq A_{2,2}^{vec}$. 

We endow $A_{2,2}$ with a left-invariant metric $d$, derived from
a Riemannian metric, and endow $A_{2,2}^{vec}$ with a norm. These
induce the same topology on $A_{2,2},$ but the metrics are not bi-Lipschitz
equivalent.

An affine map $\varphi$ can be written as $\varphi(x)=Ax+b$ for
a matrix $A$ and vector $b$. In general, we denote $A,b$ by $A_{\varphi},b_{\varphi}$,
respectively.

For a subspace $W\leq\mathbb{R}^{2}$, we write $\pi_{W}:\mathbb{R}^{2}\rightarrow W$
for the orthogonal projection onto $W$. We often identify a projection
$\pi_{W}$ with the affine map $\mathbb{R}^{2}\rightarrow\mathbb{R}$
of norm $1$, obtained by endowing $W$ with a unit vector and corresponding
coordinate system. Conversely, a functional $\pi$ of norm $1$ corresponds
to an orthogonal projection to $(\ker\pi)^{\perp}$. With this identification,
for any line $W$ and affine map $\varphi:\mathbb{R}^{2}\to\mathbb{R}^{2}$
with $\varphi(x)=Ax+b$, it is easy to check that 
\begin{equation}
\pi_{W}\circ\varphi(x)=(\pm1)\|\pi_{W}\circ A\|\cdot\pi_{A^{*}W}(x)+\pi_{W}(b)\;,\label{eq:affine-map-to-projection}
\end{equation}
where the sign depends on the orientation we used to identify $W$
and $A^{*}W$ with $\mathbb{R}$.

The operations of dilation and translation in $\mathbb{R}^{k}$ we
denote by $S_{c}$ and $T_{a}$ respectively, i.e., for $c\in\mathbb{R}$
we write $S_{c}(x)=c\cdot x$, and for $a\in\mathbb{R}^{k}$ we write
$T_{a}(x)=x+a$.

\subsection{\label{subsec:The-function-L}Projective space, singular values and
the function $L$}

We write $\PR$ for the $1$-dimensional projective space, i.e. the
space of lines in $\mathbb{R}^{2}$. We define the metric $d_{\PR}(\cdot,\cdot)$
on $\PR$ by
\[
d_{\PR}(V,W)=\left\Vert \pi_{V}-\pi_{W}\right\Vert _{op},
\]
where $\left\Vert \cdot\right\Vert _{op}$ is the operator norm. We
note that there is a constant $c>1$ such that 
\begin{equation}
|\sin\varangle(V,W)|\leq d_{\PR}(V,W)\leq c|\sin\varangle(V,W)|\;.\label{eq:compar-dPR-and-sine-of-angle}
\end{equation}

For $v\in\mathbb{R}^{2}\setminus\{0\}$ we write $\overline{v}=\mathbb{R}v\in\PR$,
and also denote elements of $\PR$ generically by $\overline{x}$,
even when no representative $x$ was chosen. We continue to also denote
linear subspaces of $\mathbb{R}^{2}$ by $V,W$ etc. 

Given $A\in Gl_{2}(\mathbb{R})$, let $\alpha_{1}(A)\geq\alpha_{2}(A)$
denote its singular values, i.e. if $A=VDU$ is a singular value decomposition,
then $D=\diag(\alpha_{1}(A),\alpha_{2}(A))$. These are also characterized
by $\alpha_{1}(A)=\left\Vert A\right\Vert $ and $\alpha_{2}(A)=\left\Vert A^{-1}\right\Vert ^{-1}$,
and represent the length of the major and minor axes of the ellipse
which is the image $A(B_{1}(0))$ of the unit ball.

Let $e_{1},e_{2}$ denote the standard basis vectors in $\mathbb{R}^{2}$.
Assuming $\alpha_{1}(A)>\alpha_{2}(A)$, write 
\[
L(A)=\overline{Ve_{1}}\in\PR
\]
for the direction of $Ve_{1}$ ($L(A)$ is not defined if $\alpha_{1}(A)=\alpha_{2}(A)$). 

For $\mathbf{i}\in\Lambda^{n}$ and $\varphi_{\mathbf{i}}=\varphi_{i_{1}}\ldots\varphi_{i_{n}}$
we call $L(A_{\mathbf{i}})$ the direction of $\varphi_{\mathbf{i}}$
and of the cylinder $\varphi_{\mathbf{i}}\mu$. We also say that $\alpha_{1}(A_{\mathbf{i}})$
is the diameter, or length, of the cylinder $\varphi_{\mathbf{i}}\mu$
and that $\alpha_{2}(A_{\mathbf{i}})$ is its width.
\begin{lem}
\label{lem:norm-of-projection-composed-with-matrix}Let $W\in\PR$
and $A\in Gl_{2}(\mathbb{R})$, and suppose that $L(A)$ is well defined.
Then, 
\[
\left\Vert A\right\Vert \cdot|\sin\varangle(L(A),W^{\perp})|\leq\left\Vert \pi_{W}\circ A\right\Vert \leq\left\Vert A\right\Vert \;,
\]
and in particular, for $c$ as in (\ref{eq:compar-dPR-and-sine-of-angle}),
\[
c^{-1}\cdot\left\Vert A\right\Vert \cdot d_{\PR}(L(A),W^{\perp})\leq\left\Vert \pi_{W}\circ A\right\Vert \leq\left\Vert A\right\Vert \;.
\]
\end{lem}

\begin{proof}
The inequality on the right follows from $\left\Vert \pi_{W}A\right\Vert \leq\left\Vert \pi_{W}\right\Vert \left\Vert A\right\Vert =\left\Vert A\right\Vert $
and the one on the left by considering a unit vector $v$ such that
$\left\Vert Av\right\Vert =\left\Vert A\right\Vert $, and noting
that $Av$ points in direction $L(A)$, so $\left\Vert \pi_{W}Av\right\Vert =\left\Vert Av\right\Vert \cdot|\sin\varangle(L(A),W^{\perp})|$. 
\end{proof}

\subsection{Dyadic partitions\label{sub:q-adic-partitions}\label{sub:dyadic-partitions}}

We work extensively with the dyadic partitions of $\mathbb{R}$ and
$\mathbb{R}^{2}$. The level-$n$ partition of $\mathbb{R}$ is defined
by 
\[
\mathcal{D}_{n}=\left\{ [\frac{k}{2^{n}},\frac{k+1}{2^{n}})\,:\,k\in\mathbb{Z}\right\} \:.
\]
We write $\mathcal{D}_{t}=\mathcal{D}_{[t]}$ when $t\in\mathbb{R}$
is non-integer. In $\mathbb{R}^{d}$ we write, 
\[
\mathcal{D}_{n}^{d}=\{I_{1}\times\ldots\times I_{d}\,:\,I_{i}\in\mathcal{D}_{n}\}\;,
\]
and generally omit the superscript. For $W\in\PR$ and $m\ge0$ write
\[
\mathcal{D}_{m}^{W\oplus W^{\perp}}=(\pi_{W}^{-1}\mathcal{D}_{m})\vee(\pi_{W^{\perp}}^{-1}\mathcal{D}_{m})\:.
\]
This is just a dyadic partition in the coordinate system determined
by $W,W^{\perp}.$

Two partitions are $C$-commensurable if each element of one intersects
at most $C$ elements of the other. If $\varphi$ is an isometry of
$\mathbb{R}$ or $\mathbb{R}^{d}$ then $\mathcal{D}_{n}$ and $\varphi\mathcal{D}_{n}$
are $O_{d}(1)$-commensurable, and also $\mathcal{D}_{n}^{W\oplus W^{\perp}}$
and $\mathcal{D}_{n}$ are $O(1)$-commensurable.

We will need a similar system of partitions of $A_{2,2}$. By \cite[Remark 2.2]{KaenmakiRajalaSuomala2012},
there exists a collection of Borel sets 
\[
\{Q_{n,i}\subset A_{2,2}\::\:n\in\mathbb{Z},\:i\in\mathbb{N}\}\;,
\]
having the following properties: 
\begin{enumerate}
\item $A_{2,2}=\cup_{i\in\mathbb{N}}Q_{n,i}$ for every $n\in\mathbb{Z}$;
\item $Q_{n,i}\cap Q_{m,j}=\emptyset$ or $Q_{n,i}\subset Q_{m,j}$ whenever
$n,m\in\mathbb{Z}$, $n\ge m$, $i,j\in\mathbb{N}$;
\item \label{item:ball} There exists a constant $C>1$ such that for every
$n\in\mathbb{Z}$ and $i\in\mathbb{N}$ there exists $\psi\in Q_{n,i}$
with 
\[
B(\psi,C^{-1}2^{-n})\subset Q_{n,i}\subset B(\psi,C2^{-n}),
\]
where the balls are taken with respect to the left-invariant metric
$d$.
\end{enumerate}
For each $n\in\mathbb{Z}$, denote by $\mathcal{D}_{n}^{A_{2,2}}$
the partition $\{Q_{n,i}\::\:i\in\mathbb{N}\}$ of $A_{2,2}$. These
partitions behave\footnote{One difference between $\mathcal{D}_{n}^{A_{2,2}}$ and dyadic partitions
in $\mathbb{R}^{d}$ is that there is no guarantee that a decreasing
sequence of cells $E_{1}\supseteq E_{2}\supseteq\ldots$ with $E_{n}\in\mathcal{D}_{n}^{A_{2,2}}$
must be strictly decreasing. For some $n$ it might be that $E_{n+1}=E_{n}$.
But property (\ref{item:ball}) ensures that this only can happen
for at most boundedly many consecutive values of $n$. In any case,
this will never be an issue. } much like the dyadic partitions of $\mathbb{R}^{d}$ and we usually
denote them simply by $\mathcal{D}_{n}$ (whether we mean the partition
of $\mathbb{R}^{d}$ or $A_{2,2}$ will be clear from the context). 
\begin{lem}
\label{lem:Qn-has-bounded-degree} There exists a constant $C'\ge1$
such that for every $n\geq0$ and $Q\in\mathcal{D}_{n}^{A_{2,2}}$,
\[
\#\{Q'\in\mathcal{D}_{n+1}^{A_{2,2}}\::\:Q'\subset Q\}\leq C'\:.
\]
\end{lem}

We omit the proof. For a similar statement with proof see \cite[Lemma 2.4]{BHR}.

\subsection{\label{q-adic-components}Component measures}

For a partition $\mathcal{Q}$ (in $\mathbb{R}^{d}$ or in $A_{2,2}$
respectively) we write $\mathcal{Q}(x)$ for the unique partition
element containing $x$. For a probability measure $\theta$, write
\[
\theta_{A}=\frac{1}{\theta(A)}\theta|_{A}
\]
for the conditional measure of $\theta$ on $A$, assuming $\theta(A)>0$.

For a probability measure $\theta$ on a space equipped with refining
partitions $\mathcal{Q}_{1},\mathcal{Q}_{2},\ldots$, we define measure
valued random variables $\theta_{x,n}$ such that $\theta_{x,n}=\theta_{\mathcal{Q}_{n}(x)}$
with probability $\theta(\mathcal{Q}_{n}(x))$. We call $\theta_{x,n}$
an $n$-th level component of $\theta$. When several components appear
together, e.g. $\theta_{x,n}$ and $\tau_{y,n}$, we assume $x,y$
are chosen independently unless stated otherwise. Sometimes $n$ is
chosen randomly as well, usually uniformly in some range. For example,
we write for $n_{2}\geq n_{1}$ integers and an event $\mathcal{U}$,
\begin{equation}
\mathbb{P}_{n_{1}\leq i\leq n_{2}}\left(\mu_{x,i}\in\mathcal{U}\right)=\frac{1}{n_{2}-n_{1}+1}\sum_{n=n_{1}}^{n_{2}}\mathbb{P}(\mu_{x,n}\in\mathcal{U})\;.\label{eq:p}
\end{equation}
We write $\mathbb{E}$ and $\mathbb{E}_{n_{1}\leq i\leq n_{2}}$ for
the expected value with respect to the probabilities $\mathbb{P}$
and $\mathbb{P}_{n_{1}\leq i\leq n_{2}}$. 

We also introduce notation for randomely chosen integers in interval
ranges: Given integers $n\ge m\ge1$ let $\mathcal{N}_{m,n}=\{m,m+1,..,n\}$
and denote the normalized counting measure on $\mathcal{N}_{m,n}$
by $\lambda_{m,n}$, i.e. $\lambda_{m,n}\{i\}=\frac{1}{n-m+1}$ for
each $m\le i\le n$. Write $\mathcal{N}_{n}$ and $\lambda_{n}$ in
place of $\mathcal{N}_{1,n}$ and $\lambda_{1,n}$.

In Euclidean space we also introduce re-scaled components: For $\theta\in\mathcal{P}(\mathbb{R}^{d})$,
denote by $\theta^{x,n}$ the push-forward of $\theta_{x,n}$ by the
unique homothety which maps $\mathcal{D}_{n}(x)$ onto $[0,1)^{2}$.
We view these as random variables using the same conventions as above. 

Component distributions have the convenient property that they are
almost invariant under repeated sampling, i.e. choosing components
of components. More precisely, for $\nu\in\mathcal{P}(\mathbb{R}^{d})$
and $m,n\in\mathbb{N}$, let $\mathbb{P}_{n}^{\nu}$ denote the distribution
of components $\nu^{x,i}$, $0\leq i\leq n$, as defined above; and
let $\mathbb{Q}_{n,m}^{\nu}$ denote the distribution on components
obtained by first choosing a random component $\nu_{x,i}$, $0\leq1\leq n$,
and then, conditionally on $\theta=\nu_{x,i}$, choosing a component
$\theta^{y,j}$, $i\leq j\leq i+m$ with the usual distribution (note
that $\theta^{y,j}=\nu^{y,j}$ is indeed a component of $\nu$). 
\begin{lem}
\label{lem:distribution-of-components-of-components}Given $\nu\in\mathcal{P}(\mathbb{R}^{d})$
and $m,n\in\mathbb{N}$, the total variation distance between $\mathbb{P}_{n}^{\nu}$
and $\mathbb{Q}_{n,m}^{\nu}$ satisfies 
\[
d_{TV}(\mathbb{P}_{n}^{\nu},\mathbb{Q}_{n,m}^{\nu})=O(\frac{m}{n})\;.
\]
\end{lem}

For the proof, see \cite[Lemma 2.7]{HO2}.

\subsection{\label{sub:random-cylinder-measures}Random cylinder measures with
prescribed geometry}

The symbolic space $\Lambda^{\mathbb{N}}$ comes with the natural
partitions $\mathcal{P}_{n}$ into level-$n$ cylinder sets. It will
be convenient to consider more general partitions into cylinders of
varying length. Thus, if $\Xi\subseteq\Lambda^{*}$ is a collection
of words such that the cylinder sets corresponding to the words in
$\Xi$ form a partition of $\Lambda^{\mathbb{N}}$, then we say that
$\Xi$ is a partition. In this case we also let $\Xi$ denote the
associated ``name'' function $\Xi\colon\Lambda^{\mathbb{N}}\mapsto\Lambda^{*}$,
so $\Xi(\mathbf{i})$ is the unique word in $\Xi$ such that $\mathbf{i}\in[\Xi(\mathbf{i})]$.

We return to our self-affine measure $\mu$, recalling the notation
from Sections \ref{sub:self-affine-measures} and \ref{sub:linear-maps}.
We first note that by iterating the basic identity $\mu=\sum_{i\in\Lambda}p_{i}\cdot\varphi_{i}\mu$,
for any partition $\Xi\subseteq\Lambda^{*}$, we get 
\begin{equation}
\mu=\sum_{\mathbf{i}\in\Xi}p_{\mathbf{i}}\varphi_{\mathbf{i}}\mu\;,\label{eq:iterated-mu}
\end{equation}
and if $V\in\PR$ then by applying $\pi_{V}$ to the above, we get
\begin{equation}
\pi_{V}\mu=\sum_{\mathbf{i}\in\Xi}p_{\mathbf{i}}\cdot\pi_{V}\varphi_{\mathbf{i}}\mu\;.\label{eq:iterated-convolution-for-mu}
\end{equation}
In these identities, if $\Xi=\Lambda^{n}$ for large $n$ then the
measures $\varphi_{\mathbf{i}}\mu$ and $\pi_{V}\varphi_{\mathbf{i}}\mu$
exhibit substantial variation in geometry as $\mathbf{i}$ ranges
over $\Xi$. Instead, it is useful to choose other partitions which
make their behavior more uniform. We present these next.

First, we would like (the supports of) the measures $\varphi_{\mathbf{i}}\mu$
to all have roughly the same diameter. To this end, for $n\geq1$
let 
\[
\Psi_{n}=\left\{ i_{0},\ldots,i_{m}\in\Lambda^{*}:\alpha_{1}(A_{i_{0},\ldots,i_{m}})\leq2^{-n}<\alpha_{1}(A_{i_{0},\ldots,i_{m-1}})\right\} 
\]
(we could have equivalently used norms instead of $\alpha_{1}$).
Because the $\varphi_{i}$ are contractions, $\Psi_{n}$ forms a partition
of $\Lambda^{\mathbb{N}}$ for every $n\geq1$ and it is easy to see
that there exists a constant $c_{0}>0$, depending on the matrices
but independent of $n$, such that for every $\mathbf{i}\in\Psi_{n}$,
\[
c_{0}2^{-n}\leq\alpha_{1}(A_{\mathbf{i}})=\Vert A_{\mathbf{i}}\|\leq2^{-n}\;.
\]

Next, we will sometimes want the ``width'' of the cylinder $\varphi_{\mathbf{i}}\mu$
to vary uniformly. Thus, for $n\geq1$ define
\[
\Upsilon_{n}=\left\{ i_{1}...i_{m}\in\Lambda^{*}\::\:\alpha_{2}(A_{i_{1}...i_{m}})\le2^{-n}<\alpha_{2}(A_{i_{1}...i_{m-1}})\right\} \:.
\]
Then there is a constant $c'_{0}>0$ such that for every $\mathbf{i}\in\Upsilon_{n}$,
\[
c'_{0}2^{-n}\leq\alpha_{2}(A_{\mathbf{i}})\leq2^{-n}\;.
\]

Every measure on Euclidean space has associated to it its dyadic components.
For a planar self-affine measure $\mu$, one can also decompose $\mu$
into cylinder measures, i.e. measure of the form $\varphi_{\mathbf{i}}\mu$
for $\mathbf{i}\in\Lambda^{*}$. As with dyadic components it is natural
to view the cylinders as random measures, with the naturally defined
probabilities.

For any given $n\in\mathbb{N}$ we introduce a random word $\mathbf{U}(n)\in\Lambda^{n}$
chosen according to the probability measure $p^{n}$. That is,
\[
\mathbb{P}(\mathbf{U}(n)=\mathbf{i})=\left\{ \begin{array}{cc}
p_{\mathbf{i}} & \text{if }\mathbf{i}\in\Lambda^{n}\\
0 & \text{otherwise}
\end{array}\right.\;.
\]
Similarly, we define the random word $\mathbf{I}(n)\in\Psi_{n}$ according
to the probability vector $p$, i.e. 
\[
\mathbb{P}(\mathbf{I}(n)=\mathbf{i})=\begin{cases}
p_{\mathbf{i}} & \text{ if }\mathbf{i}\in\Psi_{n}\\
0 & \text{ otherwise}
\end{cases},
\]
and define $\textbf{\ensuremath{\mathbf{K}}}(n)$ to be the random
word taking values in $\Upsilon_{n}$ according to $p$, i.e.
\[
\mathbb{P}\left(\textbf{\ensuremath{\mathbf{K}}}(n)=w\right)=\begin{cases}
p_{w} & \text{if }w\in\Upsilon_{n}\\
0 & \text{otherwise}
\end{cases}\;.
\]
The representation of $\mu$ as a convex combination of cylinder measure
in equation (\ref{eq:iterated-mu}) then takes the form 
\begin{align}
\mu & =\mathbb{E}(\varphi_{\mathbf{U}(n)}\mu)\label{eq:decomposition-of-mu-into-levels}\\
 & =\mathbb{E}(\varphi_{\mathbf{I}(n)}\mu)\nonumber \\
 & =\mathbb{E}(\varphi_{\mathbf{K}(n)}\mu)\;.\nonumber 
\end{align}
The first represents $\mu$ as a combination of cylinder measures
of fixed length $n$, the second as a combination of cylinders having
diameter equal to $2^{-n}$ up to a constant factor, and the last
as a combination of cylinders of width $2^{-n}$ up to a constant
factor. We may also randomize $n$ in the same way as we do in the
case of components, thus for example for any observable $F$, 

\[
\mathbb{E}_{n_{1}\leq i\leq n_{2}}(F(\varphi_{\mathbf{I}(i)}\mu))=\frac{1}{n_{2}-n_{1}+1}\sum_{i=n_{1}}^{n_{2}}\mathbb{E}(F(\varphi_{\mathbf{I}(i)}\mu)),
\]
and use the same notation for probabilities and expectations over
the random cylinders $\varphi_{\mathbf{K}(n)}\mu$.

\subsection{Entropy\label{sub:entropy}\label{sub:entropy-dimension}}

Let $\nu$ be a probability measure and $\mathcal{Q},\mathcal{Q}'$
finite or countable partitions of the underlying probability space.
The entropy of $\nu$ with respect to the partition $\mathcal{Q}$
is denoted $H(\nu,\mathcal{Q})$, and, when conditioned on $\mathcal{Q}'$,
by $H(\nu,\mathcal{Q}|\mathcal{Q}')$. That is, 
\begin{eqnarray}
H(\nu,\mathcal{Q}) & = & -\sum_{I\in\mathcal{Q}}\nu(I)\log\nu(I)\nonumber \\
H(\nu,\mathcal{Q}|\mathcal{Q}') & = & H(\nu,\mathcal{Q}\vee\mathcal{Q}')-H(\nu,\mathcal{Q}')\label{eq:conditional-entropy-is-increment}\\
 & = & \sum_{I\in\mathcal{Q}'}\nu(I)\cdot H(\nu_{I},\mathcal{Q}),\label{eq:cond-entropy-as-expectation}
\end{eqnarray}
assuming the sums are finite. Here $\mathcal{Q}'\vee\mathcal{Q}$
denotes the common refinement of the partitions $\mathcal{Q}',\mathcal{Q}$,
and by convention the logarithms are in base $2$, and $0\log0=0$. 

The entropy function is concave and almost convex in the measure argument.
That is, if $\nu_{i}$ are measures and $(q_{i})$ a probability vector,
then 
\[
\sum q_{i}H(\nu_{i},\mathcal{Q})\leq H(\sum q_{i}\nu_{i},\mathcal{Q})\leq\sum q_{i}H(\nu_{i},\mathcal{Q})+H(q)\;,
\]
where $H(q)=-\sum q_{i}\log q_{i}$. 

If $\mathcal{Q},\mathcal{Q}'$ are $C$-commensurable partitions (i.e.
each atom of one intersects at most $C$ atoms of the other), then
they have comparable entropies; more generally, replacing any one
of the partitions in the expression $H(\nu,\mathcal{A}\lor\mathcal{B}|\mathcal{C}\lor\mathcal{D})$
by a partition that is $C$-commensurable to it results in an additive
$O_{C}(1)$ change in value. 

The entropy function $\nu\mapsto H(\nu,\mathcal{Q}\mid\mathcal{Q}')$
is continuous in the total variation distance $d_{TV}(\cdot,\cdot)$.
In fact, if $d_{TV}(\nu,\theta)<\varepsilon$ and if each atom of
$\mathcal{Q}'$ intersects at most $k$ atoms of $\mathcal{Q}$, then
as in (\cite[Lemma 3.4]{HO2}): 
\begin{equation}
|H(\nu,\mathcal{Q}|\mathcal{Q}')-H(\theta,\mathcal{Q}|\mathcal{Q}')|\leq2\varepsilon\log k+2H(\frac{\varepsilon}{2})\;.\label{eq:entropy-continuity-of-images}
\end{equation}
In particular, using the fact that for $n>m$ each atom of $\mathcal{D}_{m}^{d}$
intersects $2^{d(n-m)}$ atoms of $\mathcal{D}_{n}^{d}$, this implies
that if $d_{TV}(\nu,\theta)<\varepsilon$, then

\begin{equation}
|\frac{1}{n-m}H(\nu,\mathcal{D}_{n}|\mathcal{D}_{m})-\frac{1}{n-m}H(\theta,\mathcal{D}_{n}|\mathcal{D}_{m})|<2d\varepsilon+\frac{2H(\frac{\varepsilon}{2})}{n-m}\;.\label{eq:scale-n-entropy-continuity-in-TV}
\end{equation}
The same bound holds using dyadic partitions in any orthogonal coordinate
system $W\oplus W^{\perp}$.

\subsection{\label{subsec:Entropy-in-Rd}Entropy in $\mathbb{R}^{d}$}

For a $\nu\in\mathcal{P}(\mathbb{R}^{d})$ or $\nu\in\mathcal{P}(A_{2,2})$,
we call $H(\nu,\mathcal{D}_{n})$ the scale-$n$ entropy of $\nu$.
We collect some basic properties of this quantity.

We often normalize by $n$, in which case
\[
\frac{1}{n}H(\nu,\mathcal{D}_{n})\leq d+O\left(\frac{\log(2+\diam(\supp\nu))}{n}\right)\:.
\]
By the definition of the distribution on components, for $n,m\ge1$,
\begin{eqnarray}
H(\nu,\mathcal{D}_{n+m}|\mathcal{D}_{n}) & = & \mathbb{E}(H(\nu_{x,n},\mathcal{D}_{n+m})).\label{eq:component-entropy-is-conditional-entropy}
\end{eqnarray}
Hence, for $\nu\in\mathcal{P}(\mathbb{R}^{d})$ we have the bound
\[
\frac{1}{k}H(\nu,\mathcal{D}_{n+k}|\mathcal{D}_{n})\leq d,
\]
and similarly in $A_{2,2}$ with another constant on the right hand
side. 

Scale-$n$ entropy is insensitive to coordinate changes: for $\nu\in\mathcal{P}(\mathbb{R}^{2})$
and $W\in\PR$, the partitions $\mathcal{D}_{n}$ and $\mathcal{D}_{n}^{W\oplus W^{\perp}}$
are $O_{d}(1)$-commensurable, hence

\begin{equation}
|H(\theta,\mathcal{D}_{n})-H(\theta,\mathcal{D}_{n}^{W\oplus W^{\perp}})|=O(1)\:.\label{eq:cond on comm partitions}
\end{equation}
and similarly for conditional entropy.

Scale-$n$ entropy transforms nicely under similarity maps: For any
similarity $f:\mathbb{R}^{d}\rightarrow\mathbb{R}^{d}$ and $\nu\in\mathcal{P}(\mathbb{R}^{d})$,
writing $\lip(f)$ for the Lipschitz constant of $f$,
\begin{eqnarray}
H(f\nu,\mathcal{D}_{n}) & = & H(\nu,\mathcal{D}_{n+\log\lip(f)})+O(1)\label{eq:entropy-under-transformation-1}\\
 & = & H(\nu,\mathcal{D}_{n})+O(1+|\log\lip(f)|)\;.\label{eq:entropy-under-transformation-2}
\end{eqnarray}
In particular, recalling the notation $T_{a},S_{c}$ for translation
and scaling,
\begin{align}
H(T_{a}\nu,\mathcal{D}_{n}) & =H(\nu,\mathcal{D}_{n})+O(1)\text{ for }a\in\mathbb{R}^{d},\label{eq:entropy-under-translation}\\
H(S_{c}\nu,\mathcal{D}_{n}) & =H(\nu,\mathcal{D}_{n+\log c})+O(1)\text{ for }c>0\;.\nonumber 
\end{align}
Thus, using equation (\ref{eq:affine-map-to-projection}) and Lemma
\ref{lem:norm-of-projection-composed-with-matrix}, if $\varphi(x)=Ax+b\in A_{2,2}$
and $W\in\PR$ satisfy $d_{\PR}(L(A),W^{\perp})\geq c$, then for
every measure $\nu\in\mathcal{P}(\mathbb{R}^{2})$ and every $n$,
\begin{align}
H(\pi_{W}\varphi\nu,\mathcal{D}_{n}) & =H(\pi_{A^{*}W}\nu,\mathcal{D}_{n+\log\left\Vert A\right\Vert })+O_{c}(1)\;.\label{eq:entropy-of-projection-of-affine-image-of-a-measure}
\end{align}
Similarly, as a consequence of concavity and of (\ref{eq:entropy-under-translation}),
for any $\theta,\nu\in\mathcal{P}(\mathbb{R}^{d})$ we have
\begin{equation}
H(\theta*\nu,\mathcal{D}_{n})\geq H(\nu,\mathcal{D}_{n})+O(1)\;.\label{eq:entorpy-increase-by-convolution}
\end{equation}
Also, the entropy of images is nearly continuous in the map: If $f,g$
are such that $\sup_{x}|f(x)-g(x)|<2^{-n}$ then 
\begin{equation}
|H(f\nu,\mathcal{D}_{n})-H(g\nu,\mathcal{D}_{n})|=O(1)\;.\label{eq:entropy-under-transformation-3}
\end{equation}

For a $\nu\in\mathcal{P}(\mathbb{R}^{d})$, the entropy dimension
of $\nu$ is defined as 
\begin{eqnarray*}
\edim\nu & = & \lim_{n\to\infty}\frac{H(\nu,\mathcal{D}_{n})}{n}
\end{eqnarray*}
if the limit exists (otherwise we take limsup or liminf as appropriate,
denoted $\uedim\nu$ and $\ledim\nu$).
\begin{lem}
\label{lem:entropy-dimension} If $\nu\in\mathcal{P}(\mathbb{R}^{d})$
is exact dimensional then $\dim_{e}\nu$ exists, moreover, 
\[
\dim\nu=\lim_{n\to\infty}\frac{H(\nu,\mathcal{D}_{n})}{n}.\;
\]
\end{lem}

The proof of the lemma can be found in e.g. \cite{FanLauRao02}. 

The following lemma express entropy in terms of the contribution of
different ``scales''. The proof is identical (or in the case of
$A_{2,2}$, similar) to the proof of \cite[Lemma 3.4]{Ho}, and is
therefore omitted.
\begin{lem}
\label{lem:multiscale-entropy-formula}Let $\theta\in\mathcal{P}(\mathbb{R}^{d})$
or $\theta\in\mathcal{P}(A_{2,2})$, let $n\ge m\ge1$, and let $k\ge0$
be given. Suppose that $\diam(\supp(\theta))=O(2^{-k})$. Then,
\[
\frac{1}{n}H(\theta,\mathcal{D}_{k+n})=\mathbb{E}_{k\le i\le k+n}\left(\frac{1}{m}H\left(\theta_{\psi,i},\mathcal{D}_{i+m}\right)\right)+O\left(\frac{m}{n}\right)\;.
\]
\end{lem}

\subsection{Random matrix products, Furstenberg measure, and $L$ again\label{sub:furstenberg}}

We rely on the following classical results about random matrix products,
see e.g. \cite[Chapter III]{BL}.
\begin{thm}
\label{thm:Furstenberg-Oseledets}Let $\{B_{i}\}_{i\in\Gamma}$ be
a finite set of invertible matrices and $q=\sum_{i\in\Gamma}q_{i}\cdot\delta_{B_{i}}$
a probability measure on $GL_{2}(\mathbb{R})$, with $q_{i}>0$. Assume
that $\{B_{i}\}$ is non-conformal and totally irreducible (in the
sense in the introduction). Let $\zeta_{1},\zeta_{2},\ldots$ be an
i.i.d. sequence of matrices with marginal distribution $q$. Then,
\begin{enumerate}
\item \label{enu:Lyapunov exponents}There exist constants $\chi_{1}>\chi_{2}$
(called the Lyapunov exponents), such that with probability one,
\begin{align*}
\alpha_{1}(\zeta_{1}\ldots\zeta_{n}) & =2^{(\chi_{1}+o(1))n}\\
\alpha_{2}(\zeta_{1}\ldots\zeta_{n}) & =2^{(\chi_{2}+o(1))n}
\end{align*}
as $n\rightarrow\infty$. The same holds if the order of the products
is reversed (since $B,B^{*}$ have the same singular values).
\item For every $v\in\mathbb{R}^{2}$, with probability one, 
\begin{align*}
\left\Vert \zeta_{n}\ldots\zeta_{1}v\right\Vert  & =2^{(\chi_{1}+o(1))n}\\
\left\Vert \zeta_{n}^{-1}\ldots\zeta_{1}^{-1}v\right\Vert  & =2^{(-\chi_{2}+o(1))n}
\end{align*}
as $n\rightarrow\infty$ (The $o(n)$ error terms depend on the sample
$(\zeta_{i})$ and on $v$). If the matrices are multiplied in the
opposite order, the limits exist in probability.
\item \label{enu:convergence-of-L}There exists a random subspace $W\in\PR$
(which is a measurable function of $\zeta_{1},\zeta_{2},\ldots$),
such that with probability one,
\begin{align*}
\lim_{n\rightarrow\infty}L(\zeta_{1}\ldots\zeta_{n}) & =W\;.
\end{align*}
If the product is taken in the opposite order then $W$ is still the
limit in distribution (but generally not in probability). 
\item The distribution $\tau\in\mathcal{P}(\PR)$ of $W$ is the Furstenberg
measure associated to $q$. It is the unique measure satisfying $\tau=\sum_{i\in\Gamma}q_{i}\cdot B_{i}\tau$.
It has no atoms and $\dim\tau>0$.
\item \label{enu:convergence-of-RW}For any continuous measure $\lambda$
on $\PR$, we have 
\[
\lim_{n\rightarrow\infty}\mathbb{E}(\zeta_{1}\ldots\zeta_{n}(\lambda))=\tau\;,
\]
and, with probability one,
\begin{align*}
\lim_{n\rightarrow\infty}\zeta_{1}\ldots\zeta_{n}(\lambda) & =\delta_{W}\;.
\end{align*}
Furthermore,
\[
\lim_{n\rightarrow\infty}\zeta_{n}\ldots\zeta_{1}V=W\quad\text{in distribution and uniformly in \ensuremath{V\in\PR}.}
\]
\end{enumerate}
\end{thm}

We can view the function $L$ on matrices (Section \ref{subsec:The-function-L})
as a partially defined function on words in $\Lambda^{*}=\bigcup_{n=1}^{\infty}\Lambda^{n}$,
given by 
\[
L(i_{1}\ldots i_{n})=L(A_{i_{1}}\ldots A_{i_{n}})
\]
(it is defined whenever $A_{i_{1}}\ldots A_{i_{n}}$ have distinct
singular values). In view of Theorem \ref{thm:Furstenberg-Oseledets}
(\ref{enu:convergence-of-L}), we can extend the function $L$ to
a $\xi$-a.e. defined function of infinite sequences:
\begin{defn}
Given our system of affine maps $\{\varphi_{i}\}_{i\in\Lambda}$ with
$\varphi_{i}(x)=A_{i}x+b_{i}$, and a probability vector $p=(p_{i})_{i\in\Lambda}$,
we define $L:\Lambda^{\mathbb{N}}\rightarrow\PR$ by
\[
L(\omega)=\lim_{n\rightarrow\infty}L(A_{\omega_{1}}\ldots A_{\omega_{n}})\;.
\]
The limit in the definition exists $\xi$-a.e. by Theorem \ref{thm:Furstenberg-Oseledets}.
We also define $\eta=L\xi$, and note that for any continuous measure
$\lambda$ on $\PR$, by part (\ref{enu:convergence-of-RW}) of the
same theorem, for $\xi$-a.e. $\omega\in\Lambda^{\mathbb{N}}$,
\[
\delta_{L(\omega)}=\lim_{n\rightarrow\infty}A_{\omega_{1}}\ldots A_{\omega_{n}}\lambda\;.
\]
We define $\eta^{*}$ analogously, using the system of matrices $(A_{i}^{*})$
and $p$.
\end{defn}

The following is a variant of \cite[Lemma 2.6]{BHR}. We include it
here for completeness:
\begin{prop}
\label{prop:convergence-to-eta*}With the notation of Section \ref{sub:random-cylinder-measures},
for every $V\in\PR$ we have 
\[
\mathbb{E}_{1\leq i\leq n}\left(\delta_{A_{\mathbf{I}(i)}^{*}V}\right)\ll\mathbb{E}_{1\leq i\leq Cn}\left(\delta_{A_{\mathbf{U}(i)}^{*}V}\right)\;,
\]
where $C=C(\{\varphi_{i}\})$, and furthermore, the Radon-Nykodym
derivative of the measures above is bounded by $C$, uniformly in
$n$ and $V$. Consequently, if $\mathcal{U}\subseteq\PR$ is an open
set and $\eta^{*}(\mathcal{U})>1-\varepsilon$ for some $\varepsilon>0$
then for $n>n(\varepsilon,\mathcal{U})$, 
\[
\underset{V\in\PR}{\inf}\:\mathbb{E}_{1\leq i\leq n}\left(\delta_{A_{\mathbf{I}(i)}^{*}V}(\mathcal{U})\right)>1-C\varepsilon\;.
\]
\end{prop}

\begin{proof}
Choose $C$ such that $\max\lip(\varphi_{i})^{C/2}<1/2$. If $u\in\Lambda^{k}$
appears as $\mathbf{I}(i)$ on the left hand side then $\left\Vert A_{u}\right\Vert \geq c_{0}2^{-i}\geq c_{0}2^{-n}$
(recall the definition of $\mathbf{I}(i)$), which, using $\left\Vert A_{u}\right\Vert \leq\prod_{i=1}^{k}\left\Vert A_{u_{j}}\right\Vert $
implies that $k\leq(C/2)(n-\log c_{0})$, which is $\leq Cn$ for
large $n$; so $u$ appears on the right hand side as well. Furthermore
its probability in the expectation on the left is $p_{u}/n$, while
on the right the corresponding term has probability $p_{u}/Cn$. This
proves absolute continuity and shows that the Radon-Nykodym derivative
is $\leq C$. For the last statement, by theorem \ref{thm:Furstenberg-Oseledets}
(\ref{enu:convergence-of-RW}), $\mathbb{E}_{1\leq i\leq Cn}(\delta_{A_{\mathbf{U}(i)}^{*}V})\rightarrow\eta^{*}$
as $n\rightarrow\infty$ uniformly in $V\in\PR$. We conclude that
\[
\limsup_{n\rightarrow\infty}\:\underset{V\in\PR}{\sup}\:\mathbb{E}_{1\leq i\leq Cn}(\delta_{A_{\mathbf{U}(i)}^{*}V}(\PR\setminus\mathcal{U}))\leq\eta^{*}(\PR\setminus\mathcal{U})<\varepsilon\;,
\]
and apply the first part to find that
\[
\limsup_{n\rightarrow\infty}\:\underset{V\in\PR}{\sup}\:\mathbb{E}_{1\leq i\leq n}(\delta_{A_{\mathbf{I}(i)}^{*}V}(\PR\setminus\mathcal{U}))<C\varepsilon\;.
\]
\end{proof}

\section{\label{sec:Entropy}Entropy of projections and slices of $\mu$}

In this section we assume that $\Phi$ is totally irreducible and
non-conformal, but we do not assume exponential separation or $\dim\mu\geq1$. 

Recall that
\begin{align*}
\alpha & =\dim\mu\\
\beta & =\dim\pi_{W}\mu\qquad\text{for }\eta^{*}\text{-a.e. }W\\
\gamma & =\alpha-\beta
\end{align*}
($\beta$ is well defined by Theorem \ref{thm:LY}). Lemma \ref{lem:entropy-dimension}
tells us that for $\eta^{*}$-a.e. $W$, the entropy of $\pi_{W}\mu$
at a large scale $n$ is close to $n\beta$. In this section we get
a similar lower bound for all (rather than $\eta^{*}$-almost-all)
projections of $\mu$, uniformly in the direction of projection, and
also projections of cylinders $\varphi_{i_{1}}\ldots\varphi_{i_{n}}\mu$,
and of components $\mu_{x,i}$. We also examine certain conditional
measures of $\mu$ along lines perpendicular to $\eta^{*}$-typical
directions, and determine their entropies. 

The methods here are mostly not new, and some of the statements have
also appeared elsewhere, but others have not. We give a full development
for completeness.

\subsection{\label{subsec:Projections-of-mu-cylinders}Projections of $\mu$
and its cylinders}

One of the basic mechanisms in the study of self-affine measures is
that projecting a \emph{typical }cylinder measure in a \emph{fixed
}direction is essentially the same as projecting $\mu$ in an $\eta^{*}$-\emph{random
}direction, because the ``orientation'' of high-generation cylinders
becomes increasingly random. In the discussion below, the reader should
note the different roles of the Furstenberg measure $\eta$ associated
to the random matrix product of the $A_{i}$, and the Furstenberg
measure $\eta^{*}$, associated to the products of the transposed
matrices, $A_{i}^{*}$. 

To see how $\eta^{*}$ comes into the picture, observe that if $\mathbf{i}=i_{1}\ldots i_{n}\in\Lambda^{n}$
and $W\in\PR$ are fixed, then, writing $t=t(\mathbf{i})=\left\Vert \pi_{W}A_{i_{1}}\ldots A_{i_{n}}\right\Vert $,
by (\ref{eq:affine-map-to-projection}) we have 
\[
\pi_{W}A_{i_{1}}\ldots A_{i_{n}}=\pm S_{t}\pi_{A_{i_{n}}^{*}\ldots A_{i_{1}}^{*}W}
\]
(recall that $S_{t}x=tx$ is the scaling operator). This means that,
up to a translation and reflection, the projection onto $W$ of the
cylinder $\varphi_{i_{1}}\ldots\varphi_{i_{n}}\mu$ is just the projection
of $\mu$ to another line (the line $A_{i_{n}}^{*}\ldots A_{i_{1}}^{*}W$),
but scaled by $\left\Vert \pi_{W}A_{i_{1}}\ldots A_{i_{n}}\right\Vert $.
The subspace $A_{i_{n}}^{*}\ldots A_{i_{1}}^{*}W$, when $i_{1}\ldots i_{n}$
are chosen at random according to $p^{n}$, is asymptotically (as
$n\rightarrow\infty$) distributed like $\eta^{*}$.

To see how $\eta$ enters the picture, note that in order for the
analysis above to be useful we must have control of the norm $t=\left\Vert \pi_{W}A_{i_{1}}\ldots A_{i_{n}}\right\Vert $.
This norm depends on two factors. The first is the norm $\left\Vert A_{i_{1}}\ldots A_{i_{n}}\right\Vert $
of the matrix product, which is a function of the sequence $i_{1}\ldots i_{n}$
(not only of $n$). Because of this, later we will usually not choose
a sequence of constant length $n$, but rather condition the sequence
on the desired norm. This is what the random word $\mathbf{I}(n)$
does (see Section \ref{sub:random-cylinder-measures}).\footnote{Choosing variable length words complicates the equidistribution properties
of $A_{i_{n}}^{*}\ldots A_{i_{1}}^{*}W$ and is the reason we need
Proposition \ref{prop:convergence-to-eta*}.} The second factor controlling the norm $t$ is how the direction
$L(A_{i_{1}}\ldots A_{i_{n}})$ of the cylinder $\varphi_{i_{1}}\ldots\varphi_{i_{n}}\mu$
lies in relation to $W^{\perp}$: if $L(A_{i_{1}}\ldots A_{i_{n}})$
is far from $W^{\perp}$ then the norms of $\pi_{W}A_{i_{1}}\ldots A_{i_{n}}$
and $A_{i_{1}}\ldots A_{i_{n}}$ will be comparable; if they are close,
the former might be far smaller. The directions $L(A_{i_{1}}\ldots A_{i_{n}})$,
when $i_{1}\ldots i_{n}$ is chosen at random according to $p^{n}$,
are asymptotically distributed like $\eta$.

These considerations underlie the following lemmas. Since our ultimate
goal is to compute entropies, they are formulated in that way. Recall
the definition of $\Psi_{n}$ and $\mathbf{I}(n)$ from Section \ref{sub:random-cylinder-measures},
and that $\Psi_{n}(\omega)$ denotes the unique word $w\in\Psi_{n}$
with $\omega\in[w]$.
\begin{lem}
\label{lem:projecting-cylinders-away-from-L-perp}For every $\varepsilon>0$
and $\rho>0$, if $m>M(\varepsilon,\rho)$, the following holds for
every $n\geq1$:

For every $W\in\PR$ and every $u\in\Psi_{n}$ satisfying $d_{\PR}(L(A_{u}),W^{\perp})\geq\rho$,
\[
\left|\frac{1}{m}H(\pi_{W}\varphi_{u}\mu,\mathcal{D}_{n+m})-\frac{1}{m}H(\pi_{A_{u}^{*}W}\mu,\mathcal{D}_{m})\right|<\varepsilon\;.
\]
\end{lem}

\begin{proof}
Using $\left\Vert A_{u}\right\Vert =2^{-n+O(1)}$ (because $u\in\Psi_{n}$)
and the hypothesis $d(L(A_{u}),W^{\perp})\geq\rho$, equation (\ref{eq:entropy-of-projection-of-affine-image-of-a-measure})
implies 
\[
\frac{1}{m}H\left(\pi_{W}\varphi_{u}\mu,\mathcal{D}_{n+m}\right)=\frac{1}{m}H\left(\pi_{A_{u}^{*}W}\mu,\mathcal{D}_{m}\right)+O_{\rho}(\frac{1}{m})\;,
\]
which gives the claim as soon as $m$ is large enough. 
\end{proof}
For this lemma to be useful we must bound the probability that $L(A_{u})$
is close to $W^{\perp}$. We have already observed that when $n$
is large, $L(A_{u})$ is distributed approximately like $\eta$, which
is a continuous measure (has no atoms), and so the probability that
$L(A_{u})$ is within distance $\rho$ of a fixed $W^{\perp}$ is
asymptotically $\eta(B_{\rho}(W^{\perp}))$, which is negligible when
$\rho$ is small. This argument is formalized in the next lemma. 
\begin{lem}
\label{lem:L-of-cylinders-equidistributes}For every $\varepsilon>0$
and every $0<\rho\leq\rho(\varepsilon)$, if $n\geq N(\varepsilon,\rho)$
then for every $W\in\PR$,
\[
\mathbb{P}(d_{\PR}(L(A_{\mathbf{I}(n)}),W^{\perp})\geq\rho)>1-\varepsilon\;.
\]
 
\end{lem}

\begin{proof}
The measure $\eta=L\xi$ is continuous, hence there exists $\rho(\varepsilon)>0$
such that for any $0<\rho\leq\rho(\varepsilon)$ we have $L\xi(B(W,2\rho))<\varepsilon/2$
for all $W\in\PR$. 

By the definition of $L$, the sequence $\{L(A_{\Psi_{n}(\omega)})\}_{n\ge1}$
converges to $L(\omega)$ for $\xi$-a.e. $\omega\in\Lambda^{\mathbb{N}}$.
For each $n\ge1$ and $w\in\Lambda^{*}$, by definition
\[
\mathbb{P}\left(\mathbf{I}(n)=w\right)=\xi\{\omega\::\:\Psi_{n}(\omega)=w\}\:.
\]
It follows that $\{L(A_{\mathbf{I}(n)})\}_{n\ge1}$ converges in distribution
to $L$, where we consider $L$ as a random variable on $(\Lambda^{\mathbb{N}},\xi)$.
Hence for every $n\ge1$ large enough in a manner depending on $\varepsilon$
and $\rho$, and for any $W\in\PR$,
\[
\mathbb{P}\left(L(A_{\mathbf{I}(n)})\in B(W,\rho)\right)<\varepsilon\;,
\]
as claimed.
\end{proof}
What we have done so far shows that $\pi_{W}\varphi_{\mathbf{I}(n)}\mu$
is, with high probability, comparable to $\pi_{A_{\mathbf{I}(n)}^{*}W}\mu$
at another scale. For this to be useful we now must understand the
distribution of $A_{\mathbf{I}(n)}^{*}$. Here we meet the random
matrix product associated of the transpose matrices $A_{i}^{*}$.
These should heuristically converge to $\eta^{*}$, but the equidistribution
properties of this random walk are not as good, due to the fact that
we have only convergence in distribution (and not pointwise, due to
the order of composition), and because we are interested in the behavior
along a certain random sub-sequence of times (those which define the
lengths of $\mathbf{I}(n)$). Nevertheless in the Cesaro sense the
random walk $A_{\mathbf{I}(n)}^{*}W$ does equidistribute to $\eta^{*}$,
allowing us in the next lemma to get information about the projections
of typical cylinders (and hence of $\mu$) in a fixed direction $W$. 
\begin{lem}
\label{lem:uni conv of ent}For every $\varepsilon>0$ and $n\ge N(\varepsilon)\ge1$,
\[
\underset{W\in\PR}{\inf}\:\frac{1}{n}H(\pi_{W}\mu,\mathcal{D}_{n})>\beta-\varepsilon\:.
\]
\end{lem}

\begin{proof}
Let $\varepsilon>0$, choose $\rho$ suitable for the previous lemma,
and let $n>m\ge1$, with $m$ large with respect to $\varepsilon$
and $\rho$, and $n$ large with respect to all parameters, we shall
see the relations later.

By Lemma \ref{lem:multiscale-entropy-formula} and by assuming that
$n$ is sufficiently large with respect to $m$, it follows that for
$W\in\PR$
\[
\frac{1}{n}H(\pi_{W}\mu,\mathcal{D}_{n})=\frac{1}{n}\sum_{k=1}^{n}\frac{1}{m}H(\pi_{W}\mu,\mathcal{D}_{k+m}\mid\mathcal{D}_{k})+O(\varepsilon)\:.
\]
For each $k\ge1$ we have $\pi_{W}\mu=\mathbb{E}_{i=k}\left(\pi_{W}\varphi_{\mathbf{I}(i)}\mu\right)$,
thus by the concavity of conditional entropy
\[
\frac{1}{n}H(\pi_{W}\mu,\mathcal{D}_{n})\ge\frac{1}{n}\sum_{k=1}^{n}\mathbb{E}\left(\frac{1}{m}H(\pi_{W}\varphi_{\mathbf{I}(k)}\mu,\mathcal{D}_{i+m}\mid\mathcal{D}_{i})\right)-O(\varepsilon)\:.
\]
Since $\diam(\supp(\varphi_{\mathbf{I}(i)}\mu))=\Theta(2^{-i})$ and
by assuming that $m$ is sufficiently large with respect to $\varepsilon$,
we can do away with the conditioning at the expense of a slight increase
to the error term:
\begin{eqnarray*}
\frac{1}{n}H(\pi_{W}\mu,\mathcal{D}_{n}) & \ge & \frac{1}{n}\sum_{k=1}^{n}\mathbb{E}\left(\frac{1}{m}H(\pi_{W}\varphi_{\mathbf{I}(k)}\mu,\mathcal{D}_{i+m})\right)-O(\varepsilon)\\
 & = & \mathbb{E}_{1\le i\leq n}\left(\frac{1}{m}H(\pi_{W}\varphi_{\mathbf{I}(i)}\mu,\mathcal{D}_{i+m})\right)-O(\varepsilon)\;.
\end{eqnarray*}
By Lemma \ref{lem:L-of-cylinders-equidistributes} and Lemma \ref{lem:projecting-cylinders-away-from-L-perp}
above, by our choice of $\rho$ and by assuming $m,n$ are large relative
to $\varepsilon,\rho$, outside an event of probability $<\varepsilon$,
the expression in the last expectation can be replaced with projection
to $A_{\mathbf{I}(n)}^{*}W$ at the expense of another $\varepsilon$
error, hence
\begin{equation}
\frac{1}{n}H(\pi_{W}\mu,\mathcal{D}_{n})\geq\mathbb{E}_{1\le i\leq n}\left(\frac{1}{m}H(\pi_{A_{\mathbf{I}(i)}^{*}W}\mu,\mathcal{D}_{m})\right)-O(\varepsilon)\;.\label{eq:bound-on-arbitrary-projcetion}
\end{equation}
The point now is that, roughly speaking, $A_{\mathbf{I}(n)}^{*}W$
equidistributes to $\eta^{*}$. This is not precisely true; what is
true is that $A_{\mathbf{U}(n)}^{*}W$ equidistributes for $\eta^{*}$.
The two sequences are not quite comparable, but the two distributions
are close enough that they hit high-probability events with roughly
proportional probabilities, and this is enough to complete the proof;
the technical step is given by Proposition \ref{prop:convergence-to-eta*}.
In more detail, observe that, since $\dim\pi_{V}\mu=\beta$ for $\eta^{*}$-a.e.
$V$, if $m$ is large enough then $\frac{1}{m}H(\pi_{V}\mu,\mathcal{D}_{m})>\beta-\varepsilon/2$
for a set of $V$ of $\eta^{*}$-measure greater than $1-\varepsilon$.
Hence, using also the almost-continuity of entropy of projections,
we can find an open set $\mathcal{U}\subseteq\PR$ with $\eta^{*}(\mathcal{U})>1-\varepsilon$,
and such that $\frac{1}{m}H(\pi_{V}\mu,\mathcal{D}_{m})>\beta-\varepsilon$
for all $V\in\mathcal{U}$. Applying Proposition \ref{prop:convergence-to-eta*}
we conclude that for $n$ large relative to $\varepsilon$,
\[
\mathbb{P}_{1\le i\leq n}\left(\frac{1}{m}H(\pi_{A_{\mathbf{I}(i)}^{*}W}\mu,\mathcal{D}_{m})>\beta-\varepsilon\right)\geq1-O(\varepsilon)\;.
\]
Combined with (\ref{eq:bound-on-arbitrary-projcetion}) this completes
the proof.
\end{proof}
Lastly, we obtain a similar result for cylinders:
\begin{lem}
\label{lem:ent of proj of cyl}For every $\varepsilon>0$, for $m\ge M(\varepsilon)$
and $n\ge N(\varepsilon)$, 
\[
\underset{W\in\PR}{\inf}\:\mathbb{P}\left(\frac{1}{m}H\left(\pi_{W}\varphi_{\mathbf{I}(n)}\mu,\mathcal{D}_{i+m}\right)\ge\beta-\varepsilon\right)>1-\varepsilon\:.
\]
\end{lem}

\begin{proof}
From Lemmas \ref{lem:L-of-cylinders-equidistributes} and \ref{lem:projecting-cylinders-away-from-L-perp}
again, it is enough to prove (perhaps for another $\varepsilon$)
that 
\[
\underset{W\in\PR}{\inf}\:\mathbb{P}\left(\frac{1}{m}H\left(\pi_{A_{\mathbf{I}(n)}^{*}W}\mu,\mathcal{D}_{m}\right)\ge\beta-\varepsilon\right)>1-\varepsilon\:,
\]
and this follows from the previous lemma.
\end{proof}

\subsection{\label{subsec:Projections-mu-components}Projections of components
of $\mu$}

Another basic method is ``covering'', i.e. decomposition of measures
as convex combinations of well-behaved ones (and possibly a small
remainder). For example, one can cover (the restriction of $\mu$
to) dyadic cells by cylinders of roughly the same diameter. Since
entropy is concave, if in a cell $C\in\mathcal{D}_{n}$ we can express
$\mu$ as a convex combination of measures, most of which are cylinders
which project with large entropy in direction $W\in\PR$, then the
same should be true of the conditional measure $\mu_{C}$. A complication
arises here because there will in general be cylinder measures which
are partly, but not completely, supported on $C$, and then we lose
control of the behavior of the part of them that lies inside $C$.
But by controlling the mass of such cut-off cylinders, we can obtain
good decompositions of $\mu_{C}$ for most choices of $C$. This argument
depends on controlling the mass of small neighborhoods of $\partial C$.
That is the purpose of the following lemma:
\begin{lem}
\label{lem:dyadic-cells-with-small-boundaries}For every $\varepsilon>0$
there is a $\delta>0$ such that for every $W\in\PR$,
\[
\pi_{W}\mu(B_{\delta r}(x))\le\epsilon\cdot\pi_{W}\mu(B_{r}(x))\quad\text{ for all \ensuremath{x\in\mathbb{R}} and \ensuremath{0<r<1}.}
\]
In particular, for every $\varepsilon>0$ there is a $\delta>0$ such
that for all $n\ge1$,
\[
\mu\left(\bigcup_{D\in\mathcal{D}_{n}}(\partial D)^{(2^{-n}\delta)}\right)<\varepsilon\:.
\]
\end{lem}

\begin{proof}
The first part is a direct consequence of \cite[Lemma 3.13]{BHR}.
The second follows by decomposing $\cup_{D\in\mathcal{D}_{n}}(\partial D)^{(2^{-n}\delta)}$
into vertical strips and horizontal strips of width $2^{1-n}\delta$
and using the first part to estimate their mass. We omit the details.
\end{proof}
\begin{prop}
\label{prop:lb on ent of proj of comp of mu}For every compact $E\subseteq A_{2,2}$,
$\varepsilon>0$, $m\ge M(E,\varepsilon)$, and $n\ge N(\varepsilon)$,
\[
\underset{h\in E}{\inf}\:\underset{W\in\PR}{\inf}\:\frac{1}{m}H(h\mu,\pi_{W}^{-1}\mathcal{D}_{n+m}\mid\mathcal{D}_{n})\ge\beta-\varepsilon\:.
\]
\end{prop}

\begin{proof}
Let $E\subseteq A_{2,2}$ be compact. Given $h\in E$, $W\in\PR$,
and $n,m\ge1$, note that $h^{-1}\mathcal{D}_{n}$ is $O_{E}(1)$-commensurable
with $\mathcal{D}_{n}$, and also $h^{-1}\pi_{W}^{-1}\mathcal{D}_{n+m}$
is $O_{E}(1)$-commensurable with $S_{\left\Vert \pi_{W}A_{h}\right\Vert }^{-1}\pi_{A_{h}^{*}W}^{-1}\mathcal{D}_{n+m}$.
Thus by basic properties of entropy (see Section \ref{sub:entropy})
and the bound $\left\Vert \pi_{W}\circ A_{h}\right\Vert =\Theta_{E}(1)$
(because $E$ is compact), 
\begin{align*}
H(h\mu,\pi_{W}^{-1}\mathcal{D}_{n+m}\mid\mathcal{D}_{n}) & =H(\mu,h^{-1}\pi_{W}^{-1}\mathcal{D}_{n+m}\mid h^{-1}\mathcal{D}_{n})\\
 & =H(\mu,\pi_{A_{h}^{*}(W)}^{-1}\mathcal{D}_{n+m}\mid\mathcal{D}_{n})+O_{E}(1)\:.
\end{align*}
Hence it suffices to prove the proposition with $E=\{Id\}$.

Let $\varepsilon>0$ and let $m\ge M(\varepsilon)$ and $n\ge N(\varepsilon)$
be as in Lemma \ref{lem:ent of proj of cyl}. Fix $W\in\PR$. By the
concavity of conditional entropy and the fact that $\diam(\supp(\varphi_{\mathbf{I}(n)}\mu))=O(2^{-n})$,
\begin{align*}
\frac{1}{m}H(\mu,\pi_{W}^{-1}\mathcal{D}_{n+m}\mid\mathcal{D}_{n}) & \ge\mathbb{E}_{i=n}\left(\frac{1}{m}H(\varphi_{\mathbf{I}(i)}\mu,\pi_{W}^{-1}\mathcal{D}_{n+m}\mid\mathcal{D}_{n})\right)\\
 & \geq\mathbb{E}_{i=n}\left(\frac{1}{m}H(\varphi_{\mathbf{I}(i)}\mu,\pi_{W}^{-1}\mathcal{D}_{n+m})\right)+O(\frac{1}{m})\:.
\end{align*}
The proof is completed by an application of Lemma \ref{lem:ent of proj of cyl}.
\end{proof}
\begin{lem}
\label{lem:ent of proj of comp mostly large }For every $\varepsilon>0$,
$m\ge M(\varepsilon)\ge1$, and $n\ge N(\varepsilon)$,
\[
\underset{W\in\PR}{\inf}\:\mathbb{P}_{i=n}\left(\frac{1}{m}H(\pi_{W}\mu_{x,i},\mathcal{D}_{i+m})>\beta-\varepsilon\right)>1-\varepsilon\:.
\]
\end{lem}

\begin{proof}
When $\beta=1$ (which is the case under the assumptions of Theorem
\ref{thm:main-reformulated}, and what is needed to prove our main
theorem) the lemma is immediate from the previous proposition by starting
with $E=\{\id\}$ and a smaller $\varepsilon$, observing that
\[
H(\mu,\pi_{W}^{-1}\mathcal{D}_{n+m}\mid\mathcal{D}_{n})=\mathbb{E}_{i=n}(H(\pi_{W}\mu_{x,i},\mathcal{D}_{i+m}))\;,
\]
and applying Markov's inequality.

We include the proof of the case $\beta<1$ for completeness and future
reference. Let $\varepsilon>0$, let $\delta>0$ be small with respect
to $\varepsilon$, let $k\ge1$ be large with respect to $\delta$,
and let $m\ge1$ be large with respect to $k$. Also, let $n\ge1$
be large with respect to $\epsilon$ and fix $W\in\PR$.

By Lemma \ref{lem:dyadic-cells-with-small-boundaries} we may assume
that,
\[
\mu\left(\cup_{D\in\mathcal{D}_{n}}(\partial D)^{(2^{-n}\delta)}\right)<\varepsilon\:.
\]
Let $C=\diam(\supp\mu)$. Since $k$ is large with respect to $\delta$,
we may assume that if $\nu\in\mathcal{P}(\mathbb{R}^{2})$ is such
that $\diam(\supp\nu)\leq C\cdot2^{-n-k}$ and
\[
\#\{D\in\mathcal{D}_{n}\::\:\supp(\nu)\cap D\ne\emptyset\}>1,
\]
then $\supp(\nu)\subseteq\cup_{D\in\mathcal{D}_{n}}(\partial D)^{(2^{-n}\delta)}$.
It follows that
\begin{gather*}
\mathbb{P}_{i=n+k}\left(\varphi_{\mathbf{I}(i)}\mu\text{ is contained in a level-}n\text{ dyadic cell}\right)\\
\begin{aligned}> & 1-\mu(\cup_{D\in\mathcal{D}_{n}}(\partial D)^{(2^{-n}\delta)})\\
> & 1-\varepsilon\;.
\end{aligned}
\end{gather*}
On the other hand, by Lemma \ref{lem:ent of proj of cyl} (applied
with $n+k$ instead of $n$), 
\[
\mathbb{P}_{i=n+k}\left(\frac{1}{m}H\left(\pi_{W}\varphi_{\mathbf{I}(i)}\mu,\mathcal{D}_{i+m}\right)\ge\beta-\varepsilon\right)>1-\varepsilon\:.
\]
From the last two probability bounds and Markov's inequality, for
a $1-O(\sqrt{\varepsilon})$ fraction of dyadic cells $D\in\mathcal{D}_{n}$,
all but a $1-O(\sqrt{\varepsilon})$ fraction of the mass of $\mu_{D}$
can be expressed as a convex combination of cylinders $\varphi_{i}\mu$
whose projection in direction $W$ satisfies $(1/m)H(\pi_{W}\varphi_{i}\mu,\mathcal{D}_{n+k+m})>\beta-\varepsilon$.
For such a component, by concavity of entropy, we have $(1/m)H(\pi_{W}\mu_{D},\mathcal{D}_{n+k+m})>\beta-O(\sqrt{\varepsilon})$,
and, adjusting the scale from $n+k+m$ to $n+m$ at the cost of an
$O(k/m)$ error to entropy, and making $m$ large enough so that it
can be absorbed in the error term, we obtain
\[
\mathbb{P}_{i=n}\left(\frac{1}{m}H(\pi_{W}\mu_{x,i},\mathcal{D}_{i+m})>\beta-O(\sqrt{\varepsilon})\right)>1-O(\sqrt{\varepsilon})\;.
\]
This is what we wanted if we start from a smaller $\varepsilon$.
\end{proof}

\subsection{\label{subsec:Entropy-of-thickened-slices}Entropy of thickened slices}

In this section we use the eccentricity of cylinders in another way,
to control the conditional measures on fibers of an orthogonal projection.
More precisely, we condition the measure on $\pi_{W}^{-1}(I)$ for
a short interval $I$. If $\varphi_{i_{1}}\ldots\varphi_{i_{n}}\mu$
is a cylinder whose ``long'' direction is approximately $W^{\perp}$
then it will be contained in $\pi_{W}^{-1}(I)$ for some interval
$I$ whose length is close to $\alpha_{2}(A_{i_{1}}\ldots A_{i_{n}})$.
Its entropy, at scale $|I|$, will be comparable to the entropy of
its \emph{projection }to $W^{\perp}$, and this we know will be large.
Thus, restricting $\mu$ to the cylinders pointing in direction $W^{\perp}$,
we get good lower bounds on the conditional entropy with respect to
$\pi_{W}^{-1}\mathcal{D}_{n}$.

For $E\subset\Lambda^{\mathbb{N}}$ write $\mu_{E}=\Pi(\xi_{E})$
(recall that $\xi_{E}=\frac{1}{\xi(E)}\xi|_{E}$). 
\begin{lem}
\label{lem:ent of proj of restrictions}For every $\varepsilon,\rho>0$
and every $m\ge M_{1}(\varepsilon,\rho)$, the following holds. Let
$E\subseteq\Lambda^{\mathbb{N}}$ be a Borel set and $J\subset\PR$
be an open interval with $\xi(E\cap L^{-1}(J))>0$. Then for each
$W\in\PR$ with $d_{\PR}(W^{\perp},J)\ge\rho$ and $n\ge N_{1}(\varepsilon,\rho,m,E,J,W)$,
\[
\frac{1}{m}H\left(\mu_{E\cap L^{-1}(J)},\pi_{W}^{-1}\mathcal{D}_{n+m}\mid\mathcal{D}_{n}^{W\oplus W^{\perp}}\right)\ge\beta-\varepsilon\:.
\]
\end{lem}

\begin{proof}
Let $m\ge1$ be large in a manner depending on $\varepsilon,\rho$,
let $E\subset\Lambda^{\mathbb{N}}$ be a Borel set, let $J\subset\PR$
be an open interval with $\xi(E\cap L^{-1}(J))>0$, let $W\in\PR$
satisfy $d_{\PR}(W^{\perp},J)\ge\rho$, and let $n$ be large in a
manner depending on all parameters.

Write $F=E\cap L^{-1}(J)$. Since $\xi$ is a Borel probability measure
on $\Lambda^{n}$ it is a regular measure, so there exists an open
set $V\subset\Lambda^{\mathbb{N}}$ with $F\subset V$ and $\xi(V\setminus F)<\varepsilon\cdot\xi(F)$. 

Let $\mathcal{U}\subseteq\Psi_{n}$ be the set\footnote{In the definition of $\mathcal{U}$ we only take $u$ for which $L(A_{u})$
is defined. It may not be defined for all $u$, because it could be
that $A_{u}$ has equal singular values; but the probability of this
with respect to $\xi$ tends to zero as $n\rightarrow\infty$. } 
\[
\mathcal{U}=\{u\in\Psi_{n}\;:\;[u]\subseteq V\text{ and }L(A_{u})\in J\},
\]
 and write 
\[
U=\bigcup_{u\in\mathcal{U}}[u]\;.
\]
Since $V$ and $J$ are open and $L(A_{\omega_{1}\ldots\omega_{n}})\rightarrow L(\omega)$
for $\xi$-a.e. $\omega$, by assuming that $n$ is sufficiently large
we can ensure
\[
\xi_{V}(U)\ge\xi_{V}(F)-\varepsilon\ge1-2\varepsilon\;.
\]
Since $U,F\subseteq V$ and both differ in $\xi$-measure from $V$
by mass at most $2\varepsilon\xi(V)$, we conclude that $F\cap U$
differs from both $F$ and $U$ by at most $4\varepsilon\xi(V)$.
Hence in the sum $\xi|_{F}=\xi|_{F\cap U}+\xi|_{F\setminus U}$ all
but a relative $O(\varepsilon)$ of the mass is in the first term,
and similarly for $\xi|_{U}=\xi|_{F\cap U}+\xi|_{U\setminus F}$.
It follows that 
\[
d_{TV}(\xi_{U},\xi_{F})=O(\varepsilon)\;,
\]
hence
\[
d_{TV}(\mu_{U},\mu_{F})=O(\varepsilon)\;.
\]

By the definition of $\mathcal{U}$ and Lemmas \ref{lem:projecting-cylinders-away-from-L-perp}
and \ref{lem:uni conv of ent}, the fact that $\diam(\supp(\varphi_{u}\mu))=\Theta(2^{-n})$
and $d_{\PR}(W^{\perp},J)\ge\rho$, and assuming $m$ large relative
to $\varepsilon$ and $\rho$, we have 
\begin{align}
\frac{1}{m}H\left(\varphi_{u}\mu,\pi_{W}^{-1}\mathcal{D}_{n+m}\mid\mathcal{D}_{n}^{W\oplus W^{\perp}}\right) & \geq\frac{1}{m}H\left(\pi_{W}\varphi_{u}\mu,\mathcal{D}_{n+m}\right)-O(\frac{1}{m})\label{eq:proj ent big for u in U}\\
 & \ge\beta-O(\varepsilon)\text{ for }u\in\mathcal{U}\:.\nonumber 
\end{align}
Since $U$ is a union of cylinders from $\Psi_{n}$,
\[
\mu_{U}=\mathbb{E}(\varphi_{\mathbf{I}(n)}\mu\;|\;\mathbf{I}(n)\in\mathcal{U})\;,
\]
so by concavity of entropy and the previous inequality,
\[
\frac{1}{m}H\left(\mu_{U},\pi_{W}^{-1}\mathcal{D}_{n+m}\mid\mathcal{D}_{n}^{W\oplus W^{\perp}}\right)\geq\beta-O(\varepsilon)\;.
\]
The result now follows from $d_{TV}(\mu_{U},\mu_{F})=O(\varepsilon)$
combined with (\ref{eq:scale-n-entropy-continuity-in-TV}).
\end{proof}
\begin{lem}
\label{lem:ent of sections of restrictions}Let $\varepsilon>0$.
For every $m\ge M_{2}(\varepsilon)$ there exists $\delta=\delta(\varepsilon,m)>0$
such that the following holds. 

Let $E\subset\Lambda^{\mathbb{N}}$ be a Borel set and $I\subset\PR$
be an open interval with $\diam(I)<\delta$ and $\xi(E\cap L^{-1}(I))>0$.
Then for each $W\in\PR$ with $W^{\perp}\in I$ and $n\ge N_{2}(\varepsilon,m,\delta,E,I,W)$,
\[
\frac{1}{m}H\left(\mu_{E\cap L^{-1}(I)},\mathcal{D}_{n+m}^{W\oplus W^{\perp}}\mid\mathcal{D}_{n}^{W\oplus W^{\perp}}\vee\pi_{W}^{-1}\mathcal{D}_{n+m}\right)\ge\beta-\varepsilon\:.
\]
\end{lem}

\begin{proof}
Let $m\ge1$ be large in a manner depending on $\varepsilon$, let
$\delta>0$ small in a manner depending on $\varepsilon$ and $m$.
Let $E\subset\Lambda^{\mathbb{N}}$, $I\subseteq\PR$ and $W\in\PR$
be as in the statement and let $n$ be large in a manner depending
on all parameters.

Write $F=E\cap L^{-1}(I)$. Since $\xi$ is regular there exists an
open $V\in\Lambda^{\mathbb{N}}$ with $F\subset V$ and $\xi(V\setminus F)<\varepsilon\cdot\xi(F)$.
Let 
\[
\mathcal{U}=\{u\in\Psi_{n}\;:\;[u]\subseteq V\;,\;\frac{\alpha_{1}(A_{u})}{\alpha_{2}(A_{u})}>2^{m}\text{ and }L(A_{u})\in I\}\;,
\]
and write 
\[
U=\bigcup_{u\in\mathcal{U}}[u]\;.
\]
Since $V$ and $I$ are open, and by assuming that $n$ is sufficiently
large,
\[
\xi_{V}(U)\ge\xi_{V}(F)-\varepsilon\ge1-2\varepsilon\:.
\]

For $u\in\mathcal{U}$ we have $L(A_{u})\in I$. Since $W^{\perp}\in I$
and $\diam(I)<\delta$ it follows (assuming $\delta<1/20$, say) that
$d(W,L(A_{u}))>1/100$. Hence by Lemma \ref{lem:projecting-cylinders-away-from-L-perp}
and Lemma \ref{lem:uni conv of ent},
\begin{equation}
\frac{1}{m}H\left(\varphi_{u}\mu,\mathcal{D}_{n+m}^{W\oplus W^{\perp}}\right)\geq\frac{1}{m}H\left(\pi_{W^{\perp}}\varphi_{u}\mu,\mathcal{D}_{n+m}\right)\ge\beta-O(\varepsilon)\:.\label{eq:ent big for u in U}
\end{equation}
Since $\left\Vert A_{u}\right\Vert =2^{-n+O(1)}$ we have $\diam\supp\varphi_{u}\mu=2^{-n+O(1)}$,
so $\frac{1}{m}H(\varphi_{u}\mu,\mathcal{D}_{n})=O(\frac{1}{m})$,
and the last equation implies 
\[
\frac{1}{m}H\left(\varphi_{u}\mu,\mathcal{D}_{n+m}^{W\oplus W^{\perp}}|\mathcal{D}_{n}^{W\oplus W^{\perp}}\right)\geq\beta-O(\varepsilon)\;.
\]
Now assume that $\delta<2^{-m}$. From $L(A_{u})\in I$ it follows
$d_{\PR}(L(A_{u}),W^{\perp})<2^{-m}$. Also, $\alpha_{1}(A_{u})=2^{-n+O(1)}$
and $\alpha_{1}(A_{u})/\alpha_{2}(A_{u})>2^{m}$, hence $\alpha_{2}(A_{u})<2^{-(n+m)+O(1)}$.
This implies that $\varphi_{u}\mu$ is contained in the $2^{-(n+m)+O(1)}$-neighborhood
of a translate of $W^{\perp}$. Hence 
\[
\diam(\supp\pi_{W}\varphi_{u}\mu)=O(2^{-(n+m)}),
\]
and so
\[
\frac{1}{m}H(\varphi_{u}\mu,\pi_{W}^{-1}\mathcal{D}_{n+m})=\frac{1}{m}H(\pi_{W}\varphi_{u}\mu,\mathcal{D}_{n+m})=O(\frac{1}{m})\;.
\]
Combined with the previous bound, it follows that for every $u\in\mathcal{U}$,
\[
\frac{1}{m}H\left(\varphi_{u}\mu,\mathcal{D}_{n+m}^{W\oplus W^{\perp}}|\mathcal{D}_{n}^{W\oplus W^{\perp}}\lor\pi_{W}^{-1}\mathcal{D}_{n+m}\right)=\beta-O(\varepsilon)\;.
\]
Since $\mu_{U}$ is a convex combination of measures $\varphi_{u}\mu$
over $u\in\mathcal{U}$, concavity of entropy implies
\[
\frac{1}{m}H\left(\mu_{U},\mathcal{D}_{n+m}^{W\oplus W^{\perp}}|\mathcal{D}_{n}^{W\oplus W^{\perp}}\lor\pi_{W}^{-1}\mathcal{D}_{n+m}\right)=\beta-O(\varepsilon)\;.
\]
The argument is now completed as in the previous lemma, by showing
that $\mu_{U},\mu_{F}$ are close in total variation.
\end{proof}

\subsection{\label{subsec:Entropy-of-slices}Entropy of slices}

Denote the Borel $\sigma$-algebra by $\mathcal{B}$. For $\nu\in\mathcal{P}(\mathbb{R}^{2})$
and a $\sigma$-algebra $\mathcal{A}\subset\mathcal{B}$ let $\{\nu_{x}^{\mathcal{A}}\}_{x\in\mathbb{R}^{2}}$
be the disintegration of $\nu$ with respect to $\mathcal{A}$. For
$W\in\PR$ we write $\mathcal{B}_{W}\subseteq\mathcal{B}$ for the
$\sigma$-algebra of $W$-saturated sets (that is, sets $E$ such
that if $x\in E$ then $W+x\subseteq E$), and write $\{\nu_{x}^{W}\}_{x\in\mathbb{R}^{2}}$
in place of $\{\nu_{x}^{\mathcal{B}_{W}}\}_{x\in\mathbb{R}^{2}}$,
the family of conditional measures on translates of $W$. The following
is standard equivariance of measure disintegration, we omit the proof:
\begin{lem}
\label{lem:image of slices}Let $\varphi\in A_{2,2}$, $W\in\PR$,
and $\nu\in\mathcal{P}(\mathbb{R}^{2})$ be given. Then for $\nu$-a.e.
$x\in\mathbb{R}^{2}$,
\[
\left(\varphi\nu\right)_{\varphi x}^{W}=\varphi\left(\nu_{x}^{A_{\varphi}^{-1}W}\right)\:.
\]
Equivalently,
\[
\left(\varphi\nu\right)_{\varphi x}^{W^{\perp}}=\varphi\left(\nu_{x}^{(A_{\varphi}^{*}W)^{\perp}}\right)\:.
\]
\end{lem}

\begin{rem}
The last form is the one we will use. Usually $W$ will be a subspace
onto which we are projecting $\mu$, and since $\pi_{W}^{-1}\mathcal{B}$
consists of lines perpendicular to $W$ the disintegration of $\mu$
over this map is then given by $\{\mu_{x}^{W^{\perp}}\}$.
\end{rem}

Recall the definition of $\Upsilon_{n}$ and $\textbf{\ensuremath{\mathbf{K}}}(n)$
from Section \ref{sub:random-cylinder-measures} and that we write
$\gamma$ for $\alpha-\beta$. As mentioned above, from Theorem \ref{thm:LY}
it follows that
\begin{equation}
\dim\mu_{x}^{W^{\perp}}=\gamma\text{ for \ensuremath{\eta^{*}}-a.e \ensuremath{W} and \ensuremath{\mu}-a.e \ensuremath{x}.}\label{eq:exact dim of slices}
\end{equation}
\begin{lem}
\label{lem:lb for slice rand U}For $\varepsilon>0$, $m\ge M(\varepsilon)\ge1$,
and $n\ge1$,
\begin{equation}
\int\mathbb{E}\left(\varphi_{\textbf{\ensuremath{\mathbf{K}}}(n)}\mu\left\{ x\::\:\frac{1}{m}H\left(\left(\varphi_{\textbf{\ensuremath{\mathbf{K}}}(n)}\mu\right)_{x}^{W^{\perp}},\mathcal{D}_{n+m}\right)>\gamma-\varepsilon\right\} \right)\:d\eta^{*}(W)>1-\varepsilon\:.\label{eq:lb for slice rand U}
\end{equation}
\end{lem}

\begin{proof}
Let $\varepsilon>0$, let $m\ge1$ be large with respect to $\varepsilon$,
and let $n\ge1$. By Lemma \ref{lem:image of slices}, for each $W$
and $\mu$-a.e. $x$,
\[
(\varphi_{\textbf{\ensuremath{\mathbf{K}}}(n)}\mu)_{\varphi_{\textbf{\ensuremath{\mathbf{K}}}(n)}x}^{W^{\perp}}=\varphi_{\textbf{\ensuremath{\mathbf{K}}}(i)}\left(\mu_{x}^{(A_{\mathbf{K}(i)}^{*}W)^{\perp}}\right)\;.
\]
For $w\in\Upsilon_{n}$, the map $\varphi_{w}^{-1}$ expands by at
most $O(2^{n})$ in every direction. Therefore there exists constants
$C,C'>0$, independent of $m$ and $n$, such that, for every $w\in\Upsilon_{n}$,
each atom of $\varphi_{w}^{-1}(\mathcal{D}_{n+m})$ is of diameter
at most $C\cdot2^{-m}$, so it intersects at most $C'$ atoms of $\mathcal{D}_{m}$.
Hence
\begin{align*}
\frac{1}{m}H\left((\varphi_{\textbf{\ensuremath{\mathbf{K}}}(n)}\mu)_{\varphi_{\textbf{\ensuremath{\mathbf{K}}}(n)}x}^{W^{\perp}},\mathcal{D}_{m+n}\right) & =\frac{1}{m}H\left(\varphi_{\textbf{\ensuremath{\mathbf{K}}}(n)}\left(\mu_{x}^{(A_{\textbf{\ensuremath{\mathbf{K}}}(n)}^{*}W)^{\perp}}\right),\mathcal{D}_{m+n}\right)\\
 & =\frac{1}{m}H\left(\mu_{x}^{(A_{\textbf{\ensuremath{\mathbf{K}}}(n)}^{*}W)^{\perp}},\varphi_{\textbf{\ensuremath{\mathbf{K}}}(n)}^{-1}\mathcal{D}_{m+n}\right)\\
 & \geq\frac{1}{m}H\left(\mu_{x}^{(A_{\textbf{\ensuremath{\mathbf{K}}}(n)}^{*}W)^{\perp}},\mathcal{D}_{m}\right)+O(\frac{1}{m})\;.
\end{align*}
Assuming that $m$ is large enough with respect to $\varepsilon$,
the left hand side of (\ref{eq:lb for slice rand U}) is at least
as large as
\[
\int\mathbb{E}\left(\mu\left\{ x\::\:\frac{1}{m}H\left(\mu_{x}^{(A_{\textbf{\ensuremath{\mathbf{K}}}(n)}^{*}W)^{\perp}},\mathcal{D}_{m}\right)>\gamma-\frac{\varepsilon}{2}\right\} \right)\:d\eta^{*}(W)\:.
\]
Expanding the expectation as a weighted sum over $w\in\Upsilon(n)$
and using $\eta^{*}=\sum_{w\in\Upsilon(n)}p_{w}\cdot A_{w}^{*}\eta^{*}$,
we get
\[
\ge\int\mu\left\{ x\::\:\frac{1}{m}H\left(\mu_{x}^{W^{\perp}},\mathcal{D}_{m}\right)>\gamma-\frac{\varepsilon}{2}\right\} \:d\eta^{*}(W)\;.
\]
The lemma now follows from (\ref{eq:exact dim of slices}).
\end{proof}
\begin{lem}
\label{lem:ent of slices of comp}For every $\varepsilon>0$, $m\ge M(\varepsilon)\ge1$,
and $n\ge1$,
\[
\int\mathbb{P}_{i=n}\left(\frac{1}{m}H(\mu_{x,i},\mathcal{D}_{i+m}\mid\pi_{W}^{-1}(\mathcal{B}))>\gamma-\varepsilon\right)\:d\eta^{*}(W)>1-\varepsilon\:.
\]
\end{lem}

\begin{proof}
Let $\varepsilon>0$ be small, let $\delta>0$ be small with respect
to $\varepsilon$, let $k\ge1$ be large with respect to $\delta$,
let $m\ge1$ be large with respect to $k$, and let $n\ge1$. The
measure $L\xi$ is continuous, hence we can assume that
\begin{equation}
\mathbb{P}\left(L(A_{\textbf{\ensuremath{\mathbf{K}}}(n+k)})\in B(W,\delta)\right)<\varepsilon\text{ for each }W\in\PR\:.\label{eq:small measure near a ball}
\end{equation}
It is not hard to see that for each $W\in\PR$, $w\in\Upsilon_{n+k}$
with $L(A_{w})\notin B(W^{\perp},\delta)$, and $\varphi_{w}\mu$-a.e.
$x\in\mathbb{R}^{2}$,
\begin{equation}
\diam(\supp\left(\varphi_{w}\mu\right)_{x}^{W^{\perp}})=O_{\delta}(2^{-n-k})\:.\label{eq:ub on diam of supp of slices}
\end{equation}

Let $Z$ be the set of all $W\in\PR$ such that,
\begin{equation}
\mathbb{E}\left(\varphi_{\textbf{\ensuremath{\mathbf{K}}}(n+k)}\mu\left\{ x\::\:\frac{1}{m}H\left(\left(\varphi_{\textbf{\ensuremath{\mathbf{K}}}(n+k)}\mu\right)_{x}^{W^{\perp}},\mathcal{D}_{n+k+m}\right)>\gamma-\varepsilon\right\} \right)>1-\varepsilon\:.\label{eq:big ent of slices}
\end{equation}
By Lemma \ref{lem:lb for slice rand U} we may assume that $\eta^{*}(Z)>1-\varepsilon$.
Fix $W\in Z$ for the remainder of the proof.

Define $\Gamma\in\mathcal{P}(\Upsilon_{n+k}\times\mathbb{R}^{2})$
by
\[
\Gamma=\sum_{w\in\Upsilon_{n+k}}p_{w}\cdot\delta_{\{w\}}\times\varphi_{w}\mu\:.
\]
Let $F$ be the set of all $(w,x)\in\Upsilon_{n+k}\times\mathbb{R}^{2}$
such that (\ref{eq:ub on diam of supp of slices}) holds and
\begin{equation}
\frac{1}{m}H\left(\left(\varphi_{w}\mu\right)_{x}^{W^{\perp}},\mathcal{D}_{n+m}\right)>\gamma-\varepsilon\:.\label{eq:ent > beta - epsilon}
\end{equation}
By (\ref{eq:small measure near a ball}) and (\ref{eq:big ent of slices}),
by recalling that $m$ is large with respect to $k$, and by replacing
$\varepsilon$ with a larger quantity which is still small without
changing the notation, we may assume that $\Gamma(F)>1-\varepsilon$.

By Lemma \ref{lem:dyadic-cells-with-small-boundaries},
\[
\mu\left(\cup_{D\in\mathcal{D}_{n}}(\partial D)^{(2^{-n}\delta)}\right)<\varepsilon\:.
\]
Since $k$ is large with respect to $\delta$, we may assume that
if $\nu\in\mathcal{P}(\mathbb{R}^{2})$ is such that $\diam(\supp(\nu))=O_{\delta}(2^{-n-k})$
and
\[
\#\{D\in\mathcal{D}_{n}\::\:\supp(\nu)\cap D\ne\emptyset\}>1,
\]
then $\supp(\nu)\subset\cup_{D\in\mathcal{D}_{n}}(\partial D)^{(2^{-n}\delta)}$.
Also, it is possible to write $\mu$ as
\begin{equation}
\mu=\mathbb{E}\left(\varphi_{\textbf{\ensuremath{\mathbf{K}}}(n+k)}\mu\right)=\mathbb{E}\left(\int\left(\varphi_{\textbf{\ensuremath{\mathbf{K}}}(n+k)}\mu\right)_{x}^{W^{\perp}}\:d\varphi_{\textbf{\ensuremath{\mathbf{K}}}(n+k)}\mu(x)\right)\:.\label{eq:decomp of mu to slices}
\end{equation}
By these facts, since (\ref{eq:ub on diam of supp of slices}) holds
for $(w,x)\in F$, and by replacing $\varepsilon$ with a larger quantity
without changing the notation, we may assume that for each $(w,x)\in F$
\begin{equation}
\exists\:D\in\mathcal{D}_{n}\text{ with }\supp\:\left(\varphi_{w}\mu\right)_{x}^{W^{\perp}}\subset D\;,\label{eq:cont in atom}
\end{equation}
while still having $\Gamma(F)>1-\varepsilon$. 

Let $E$ be set of all $x\in\mathbb{R}^{2}$ for which there exist
a probability space $(\Omega_{x},\theta_{x})$, $\{\nu_{x,\omega}\}_{\omega\in\Omega_{x}}\subset\mathcal{P}(\mathbb{R}^{2})$,
$0\le\rho_{x}<\varepsilon$, and $\nu'_{x}\in\mathcal{P}(\mathbb{R}^{2})$,
such that
\begin{itemize}
\item $\mu_{x,n}=(1-\rho_{x})\int\nu_{x,\omega}\:d\theta_{x}(\omega)+\rho_{x}\nu'_{x};$
\item $\frac{1}{m}H\left(\nu_{x,\omega},\mathcal{D}_{n+m}\right)>\gamma-\varepsilon$
for $\omega\in\Omega_{x}$;
\item $\nu_{x,\omega}$ is supported on a single atom of $\pi_{W}^{-1}(\mathcal{B})$
for $\omega\in\Omega_{x}$.
\end{itemize}
From the decomposition $\mu=\mathbb{E}_{i=n}\left(\mu_{x,i}\right)$,
by (\ref{eq:decomp of mu to slices}), since (\ref{eq:ent > beta - epsilon})
and (\ref{eq:cont in atom}) hold for $(w,x)\in F$, since $\Gamma(F)>1-\varepsilon$,
and by replacing $\varepsilon$ with a larger quantity without changing
the notation, we may assume that $\mu(E)>1-\varepsilon$.

Let $x\in E$, then by concavity of conditional entropy,
\[
\frac{1}{m}H(\mu_{x,n},\mathcal{D}_{n+m}\mid\pi_{W}^{-1}(\mathcal{B}))\ge(1-\varepsilon)\int\frac{1}{m}H(\nu_{x,\omega},\mathcal{D}_{n+m}\mid\pi_{W}^{-1}(\mathcal{B}))\:d\theta_{x}(\omega)\:.
\]
For $\omega\in\Omega_{x}$,
\[
\frac{1}{m}H\left(\nu_{x,\omega},\mathcal{D}_{n+m}\right)>\gamma-\varepsilon
\]
and $\nu_{x,\omega}$ is supported on a single atom of $\pi_{W}^{-1}(\mathcal{B})$.
Hence,
\[
\frac{1}{m}H(\mu_{x,n},\mathcal{D}_{n+m}\mid\pi_{W}^{-1}(\mathcal{B}))\ge(1-\varepsilon)(\gamma-\varepsilon)\:.
\]
Since $\eta^{*}(Z)>1-\varepsilon$ and $\mu(E)>1-\varepsilon$ this
completes the proof of the lemma.
\end{proof}

\subsection{\label{subsec:Uniform-entropy-dimension}Uniform entropy dimension}

In this section we show that typical components of $\mu$ have normalized
entropy close to $\alpha=\dim\mu$, a property referred to in \cite{Ho}
as uniform entropy dimension. This will be used later on to conclude
that typical components cannot look like uniform measure on a dyadic
cell, which we use to rule out one of the alternatives that one gets
from the entropy inverse theorem in $\mathbb{R}^{2}$ (See Section
\ref{subsec:Entropy-growth-in-Rd}). 
\begin{defn}
\label{def:uniform-entropy-dim}We say that $\nu\in\mathcal{P}(\mathbb{R}^{d})$
has uniform entropy dimension $t$ if for every $\varepsilon>0$,
$m\ge M(\varepsilon)\ge1$, and $n\ge N(\varepsilon,m)\ge1$,
\[
\mathbb{P}_{0\le i\le n}\left(|H_{m}(\mu^{x,i})-t|<\varepsilon\right)>1-\varepsilon\:.
\]
 
\end{defn}

This property implies a uniformity among the components of the measure.
If $\nu$ has uniform entropy dimension $t$, then it follows from
Equation (\ref{lem:multiscale-entropy-formula}) that its entropy
dimension is well defined and $\edim\nu=t$. The converse is false,
i.e. the existence of entropy dimension does not imply existence of
uniform entropy dimension. 
\begin{prop}
\label{prop:uniform ent dim}$\mu$ has uniform entropy dimension
$\alpha$.
\end{prop}

\begin{proof}
Let $\varepsilon>0$, let $m\ge1$ be large with respect to $\varepsilon$,
and let $n\ge1$ be large with respect to $m$. Recall that for $W\in\PR$
and $k\ge1$,
\[
\mathcal{D}_{k}^{W\oplus W^{\perp}}=\left(\pi_{W}^{-1}\mathcal{D}_{k}\right)\vee\left(\pi_{W^{\perp}}^{-1}\mathcal{D}_{k}\right)\;,
\]
and that $\mathcal{D}_{k}$ and $\mathcal{D}_{k}^{W\oplus W^{\perp}}$
are commensurable partitions. Write
\begin{align*}
\delta & =\int\mathbb{P}_{0\le i\le n}\left(\frac{1}{m}H(\mu_{x,i},\mathcal{D}_{i+m}^{W\oplus W^{\perp}})\le\alpha-\varepsilon\right)\:d\eta^{*}(W)\\
\delta_{1} & =\int\mathbb{P}_{0\le i\le n}\left(\frac{1}{m}H(\mu_{x,i},\pi_{W}^{-1}\mathcal{D}_{i+m})\le\beta-\frac{\varepsilon}{2}\right)\:d\eta^{*}(W)\;,
\end{align*}
and
\[
\delta_{2}=\int\mathbb{P}_{0\le i\le n}\left(\frac{1}{m}H(\mu_{x,i},\mathcal{D}_{i+m}^{W\oplus W^{\perp}}\mid\pi_{W}^{-1}\mathcal{D}_{i+m})\le\gamma-\frac{\varepsilon}{2}\right)\:d\eta^{*}(W)\:.
\]
Since for each $W\in\PR$, $0\le i\le n$, and $x\in\mathbb{R}^{2}$,
\[
H(\mu_{x,i},\mathcal{D}_{i+m}^{W\oplus W^{\perp}})=H(\mu_{x,i},\pi_{W}^{-1}\mathcal{D}_{i+m})+H(\mu_{x,i},\mathcal{D}_{i+m}^{W\oplus W^{\perp}}\mid\pi_{W}^{-1}\mathcal{D}_{i+m})\;,
\]
any component that belongs to the event defining $\delta$ must also
belong to one of the events defining $\delta_{1}$ or $\delta_{2}$,
hence $\delta\le\delta_{1}+\delta_{2}$.

By Lemma \ref{lem:ent of proj of comp mostly large } we can assume
that $\delta_{1}<\varepsilon/2$. By Lemma \ref{lem:ent of slices of comp}
we can assume that $\delta_{2}<\varepsilon/2$. Hence $\delta\le\delta_{1}+\delta_{2}<\varepsilon$,
and so
\[
\int\mathbb{P}_{0\le i\le n}\left(\frac{1}{m}H(\mu_{x,i},\mathcal{D}_{i+m}^{W\oplus W^{\perp}})>\alpha-\varepsilon\right)\;d\eta^{*}(W)>1-\varepsilon\:.
\]
Since $\mathcal{D}_{i+m}^{W\oplus W^{\perp}}$ and $\mathcal{D}_{i+m}$
are commensurable, the entropy above depends on $W$ only up to an
additive $O(1)$ constant, so we can eliminate the outer integral
by introducing an additive $O(1/m)$ error. Therefore, assuming $m$
is large enough relative to $\varepsilon$, 
\begin{equation}
\mathbb{P}_{0\le i\le n}\left(\frac{1}{m}H(\mu_{x,i},\mathcal{D}_{i+m})>\alpha-2\varepsilon\right)>1-\varepsilon\:.\label{eq:large ent of comp}
\end{equation}
By Lemma \ref{lem:multiscale-entropy-formula} and since we can assume
that $\frac{m}{n}<\varepsilon$,
\[
\alpha=\mathbb{E}_{0\le i\le n}\left(\frac{1}{m}H(\mu_{x,i},\mathcal{D}_{i+m})\right)+O(\varepsilon)\:.
\]
This together with (\ref{eq:large ent of comp}) completes the proof
of the proposition (by starting from a smaller $\varepsilon$).
\end{proof}

\section{\label{sec:the function L}The function $L$ factors through $\Pi$}

In this section we assume that $\Phi$ is non-conformal and totally
irreducible. We also assume that $\dim\mu<2$. Exponential separation
is not needed.

\subsection{\label{subsec:Bourgain's-projection-theorem}Bourgain's projection
theorem (entropy variant)}

In the next sections we prove a result which requires, in its most
general form, the following theorem, whose proof will appear in more
quantitative form separately. It is an entropy version of Bourgain's
projection theorem, in which $\bdim$ denotes box (Minkowski) dimension
(see e.g. \cite{Ma}) and uniform entropy dimension is understood
in the sense of Definition \ref{def:uniform-entropy-dim}. 
\begin{thm}
\label{thm:bourgain-projection}For every $\delta>0$ there exists
a $\tau=\tau(\delta)>0$ such that the following holds. Let $\nu\in\mathcal{P}(\mathbb{R}^{2})$
have uniform entropy dimension $t\in(\delta,2-\delta)$, and let $E\subseteq\PR$
satisfy $\dim_{B}E>\delta$. Then for every $n>N(\delta,\nu,E)$ there
exists $W\in E$ (depending perhaps on $n$) such that 
\[
\frac{1}{n}H(\pi_{W}\nu,\mathcal{D}_{n})>\frac{1}{2}\cdot\frac{1}{n}H(\nu,\mathcal{D}_{n})+\tau\;.
\]
\end{thm}

\begin{cor}
\label{cor:bourgain-projection-of-SA-measures}If $\mu$ is a self-affine
measure defined by a non-conformal, totally irreducible system, and
if $\dim\mu<2$, then there exists $\tau>0$ such that for all large
enough $n$, for all $W\in\PR$, 
\[
\frac{1}{n}H(\pi_{W}\mu,\mathcal{D}_{n})>\frac{1}{2}\dim\mu+\tau\;.
\]
\end{cor}

\begin{proof}
Since $\frac{1}{n}H(\mu,\mathcal{D}_{n})\rightarrow\dim\mu$ as $n\rightarrow\infty$,
and since $\dim\eta^{*}>0$, it follows that for every set $E\subseteq\PR$
of positive $\eta^{*}$-measure, for every $n$ large enough (depending
on $E$), there are $W\in E$ such that the inequality in the statement
above holds. This implies that for $\eta^{*}$-a.e. $W$ there exist
arbitrarily large $n$ for which the inequality holds. But for $\eta^{*}$-a.e.
$W$ we have $\frac{1}{n}H(\pi_{W}\mu,\mathcal{D}_{n})\rightarrow\beta$,
where $\beta\geq0$ is the $\eta^{*}$-a.s. constant dimension of
$\dim\pi_{W}\mu$ of $W$; therefore $\beta\geq\frac{1}{2}\dim\mu+\tau$.
The fact that one can take $n$ uniformly in $W\in\PR$ now follows
from Lemma \ref{lem:uni conv of ent} (at the cost of a slight loss
in $\tau$).
\end{proof}
\begin{rem*}
In the case that exponential separation holds, the conclusion of the
last corollary follows easily from Theorem \ref{thm:BHR-projections}
since when $\dim\mu<2$ we certainly have 
\[
\dim\pi_{W}\mu=\min\{1,\dim\mu\}>\frac{1}{2}\dim\mu\qquad\eta^{*}\text{-a.e. }W\;.
\]
Thus, Corollary \ref{cor:bourgain-projection-of-SA-measures} will
be used only when exponential separation is not assumed.
\end{rem*}

\subsection{\label{subsec:Transversality-of-cylinders}Transversality of cylinders}
\begin{prop}
\label{prop:singular pulbacks}Let $\mu$ be a self-affine measure
defined by a non-conformal and totally irreducible system, and suppose
that $\dim\mu<2$. Then for every $\rho>0$ there exists $\delta=\delta(\mu,\rho)>0$
such that the following holds.

Let $I,J\subset\PR$ be such that $I,J$ are open intervals, $L\xi(I),L\xi(J)>0$,
$d_{\PR}(I,J)>\rho$, and $\diam(I)<\delta$. Then the measures $\mu_{L^{-1}(I)}$
and $\mu_{L^{-1}(J)}$ are singular.
\end{prop}

\begin{proof}
We first give the proof under the simplifying assumptions (which are
the ones used in the proof of Theorem \ref{thm:main}) that exponential
separation holds and $\dim\mu\geq1$. In this case, $\dim\pi_{W}\mu=1$
for all $W\in\PR$.

Assume by contradiction that there exists $\rho>0$ for which the
proposition fails. We will show that this leads to a contradiction
with the assumption $\dim\mu<2$. Let $\epsilon>0$, $M_{1}=M_{1}(\varepsilon,\rho)$
as in Lemma \ref{lem:ent of proj of restrictions}, $M_{2}=M_{2}(\varepsilon)$
as in Lemma \ref{lem:ent of sections of restrictions}, $m\ge\max\{M_{1},M_{2}\}$,
and $\delta=\delta(\varepsilon,m)$ as in Lemma \ref{lem:ent of sections of restrictions}.
Since the proposition fails for $\rho$ there exist open intervals
$I,J\subset\PR$ such that $L\xi(I),L\xi(J)>0$, $d_{\PR}(I,J)>\rho$,
$\diam(I)<\delta$, and $\mu_{L^{-1}(I)},\mu_{L^{-1}(J)}$ are not
singular.

Since $\mu_{L^{-1}(I)},\mu_{L^{-1}(J)}$ are not singular, there exists
a Borel set $E\subset\mathbb{R}^{2}$ with $\mu_{L^{-1}(I)}(E)>0$
on which the measures are equivalent, that is $(\mu_{L^{-1}(I)})_{E}\sim(\mu_{L^{-1}(J)})_{E}$.
Therefore there exists a Borel set $B\subset E$ with,
\[
\mu_{L^{-1}(I)}(B),\mu_{L^{-1}(J)}(B)>0\text{ and }d_{TV}((\mu_{L^{-1}(I)})_{B},(\mu_{L^{-1}(J)})_{B})<\varepsilon\:.
\]
(we can take $B\subseteq E$ to be any Borel set of positive $(\mu_{L^{-1}(J)})_{E}$-measure
on which the Radon-Nikodym derivative $f=d\,(\mu_{L^{-1}(I)})_{E}/d\,(\mu_{L^{-1}(J)})_{E}$
is positive and sufficiently concentrated around one value, e.g. if
$f(B)\subseteq(c-\varepsilon',c+\varepsilon')$ for some $c>0$ and
$\varepsilon'>0$ that is small relative to $c$ and $\varepsilon$).
Set $\mu^{I}=\mu_{\Pi^{-1}(B)\cap L^{-1}(I)}$ and $\mu^{J}=\mu_{\Pi^{-1}(B)\cap L^{-1}(J)}$,
then $\mu^{I}=(\mu_{L^{-1}(I)})_{B}$ and $\mu^{J}=(\mu_{L^{-1}(J)})_{B}$,
and so $d_{TV}(\mu^{I},\mu^{J})<\varepsilon$.

Fix $W\in\PR$ with $W^{\perp}\in I$, let $N_{1}=N_{1}(\varepsilon,\rho,m,\Pi^{-1}(B),J,W)$
as in Lemma \ref{lem:ent of proj of restrictions}, $N_{2}=N_{2}(\varepsilon,m,\delta,\Pi^{-1}(B),I,W)$
as in Lemma \ref{lem:ent of sections of restrictions}, and $N\ge\max\{N_{1},N_{2}\}$.
By $d_{\PR}(W^{\perp},J)\ge\rho$ and our choices of parameters,
\begin{equation}
\frac{1}{m}H\left(\mu^{J},\pi_{W}^{-1}\mathcal{D}_{n+m}\mid\mathcal{D}_{n}^{W\oplus W^{\perp}}\right)\ge1-\varepsilon\text{ for \ensuremath{n\ge N}.}\label{eq:ent of proj of mu^J}
\end{equation}
Similarly, since $W^{\perp}\in I$,
\begin{equation}
\frac{1}{m}H\left(\mu^{I},\mathcal{D}_{n+m}^{W\oplus W^{\perp}}\mid\mathcal{D}_{n}^{W\oplus W^{\perp}}\vee\pi_{W}^{-1}\mathcal{D}_{n+m}\right)\ge1-\varepsilon\text{ for \ensuremath{n\ge N}.}\label{eq:ent of sections of mu^I}
\end{equation}

By (\ref{eq:ent of proj of mu^J}), (\ref{eq:ent of sections of mu^I}),
$d_{TV}(\mu^{I},\mu^{J})<\varepsilon$, and Equation (\ref{eq:scale-n-entropy-continuity-in-TV})
(see also note after it), it follows that for $n\ge N$ with $N$
sufficiently large,
\begin{eqnarray}
\frac{1}{m}H\left(\mu^{I},\mathcal{D}_{n+m}^{W\oplus W^{\perp}}\mid\mathcal{D}_{n}^{W\oplus W^{\perp}}\right) & = & \frac{1}{m}H\left(\mu^{I},\pi_{W}^{-1}\mathcal{D}_{n+m}\mid\mathcal{D}_{n}^{W\oplus W^{\perp}}\right)\nonumber \\
 & + & \frac{1}{m}H\left(\mu^{I},\mathcal{D}_{n+m}^{W\oplus W^{\perp}}\mid\mathcal{D}_{n}^{W\oplus W^{\perp}}\vee\pi_{W}^{-1}\mathcal{D}_{n+m}\right)\nonumber \\
 & \ge & \frac{1}{m}H\left(\mu^{J},\pi_{W}^{-1}\mathcal{D}_{n+m}\mid\mathcal{D}_{n}^{W\oplus W^{\perp}}\right)+1-O(\epsilon)\nonumber \\
 & \ge & 2-O(\epsilon)\:.\label{eq:ent close to 2}
\end{eqnarray}

Since $\mu^{I}\ll\mu$ and $\mu$ has exact dimension $\alpha$, it
follows that $\mu^{I}$ also has exact dimension $\alpha$. From this
and Lemma \ref{lem:multiscale-entropy-formula} it follows that for
$k$ large enough,
\[
\alpha\ge\frac{1}{k}H\left(\mu^{I},\mathcal{D}_{k}^{W}\right)-\varepsilon\ge\mathbb{E}_{0\le n\le k}\left(\frac{1}{m}H\left(\mu^{I},\mathcal{D}_{n+m}^{W\oplus W^{\perp}}\mid\mathcal{D}_{n}^{W\oplus W^{\perp}}\right)\right)-O(\varepsilon)\:.
\]
This together with (\ref{eq:ent close to 2}) shows that $\alpha\ge2-O(\epsilon)$.
Since this holds for every $\epsilon>0$ it implies a contradiction
with $\alpha<2$, which is what we wanted.

We now explain how to modify the proof for the general case, i.e.
without exponential separation. As above, assume by contradiction
that there exists $\rho>0$ for which the proposition fails. Let $\tau>0$
be as in Corollary \ref{cor:bourgain-projection-of-SA-measures},
so that $\dim\pi_{W}\mu>\frac{\alpha}{2}+\tau$ for all $W\in\PR$.
Let $\epsilon>0$ and carry out the argument above. Then on the right-hand
side of (\ref{eq:ent of proj of mu^J}) and (\ref{eq:ent of sections of mu^I})
we will have $\frac{\alpha}{2}+\tau-\varepsilon$; proceeding from
there we eventually get $\alpha\geq\alpha+2\tau-O(\varepsilon)$.
This holds for every $\epsilon>0$ and so yields the required contradiction.
\end{proof}

\subsection{\label{subsec:L-factors-through-Pi}$L$ factors through $\Pi$}
\begin{prop}
\label{prop:const directions}Let $\mu$ be a self-affine measure
defined by non-conformal and totally irreducible system, and suppose
that $\dim\mu<2$. Let $\xi=\int\xi_{x}d\mu(x)$ denote the decomposition
of $\xi$ with respect to the partition $\{\Pi^{-1}(x)\}_{x\in X}$.
Then for $\mu$-a.e. $x$, the function $L|_{\Pi^{-1}(x)}$ is $\xi_{x}$-a.s.
constant.
\end{prop}

\begin{rem}
This implies that there is a Borel function $\widehat{L}:X\rightarrow\PR$,
defined $\mu$-a.e., such that $\widehat{L}(\Pi\omega)=L(\omega)$
$\xi$-a.s. We shall write $L$ instead of $\widehat{L}$ from now
on, which one is intended will be clear from the context.
\end{rem}

\begin{proof}
For $\omega\in\Lambda^{\mathbb{N}}$ let $\xi_{\omega}=\xi_{\Pi\omega}$,
which is defined $\xi$-a.e. It suffices to show that for $\xi$-a.e.
$\omega\in\Lambda^{\mathbb{N}}$ the measure $L\xi_{\omega}$ is a
mass point. It follows by Proposition \ref{prop:singular pulbacks}
that there exist sequences $\{I_{k}\}_{k=1}^{\infty}$ and $\{J_{k}\}_{k=1}^{\infty}$
such that,
\begin{enumerate}
\item $I_{k},J_{k}\subset\PR$ are open intervals with $L\xi(I_{k}),L\xi(J_{k})>0$
for $k\ge1$;
\item For every distinct $\overline{x},\overline{y}\in\supp(L\xi)$ there
exists $k\ge1$ with $\overline{x}\in I_{k}$ and $\overline{y}\in J_{k}$;
\item $\mu_{L^{-1}(I_{k})}$ and $\mu_{L^{-1}(J_{k})}$ are singular for
$k\ge1$.
\end{enumerate}
For each $k\ge1$ there exists a Borel set $E_{k}\subset\mathbb{R}^{2}$
with $\mu_{L^{-1}(I_{k})}(E_{k})=0$ and $\mu_{L^{-1}(J_{k})}(E_{k}^{c})=0$.
It holds that,
\begin{align*}
0 & =\xi(L^{-1}(I_{k}))\cdot\mu_{L^{-1}(I_{k})}(E_{k})\\
 & =\xi(L^{-1}(I_{k})\cap\Pi^{-1}(E_{k}))\\
 & =\int_{\Pi^{-1}(E_{k})}\xi_{\omega}(L^{-1}(I_{k}))\:d\xi(\omega)\;,
\end{align*}
and similarly
\[
\int_{\Pi^{-1}(E_{k}^{c})}\xi_{\omega}(L^{-1}(J_{k}))\:d\xi(\omega)=0\:.
\]
It follows that for $\xi$-a.e. $\omega\in\Lambda^{\mathbb{N}}$,
for each $k\ge1$
\begin{equation}
\xi_{\omega}(L^{-1}(I_{k}))=0\text{ or }\xi_{\omega}(L^{-1}(J_{k}))=0\:.\label{eq:zero measure of pulbacks}
\end{equation}
Additionally, it is clear that for $\xi$-a.e. $\omega\in\Lambda^{\mathbb{N}}$
\begin{equation}
\supp(L\xi_{\omega})\subset\supp(L\xi)\:.\label{eq:containment of supports}
\end{equation}
Fix $\omega\in\Lambda^{\mathbb{N}}$ which satisfies (\ref{eq:zero measure of pulbacks})
and (\ref{eq:containment of supports}). Assume by contradiction that
$L\xi_{\omega}$ is not a mass point. Then there exist distinct
\[
\overline{x},\overline{y}\in\supp(L\xi_{\omega})\subset\supp(L\xi)\;,
\]
and so there exists $k\ge1$ with $\overline{x}\in I_{k}$ and $\overline{y}\in J_{k}$.
Since $\overline{x},\overline{y}\in\supp(L\xi_{\omega})$ and $I_{k},J_{k}$
are open,
\[
\xi_{\omega}(L^{-1}(I_{k}))>0\text{ and }\xi_{\omega}(L^{-1}(J_{k}))>0\;,
\]
which contradicts (\ref{eq:zero measure of pulbacks}). This shows
that $L\xi_{\omega}$ is a mass point, which completes the proof of
the proposition.
\end{proof}

\subsection{\label{subsec:Projection-of-components-revisited}Projection of components,
revisited}

We continue to assume non-conformality, total irreducibility, and
$\dim\mu<2$.

As we discussed in Section \ref{subsec:Projections-of-mu-cylinders},
with the assumptions above, most cylinders of $\mu$ project well
in most directions $W\in\PR$ at the scale of their long axis. In
fact, they project well in a direction $W$ precisely when $W^{\perp}$
isn't too close to the long axis of the cylinder; that is an obstruction
because in that case, at the scale of their long axis, the cylinder
projects to essentially a point mass on $W$. 

Recall that $\beta$ is the dimension of the projection of $\mu$
to $\eta^{*}$-typical subspaces. We saw in Section \ref{subsec:Projections-mu-components}
that for a fixed $W\in\PR$, with high probability, a random component
projects well to $W$ in the sense that its normalized entropy at
small scales is close to $\beta$. This was proved essentially by
covering dyadic cells with cylinders. We now want to get finer information
and identify, for most components, which directions are the exceptions.
This is made possible by the result of the previous section: $\mu_{x,n}$
will project well to all lines except those that are close to $L(x)^{\perp}$.
This is basically proved by applying Lusin's theorem to $L:X\rightarrow\PR$
to conclude~that for most small enough cells $\mathcal{D}_{n}(x)$,
the function $L(x)$ is almost constant on the cell. This means that
most cylinders that cover the cell project well to every line except
those that are close to $L(x)^{\perp}$. 

Recall the definition of $\Psi_{n}$ from Section \ref{sub:random-cylinder-measures},
and that for $\omega\in\Lambda^{\mathbb{N}}$ we write $\Psi_{n}(\omega)$
for the unique $w\in\Psi_{n}$ for which $\omega\in[w]$.
\begin{lem}
\label{lem:without comp of comp}For $\varepsilon>0$, $m\ge M(\varepsilon)\ge1$,
and $n\ge N(\varepsilon,m)\ge1$,
\[
\mathbb{P}_{i=n}\left(\underset{W\notin B(L(x),\varepsilon)}{\inf}\;\frac{1}{m}H(\pi_{W^{\perp}}\mu_{x,i},\mathcal{D}_{i+m})>\beta-\varepsilon\right)>1-\varepsilon\:.
\]
\end{lem}

\begin{proof}
Let $\varepsilon>0$, let $\rho>0$ be small with respect to $\varepsilon$,
let $k\ge1$ be large with respect to $\rho$, let $m\ge1$ be large
with respect to $k$, and let $n\ge1$ be large with respect to $m$.
By Lemma \ref{lem:dyadic-cells-with-small-boundaries}, for each $\delta>0$
there exists $\sigma>0$, which does not depend on $n$, such that
\[
\mu\left(\cup_{D\in\mathcal{D}_{n}}(\partial D)^{(2^{-n}\sigma)}\right)<\delta\:.
\]
From this and by assuming that $k$ is sufficiently large with respect
to $\rho$, it follows that $\mu(E)>1-\rho$, where $E$ is the set
of all $x\in\mathbb{R}^{2}$ for which there exist distinct $w_{x,1},...,w_{x,\ell_{x}}\in\Psi_{n+k}$,
$\theta_{x}\in\mathcal{P}(\mathbb{R}^{2})$, $c_{x}>0$ and $0\le c'_{x}<\rho$,
such that
\begin{equation}
\mu_{x,n}=c_{x}\sum_{j=1}^{\ell_{x}}p_{w_{x,j}}\cdot\varphi_{w_{x,j}}\mu+c'_{x}\theta_{x},\label{eq:conv comb of comp}
\end{equation}
(here $c_{x}=(1-c'_{x})/\sum_{j=1}^{\ell_{x}}p_{w_{x,j}}$ is a normalizing
constant). 

By the definition of $L$ and by assuming that $n$ is large enough,
\[
\xi\left\{ \omega\::\:d_{\PR}(L(A_{\Psi_{n+k}(\omega)}),L(\omega))<\rho\right\} >1-\rho^{4}\:.
\]
From this we get $\sum_{w\in\mathcal{W}}p_{w}>1-\rho^{2}$, where
$\mathcal{W}$ is the set of all $w\in\Psi_{n+k}$ with
\[
\xi_{[w]}\left\{ \omega\::\:d_{\PR}(L(A_{w}),L(\omega))<\rho\right\} >1-\rho\:.
\]
Now since $\Pi\xi_{[w]}=\varphi_{w}\mu$ and $L$ factors through
$\Pi$,
\[
\varphi_{w}\mu\left\{ x\::\:d_{\PR}(L(A_{w}),L(x))<\rho\right\} >1-\rho\text{ for }w\in\mathcal{W}\:.
\]
Hence by $\sum_{w\in\mathcal{W}}p_{w}>1-\rho^{2}$ we can also require,
\begin{equation}
\varphi_{w_{x,j}}\mu\left\{ y\::\:d_{\PR}(L(A_{w_{x,j}}),L(y))<\rho\right\} >1-\rho\text{ for }x\in E\text{ and }1\le j\le l_{x}\;,\label{eq:close to dir of cyl}
\end{equation}
and still have $\mu(E)>1-O(\rho)$ and $c'_{x}=O(\rho)$ for $x\in E$.

Since $L$ is Borel measurable and by Lusin's theorem, for every $\delta>0$
there exists a Borel set $F\subset\mathbb{R}^{2}$ such that $\mu(F)>1-\delta$
and $L|_{F}$ is uniformly continuous. From this, since $\supp(\varphi_{w_{x,j}}\mu)\subset\overline{\mathcal{D}_{n}(x)}$
for $x\in E$ and $1\le j\le l_{x}$, and by assuming that $n$ is
large enough, we may also require
\begin{equation}
\varphi_{w_{x,j}}\mu\left\{ y\::\:d_{\PR}(L(x),L(y))<\rho\right\} >1-\rho\text{ for }x\in E\text{ and }1\le j\le l_{x}\;,\label{eq:close to dir of x}
\end{equation}
and still have $\mu(E)>1-O(\rho)$ and $c'_{x}=O(\rho)$ for $x\in E$.

Since $\mu(E)>1-O(\rho)$ it suffices to show that 
\[
\frac{1}{m}H(\pi_{W^{\perp}}\mu_{x,n},\mathcal{D}_{n+m})>\beta-O(\rho)\text{ for all }x\in E\text{ and }W\notin B(L(x),\varepsilon)\:.
\]
Let $x\in E$, $W\notin B(L(x),\varepsilon)$, and $1\le j\le l_{x}$.
By (\ref{eq:close to dir of cyl}) and (\ref{eq:close to dir of x})
it follows that $d_{\PR}(L(A_{w_{x,j}}),L(x))<2\rho$, and so $W\notin B(L(A_{w_{x,j}}),\frac{\varepsilon}{2})$.
Now by Lemmas \ref{lem:projecting-cylinders-away-from-L-perp} and
\ref{lem:uni conv of ent}, and by assuming that $m$ is large enough
with respect to $k$ and $\epsilon$,
\[
\frac{1}{m}H\left(\pi_{W^{\perp}}\varphi_{w_{x,j}}\mu,\mathcal{D}_{n+m}\right)\ge\frac{1}{m}H\left(\pi_{W^{\perp}}\varphi_{w_{x,j}}\mu,\mathcal{D}_{n+k+m}\right)-\rho\ge\beta-O(\rho)\:.
\]
From this, the decomposition (\ref{eq:conv comb of comp}), the estimate
$c'_{x}=O(\rho)$, and the concavity of entropy, we get
\[
\frac{1}{m}H(\pi_{W^{\perp}}\mu_{x,n},\mathcal{D}_{n+m})>\beta-O(\rho)\;,
\]
which completes the proof of the lemma.
\end{proof}
We reformulate this as a statement which holds for components of components.
Recall from Section \ref{sub:q-adic-partitions} that $\lambda_{n}$
is the uniform measure on $\mathcal{N}_{n}=\{1,\ldots,n\}$.
\begin{lem}
\label{lem:proj of comp of comp}For $\varepsilon>0$, $m\ge M(\varepsilon)\ge1$,
$k\ge1$, and $n\ge N(\varepsilon,m,k)\ge1$ we have $\lambda_{n}\times\mu(F)>1-\varepsilon$,
where $F$ is the set of all $(i,x)\in\mathcal{N}_{n}\times\mathbb{R}^{2}$
such that
\begin{equation}
\mathbb{P}_{i\le j\le i+k}\left(\underset{W\notin B(L(x),\varepsilon)}{\inf}\;\frac{1}{m}H\left(\pi_{W^{\perp}}((\mu_{x,i})_{y,j}),\mathcal{D}_{j+m}\right)>\beta-\varepsilon\right)>1-\varepsilon\:.\label{eq:proj of comp of comp}
\end{equation}
\end{lem}

\begin{proof}
As noted above, since $L$ is Borel measurable and by Lusin's theorem,
for every $\varepsilon>0$ there exists a Borel set $E\subset\mathbb{R}^{2}$
such that $\mu(E)>1-\varepsilon$ and $L|_{E}$ is uniformly continuous.
From this it follows easily that for every $\varepsilon>0$, $k\ge1$,
and $n\ge1$ large enough,
\[
\lambda_{n}\times\mu\left\{ (i,x)\::\:\nu_{i,i+k}\times\mu_{x,i}\left\{ (j,y)\::\:d_{\PR}(L(x),L(y))<\varepsilon\right\} >1-\varepsilon\right\} >1-\varepsilon\:.
\]
Hence it suffices to prove the lemma with $L(y)$ appearing in (\ref{eq:proj of comp of comp})
instead of $L(x)$. This together with Lemmas \ref{lem:without comp of comp}
and \ref{lem:distribution-of-components-of-components} completes
the proof.
\end{proof}

\section{\label{sec:Concentration-near-lines}Some algebraic considerations}

This section collects some algebraic facts that will play a role in
the proof of the entropy growth theorem in the next section. We assume
that $\Phi$ is non-conformal and totally irreducible.

Throughout this section we work in the vector space of all affine
maps $A_{2,2}^{vec}$, which contains the group $A_{2,2}$ of invertible
affine maps as a proper subset. We fix a norm on $A_{2,2}^{vec}$
and refer to it whenever we speak of bounded sets of affine maps,
the diameter of such sets, etc.

Recall that for $x\in\mathbb{R}^{2}\setminus\{0\}$ we write $\overline{x}=\mathbb{R}x\in\PR$
for the line (or direction) determined by it, and sometimes write
the elements of $\PR$ as $\overline{v}$ even when $v$ is not specified.Smilarly,
for a map $f:Y\rightarrow\mathbb{R}^{2}$ we write $\overline{f}:Y\setminus f^{-1}(0)\rightarrow\PR$
for the map $\overline{f}(x)=\overline{f(x)}$, and sometimes write
$\overline{f}$ for a function whose range is $\PR$ even if it does
not arise in this way from a map $f$ with range $\mathbb{R}^{2}$. 

\subsection{\label{subsec:Families-of-affine-maps-evaluating-to-line}Families
of affine maps which evaluate to lines}

In this section, which is essentially linear algebra, we consider
the evaluation operation $\psi\rightarrow\psi(x)$ which for a fixed
$x\in\mathbb{R}^{2}$ sends an affine map $\psi\in A_{2,2}^{vec}$
to a point in $\mathbb{R}^{2}$. We study the situation where a family
$\Psi$ of affine maps is mapped by the evaluation operation into
an affine line (which may depend on $x$), and show that if this is
the case, then the direction of the line must depend on $x$ in an
affine manner. We then obtain approximate versions of this statement.

For $\Psi\subseteq A_{2,2}^{vec}$ and $x\in\mathbb{R}^{2}$ we write
$\Psi x=\{\psi x\,:\,\psi\in\Psi\}$.
\begin{lem}
\label{lem:sets-which-evaluate-to-lines}Let $\emptyset\neq\Psi\subseteq A_{2,2}^{vec}$
be a family of affine maps and $Y\subseteq\mathbb{R}^{2}$. Suppose
that the set $\Psi x$ is contained in an affine line for every $x\in Y$.
Then there is an affine map $0\ne\psi\in A_{2,2}^{vec}$ such that
$\Psi x$ is contained in an affine line in direction $\overline{\psi}(x)$
for all $x\in Y\setminus\psi^{-1}(0)$. 
\end{lem}

\begin{proof}
If $\Psi=\{\psi_{0}\}$ consists of a single map then $\Psi x=\{\psi_{0}(x)\}$
is a point and so lies on a line in every direction; so any affine
map $\psi$ will satisfy the conclusion. Otherwise, let $\psi_{1},\psi_{2}\in\Psi$
be distinct maps, and define $\psi(x)=\psi_{2}(x)-\psi_{1}(x)$, so
$\psi\neq0$. Then for any $x\in Y\setminus\psi^{-1}(0)$, the set
$\Psi x$ contains the distinct points $\psi_{1}(x),\psi_{2}(x)$,
so if it is contained in a line this line must have direction $\overline{\psi}(x)$.
This proves the claim.
\end{proof}
\begin{rem}
It is possible to say more about the situation in the lemma: Assuming
also $|\Psi|\geq2$, one of the following possibilities must hold:
\begin{enumerate}
\item The set $\Psi$ lies on an affine line in the space of affine maps,
i.e. there exist $\psi_{1},\psi_{2}\in A_{2,2}^{vec}$ such that $\Psi\subseteq\psi_{1}+\mathbb{R}\psi_{2}$. 
\item There are vectors $0\neq b\in\mathbb{R}^{2}$ and $c\in\mathbb{R}^{2}$
and matrices $A,B$ with $\im(B)\subseteq\mathbb{R}b$, such that
every $\varphi\in\Psi$ is of the form $\varphi(x)=Ax+sBx+tb+c$ for
some $s,t\in\mathbb{R}$.
\end{enumerate}
\end{rem}

We next replace the pointwise version with one for measures. Recall
that for $\theta\in\mathcal{P}(A_{2,2}^{vec})$ and $x\in\mathbb{R}^{2}$
we write $\theta\conv x=\theta\conv\delta_{x}$ for the push-forward
of $\theta$ by the map $g\rightarrow g(x)$. 
\begin{lem}
\label{lem:measures-which-evaluate-to-lines}Let $\nu\in\mathcal{P}(\mathbb{R}^{2})$,
and let $\theta\in\mathcal{P}(A_{2,2}^{vec})$ be a measure satisfying
\[
\nu(x\,:\,\theta\conv x\;\text{is supported on an affine line})=1
\]
(the set is easily seen to be measurable, even closed). Then there
exists an affine map $0\neq\psi\in A_{2,2}^{vec}$ such that $\theta\conv x$
is supported on an affine line in direction $\overline{\psi}(x)$
for $\nu$-a.e. $x\in\mathbb{R}^{2}\setminus\psi^{-1}(0)$. 
\end{lem}

\begin{proof}
The hypothesis on $\nu,\theta$ is that for $\nu$-a.e. $x$, there
exists an affine line $\ell_{x}$ (which can be chosen to vary measurably
with $x$) such that 
\[
\varphi(x)\in\ell_{x}~~~~~~\text{for \ensuremath{\nu}-a.e. \ensuremath{x} and \ensuremath{\theta}-a.e. \ensuremath{\varphi}}\:.
\]
Write $Y\subseteq\mathbb{R}^{2}$ for the set of $x$ for which $\varphi(x)\in\ell_{x}$
for $\theta$-a.e. $\varphi$. The last equation and Fubini imply
that $\nu(Y)=1$. Fix $x\in Y$ and note that the condition $\varphi(x)\in\ell_{x}$
is closed in the variable $\varphi$, so, since it holds for $\theta$-a.e.
$\varphi$, it holds for every $\varphi\in\supp\theta$. Thus, $\varphi(x)\in\ell_{x}$
is true for every pair $(x,\varphi)\in Y\times\supp\theta$. We can
now apply the previous lemma to the sets $Y$ and $\Psi=\supp\theta$
and we obtain the desired map $\psi$.
\end{proof}
The next variant replaces the exact assumptions above by approximate
versions: We assume that $\theta\conv x$ is mostly supported close
to a line $\ell_{x}$ (rather than entirely supported on the line
itself). We conclude that, up to some deterioration of the constants,
$x\rightarrow\ell_{x}$ is given by an affine map at a positive proportion
of points.
\begin{defn}
\label{def:concentration-on-subspace}Let $W\leq\mathbb{R}^{2}$ be
a linear subspace and $\delta>0$. A measure $\nu\in\mathcal{P}(\mathbb{R}^{2})$
is $(W,\delta)$-concentrated if there is a translate $W+v$ of $W$
such that $1-\delta$ of the mass of $\nu$ lies within a $\delta$-distance
of $W+v$.
\end{defn}

Note that for $\overline{v}\in\PR$, saying that $\nu$ is $(\overline{v},\varepsilon)$-concentrated
does not mean that $\nu$ is supported mostly near the line $\overline{v}$,
but rather, near some translate of $\overline{v}$.
\begin{prop}
\label{prop:measures-which-evaluate-to-approximate-lines}Let $\nu\in\mathcal{P}(\mathbb{R}^{2})$
be a measure that gives mass zero to every affine line. Then for every
$\varepsilon,R>0$ there exists a $\delta=\delta(\varepsilon,R)>0$
such that the following holds. 

Let $\theta\in\mathcal{P}(A_{2,2}^{vec})$ be a measure supported
on a set of diameter $R$ (with respect to the norm on $A_{2,2}^{vec}$).
Let $\{\overline{v}_{x}\}_{x\in\mathbb{R}^{2}}\subseteq\PR$ be a
family of lines such that $x\mapsto\overline{v}_{x}$ is measurable,
and
\begin{equation}
\nu(x\,:\,\theta\conv x\;\text{is }(\overline{v}_{x},\delta)\text{-concentrated})>1-\delta\;.\label{eq:almost-concentration-on-line}
\end{equation}
Then there exists an affine map $0\neq\psi\in A_{2,2}^{vec}$ such
that
\begin{equation}
\nu(x\,:\,\theta\conv x\text{ is }(\overline{\psi}(x),\varepsilon)\text{-concentrated})>1-\varepsilon\;,\label{eq:actual-concentration-on-line}
\end{equation}
and
\begin{gather}
\nu(x\,:\,d_{\PR}(\overline{v}_{x},\overline{\psi}(x))<\varepsilon)\qquad\qquad\qquad\qquad\qquad\qquad\nonumber \\
\qquad\qquad>\;\nu(x\,:\,\theta\conv x\text{ is not }(\{0\},\varepsilon)\text{-concentrated})-\varepsilon\;.\label{eq:non-almost-concentration-on points}
\end{gather}
\end{prop}

\begin{rem}
The reason that the probability on the right hand side of (\ref{eq:non-almost-concentration-on points})
appears is that if $x$ is a point for which $\theta\conv x$ is $(\{0\},\varepsilon)$-concentrated,
then $\theta\conv x$ is $(\overline{v},\varepsilon)$-concentrated
for every $\overline{v}\in\PR$, which means that $\overline{v}_{x}$
is not determined, and there is no reason for the given function $x\mapsto\overline{v}_{x}$
to agree with any affine map $\psi$. More concretely, fix $x_{0}\in\mathbb{R}^{2}$,
let $\theta$ be some non-trivial measure on the stabilizer of $x_{0}$
in $A_{2,2}^{vec}$. Thus $\theta\conv x_{0}=\delta_{x_{0}}$ is a
point mass. Now replace $x_{0}$ by the uniform measure $\nu$ on
a small ball around $x_{0}$; by making the ball small, we ensure
that $\theta\conv x$ is still supported on a $\delta$-ball for all
$x\in\supp\nu$. Thus $\theta\conv x$ is $(\overline{v},\delta)$-concentrated
for any $\overline{v}\in\PR$, and any choice of the function $x\mapsto\overline{v}_{x}$
will satisfy the assumptions in the proposition above, and any affine
map $\psi$ will satisfy the first conclusion. But many choices of
the initial function $x\mapsto\overline{v}_{x}$ will be far from
every affine map on $\nu$-most points.
\end{rem}

\begin{proof}
If the conclusion (\ref{eq:actual-concentration-on-line}) were false,
then there would exist an $\varepsilon_{0}>0$ such that the statement
fails for every $\delta>0$. Let $\theta_{n}$ and $\overline{v}_{n,x}\in\PR$
be witnesses of this failure for $\delta_{n}=1/n$; thus, 
\begin{itemize}
\item $\theta_{n}$ is supported on a set of diameter $R$.
\item With $\nu$-probability at least $1-\delta_{n}$ over the choice of
$x$, the measure $\theta_{n}\conv x$ is $(\overline{v}_{n,x},\delta_{n})$-concentrated.
\item There is no affine map $\psi_{n}$ such that $\theta_{n}\conv x$
is $(\overline{\psi}_{n}(x),\varepsilon_{0})$-concentrated with $\nu$-probability
$>1-\varepsilon_{0}$. 
\end{itemize}
We can further assume that the $\theta_{n}$ are supported on the
ball of radius $R$ at the origin of the normed space $A_{2,2}^{vec}$,
since otherwise we can fix $\varphi_{n}\in\supp\theta_{n}$ and replace
$\theta_{n}$ by the translate $T_{\varphi_{n}}\theta_{n}$ (Note
that we are translating in the vector space $A_{2,2}^{vec}$, not
in the group $A_{2,2}$).

Since all the $\theta_{n}$ are now supported on a common compact
set, by passing to a sub-sequence, we can assume that $\theta_{n}\rightarrow\theta$
weakly for some $\theta\in\mathcal{P}(A_{2,2}^{vec})$. 

Fix $\rho>0$. For large enough $n_{0}$, we have that
\[
\nu(x\,:\,\theta_{n_{0}}\conv x\text{ is }(\overline{v}_{n_{0},x},\rho)\text{-concentrated})>1-\rho
\]
(this holds as soon as $1/n_{0}<\rho$). If $n_{0}$ is also large
enough (in a manner depending on $\rho$), then for all $n>n_{0}$
the measures $\theta_{n},\theta_{n_{0}}$ will be sufficiently close
in the weak topology that the previous equation implies 
\[
\nu(x\,:\,\theta_{n}\conv x\text{ is }(\overline{v}_{n_{0},x},2\rho)\text{-concentrated})>1-2\rho\;.
\]
Taking $n\rightarrow\infty$ and using $\theta_{n}\rightarrow\theta$,
we conclude that for every $\rho>0$, if $n_{0}=n_{0}(\rho)$ is large
enough, then
\[
\nu(x\,:\,\theta\conv x\text{ is }(\overline{v}_{n_{0},x},3\rho)\text{-concentrated})>1-3\rho\;.
\]

Choose $\rho_{k}=3\cdot2^{-k}$ and write $\overline{w}_{k,x}=\overline{v}_{n_{0}(2^{-k}),x}$.
By the last equation and Borel-Cantelli, for $\nu$-a.e. $x$ there
is a sequence of affine lines $\ell_{k,x}$ in direction $\overline{w}_{k,x}$,
intersecting a common compact set in $\mathbb{R}^{2}$, such that
for all large enough $k$ (depending on $x$),
\[
(\theta\conv x)(\ell_{k,x}^{(\rho_{k})})>1-\rho_{k}\;.
\]
Fix such an $x\in\supp\nu$, let $\ell_{x}=\lim_{i\rightarrow\infty}\ell_{k(i),x}$
be an accumulation point of the affine lines $\ell_{k,x}$, and let
$\overline{w}_{x}$ denote the direction of $\ell_{x}$, so $\overline{w}_{x}=\lim\overline{w}_{k(i),x}$.
Let $K_{x}=\supp\theta\conv x$; then it is easily seen that $K_{x}\cap\ell_{k(i),x}^{(\rho_{k(i)})}\subseteq\ell_{x}^{(\varepsilon)}$
for all $\varepsilon>0$ and all sufficiently large $i$ (depending
on $\varepsilon$), hence $\theta\conv x(\ell_{x}^{(\varepsilon)})=1$
for every $\varepsilon>0$, and so $\theta\conv x(\ell_{x})=1$.

Since this holds for $\nu$-a.e. $x$ we can apply the previous lemma
to $\nu,\theta$, and find that there exists an affine map $0\neq\psi\in A_{2,2}^{vec}$
such that $\theta\conv x$ is supported on a line in direction $\overline{\psi}(x)$
for $\nu$-a.e. $x\in\mathbb{R}^{2}\setminus\psi^{-1}(0)$; since
$\nu$ gives mass zero to every affine line, it holds unconditionally
for $\nu$-a.e. $x$.

Write $\widetilde{\ell}_{x}$ for the line in direction $\overline{\psi}(x)$
that supports $\theta\conv x$; this is defined for $\nu$-a.e. $x$
(if $\theta\conv x$ is not a point mass, we will have $\overline{\psi}(x)=\overline{w}_{x}$
and $\ell_{x}=\widetilde{\ell}_{x}$, but if $\theta\conv x$ is a
point mass $\overline{w}_{x}$ is not determined). Since $\theta_{n}\rightarrow\theta$
weakly, also $\theta_{n}\conv x\rightarrow\theta\conv x$ weakly for
every $x$. For $\nu$-a.e. $x$, from $\theta\conv x(\widetilde{\ell}_{x})=1$,
we conclude that for large enough $n$ we have $\theta_{n}\conv x(\widetilde{\ell}_{x}^{(\varepsilon_{0})})>1-\varepsilon_{0}$.
Thus for all large enough $n$, with $\nu$-probability $>1-\varepsilon_{0}$
over $x$, we have $\theta_{n}\conv x(\widetilde{\ell}_{x}^{(\varepsilon_{0})})>1-\varepsilon_{0}$.
This contradicts our choice of $\theta_{n}$ and completes the proof
of the first part of the statement.

We now turn to the proof of (\ref{eq:non-almost-concentration-on points}).
Let $\epsilon,R>0$ be given, let $\sigma>0$ be small with respect
to $\epsilon$ (we assume $\sigma=O(\epsilon^{2})$), and let $\delta>0$
be small with respect to $\sigma$ and $R$. Suppose that $\theta\in\mathcal{P}(A_{2,2}^{vec})$
is supported on a set of diameter $R$ and that $\{\overline{v}_{x}\}_{x\in\mathbb{R}^{2}}\subseteq\PR$
is a family of lines with,
\[
\nu(x\,:\,\theta\conv x\;\text{is }(\overline{v}_{x},\delta)\text{-concentrated})>1-\delta\;.
\]
By the first part, we may assume that there exists an affine map $0\neq\psi\in A_{2,2}^{vec}$
such that,
\[
\nu(x\,:\,\theta\conv x\text{ is }(\overline{\psi}(x),\sigma)\text{-concentrated})>1-\sigma\;.
\]

Let $E$ be the set of all $x\in\mathbb{R}^{2}$ for which $\theta\conv x$
is both $(\overline{v}_{x},\sigma)$-concentrated and $(\overline{\psi}(x),\sigma)$-concentrated,
then $\nu(E)>1-2\sigma$. Fix $x\in E$ and suppose that $\theta\conv x$
is not $(\{0\},\epsilon)$-concentrated. Since $x\in E$ there exist
$a_{x},b_{x}\in\mathbb{R}^{2}$ such that,
\[
\theta\conv x(a_{x}+\overline{v}_{x}^{(\sigma)})\ge1-\sigma\text{ and }\theta\conv x(b_{x}+\overline{\psi}(x)^{(\sigma)})\ge1-\sigma\:.
\]
Write
\[
Q:=(a_{x}+\overline{v}_{x}^{(\sigma)})\cap(b_{x}+\overline{\psi}(x)^{(\sigma)}),
\]
then $\theta\conv x(Q)\ge1-2\sigma$. Since $\theta\conv x$ is not
$(\{0\},\epsilon)$-concentrated it follows that $\diam(Q)\ge\epsilon$.
On the other hand, by elementary trigonometry and (\ref{eq:compar-dPR-and-sine-of-angle}),
\[
\diam(Q)=O\left(\frac{\sigma}{\sin(\varangle(\overline{v}_{x},\overline{\psi}(x)))}\right)=O\left(\frac{\sigma}{d_{\PR}(\overline{v}_{x},\overline{\psi}(x))}\right)\:.
\]
Hence, since $\sigma$ is assumed to be small relative to $\epsilon$,
\[
d_{\PR}(\overline{v}_{x},\overline{\psi}(x))\le O(\frac{\sigma}{\epsilon})<\epsilon,
\]
which gives
\[
\nu(x\in\mathbb{R}^{2}\,:\,d_{\PR}(\overline{v}_{x},\overline{\psi}(x))<\varepsilon)\ge\nu(x\in E\::\:\theta\conv x\text{ is not }(\{0\},\varepsilon)\text{-concentrated})\:.
\]
Since $\nu(E)>1-2\sigma>1-\epsilon$, this completes the proof of
the proposition.
\end{proof}
\begin{cor}
\label{cor:condition-for-L-to-be-affine}Let $\nu\in\mathcal{P}(\mathbb{R}^{2})$
be a measure that gives mass zero to every affine line and let $\overline{M}:\mathbb{R}^{2}\rightarrow\PR$
be measurable and defined $\nu$-a.e. Suppose that for some $\varepsilon,R>0$
and every $\delta>0$ there exists a measure $\theta\in\mathcal{P}(A_{2,2}^{vec})$
that is supported on a set of norm diameter $R$, and such that 
\begin{align*}
\nu(x\,:\,\theta\conv x\;\text{is }(\overline{M}(x),\delta)\text{-concentrated}) & >1-\delta,\\
\nu(x\,:\,\theta\conv x\text{ is not }(\{0\},\varepsilon)\text{-concentrated}) & >\varepsilon\;.
\end{align*}
Then there is an affine map $0\neq\psi\in A_{2,2}^{vec}$ such that
$\overline{M}=\overline{\psi}$ on a set of $\nu$-measure at least
$\varepsilon$.
\end{cor}

\begin{proof}
Fix a positive sequence $\varepsilon_{n}\searrow0$, and apply the
previous proposition to get corresponding $\delta_{n}$, which we
may assume satisfies $\delta_{n}\leq\varepsilon_{n}$. Let $\theta_{n}$
be the measure corresponding to $\delta_{n}$ in the hypothesis of
the present corollary (we start with $n$ large enough that $\varepsilon_{n}<\varepsilon$).
We obtain affine maps $\psi_{n}\neq0$ such that 
\[
\nu(x\,:\,d_{\PR}(\overline{\psi}_{n}(x),\overline{M}(x))<\varepsilon_{n})>\varepsilon-\varepsilon_{n}\;.
\]

We can assume that $\left\Vert \psi_{n}\right\Vert =1$ (in the norm
on $A_{2,2}^{vec}$), since $\psi_{n}$ and $\psi_{n}/\left\Vert \psi_{n}\right\Vert $
induce the same map $\mathbb{R}^{2}\rightarrow\PR$. Thus, passing
to a sub-sequence if necessary, we can assume that $\psi_{n}\rightarrow\psi\in A_{2,2}^{vec}$
in the norm metric on $A_{2,2}^{vec}$, in particular $\left\Vert \psi\right\Vert =1$,
so $\psi\neq0$. By the last equation, there is a set $E\subseteq\mathbb{R}^{2}$
with $\nu(E)\geq\varepsilon$ and such that every $x\in E$ belongs
to the event above for infinitely many $n$. Thus, for $x\in E$ there
is a subsequence $n(i,x)$, $i=1,2,3,\ldots$ along which $d_{\PR}(\overline{\psi}_{n(i,x)}(x),\overline{M}(x))\rightarrow0$,
i.e. $\overline{\psi}_{n(i,x)}(x)\rightarrow\overline{M}(x)$; but
also $\psi_{n}(x)\rightarrow\psi(x)$ as $n\rightarrow\infty$ in
$\mathbb{R}^{2}$, and hence for $x\in E\setminus\psi^{-1}(0)$, which
includes $\nu$-a.e. $x\in E$, we have $\overline{\psi}_{n}(x)\rightarrow\overline{\psi}(x)$
in $\PR$. Thus for $\nu$-a.e. $x\in E$, both $\overline{\psi}(x)$
and $\overline{M}(x)$ are limits of the same subsequence of $\overline{\psi}_{n}(x)$,
so they are equal, as desired. 
\end{proof}

\subsection{\label{subsec:The-mu-measure-of-quadratic-curves}The $\mu$-measure
of quadratic curves}

Let $X$ be the attractor of the affine system $\Phi=\{\varphi_{i}\}_{i\in\Lambda}$.
In this section we show that non-conformality and total irreducibility
of $\Phi$ imply that $X$ is not contained in a quadratic curve,
and that $\mu$ gives mass zero to every such curve. Here, by a quadratic
curve we mean the zero set $p^{-1}(0)$ of a quadratic polynomial
$0\ne p\in\mathbb{R}[x,y]$.
\begin{lem}
Let $C$ be a quadratic curve containing $X$. For $x\in X$ let $C_{x}$
denote the connected component of $C$ which contains $x$. Then for
every $x\in X$ and $i\in\Lambda$ we have $\varphi_{i}C_{x}=C_{\varphi_{i}(x)}$.
\end{lem}

\begin{proof}
Let $C=p^{-1}(0)$ for a quadratic polynomial $p$. Fix $x_{0}\in X$
and $i\in\Lambda$, and let $D=D_{x_{0},i}\subseteq C_{x_{0}}$ denote
the set of points $x\in C_{x_{0}}\cap\varphi_{i}^{-1}C$ which are
not isolated in $C_{x_{0}}\cap\varphi_{i}^{-1}C$. This is a non-empty
set because it contains $C_{x_{0}}\cap X$, which is relatively open
in the perfect set $X$. 

We claim that $D$ is open and closed in $C_{x_{0}}$, and hence $D=C_{x_{0}}$.
It is clear that it is closed so we need only show that it is open.
To this end fix $x\in D$. Then we can find $\delta>0$ such that
$B_{\delta}(x)\cap C_{x_{0}}$ is parameterized by an analytic (or
even polynomial) curve $\gamma:(-a,b)\rightarrow\mathbb{R}^{2}$.
Then $p(\varphi_{i}\gamma(t))=0$ whenever $\gamma(t)\in D$, which
happens on a non-discrete set, by the definition of $D$. Thus $p\varphi_{i}\gamma\equiv0$,
which means that the image of $\varphi_{i}\gamma$ lies in $C$; hence
the image of $\gamma$ lies in $D$, and constitutes a neighborhood
of $x$ in $D$. This shows that $D$ is open in $C_{x_{0}}$, as
claimed.

We have shown that every $x\in C_{x_{0}}$ is also in $\varphi_{i}^{-1}C$,
i.e. that $\varphi_{i}C_{x_{0}}\subseteq C$. Since $\varphi_{i}C_{x_{0}}$
is connected and contains $\varphi_{i}(x_{0})$, it follows that $\varphi_{i}C_{x_{0}}\subseteq C_{\varphi_{i}(x_{0})}$.
Now apply the same argument to $C_{\varphi_{i}(x_{0})}$ and $\varphi_{i}^{-1}$;
note that although $X$ is not guaranteed to be mapped into $C$ by
$\varphi_{i}^{-1}$, certainly $\varphi_{i}X$ is, which is enough
for the argument to go through. We conclude that $\varphi_{i}^{-1}C_{\varphi_{i}(x_{0})}\subseteq C_{x_{0}}$,
and altogether, we have shown that $\varphi_{i}C_{x_{0}}=C_{\varphi_{i}(x_{0})}$.
\end{proof}
\begin{cor}
Let $C_{X}$ be the union of connected components of $C$ that intersect
$X$. Then for each $i\in\Lambda$ the map $\varphi_{i}$ is a bijection
of $C_{X}$.
\end{cor}

\begin{proof}
Immediate since there are finitely many (in fact, at most two) connected
components.
\end{proof}
\begin{prop}
\label{prop:measure 0 for alg cur}Assume that $\mu$ is a self-affine
measure generated by a non-conformal and totally irreducible system
$\Phi$ without a common fixed point and a positive probability vector.
Then $\mu$ gives mass zero to every quadratic curve.
\end{prop}

\begin{proof}
Suppose otherwise. Then there is a quadratic curve $C=p^{-1}(0)$
such that $\mu(C)>0$. We claim that then $\mu$ is supported on (a
possibly different) quadratic curve. Indeed, choose a $\xi$-typical
$\omega\in\Pi^{-1}(C)$ (note that $\xi(\Pi^{-1}C)=\mu(C)>0$). Then
with probability one, 
\begin{align*}
\mu(\varphi_{\omega_{1}\ldots\omega_{n}}^{-1}C) & =\varphi_{\omega_{1}\ldots\omega_{n}}\mu(C)\\
 & =\mu_{[\omega_{1}\ldots\omega_{n}]}(C)\\
 & =\xi_{[\omega_{1}\ldots\omega_{n}]}(\Pi^{-1}C)\\
 & \rightarrow1\qquad\text{as }n\rightarrow\infty\;.
\end{align*}
Now, $\varphi_{\omega_{1}\ldots\omega_{n}}^{-1}C=p_{n}^{-1}(0)$ for
$p_{n}=p\circ\varphi_{\omega_{1}\ldots\omega_{n}}$. Normalize each
$p_{n}$ to be a unit vector in the vector space of quadratic polynomials
(normalization does not affect the zero set), and pass to a subsequence
along which $p_{n}$ converge to some quadratic polynomial $p$, and
also such that $p_{n}^{-1}(0)$ converge to a set $C'$ in the Hausdorff
metric on a ball in $\mathbb{R}^{2}$ that supports $\mu$. Then $C'\subseteq p^{-1}(0)$
and $\mu(C')=1$. We can thus replace $C$ by $p^{-1}(0)$, and assume
from the outset that $\mu(C)=1$.

Since $X=\supp\mu$ and $\mu(C)=1$, we have $X\subseteq C$. Let
$C_{X}$ denote the union of those connected components of $C$ that
intersect $X$. We have seen that $\varphi_{i}C_{X}=C_{X}$ for every
$i\in\Lambda$.

Let $M:C\rightarrow\PR$ denote the map (defined at all but at most
finitely many singular points) that takes $x\in C$ to the direction
of the tangent line to $C$ at $x$. Clearly each $\varphi_{i}|_{C_{X}}^{-1}$
induces a map of tangent vectors of $C_{X}$, hence for all but finitely
many $x\in C_{X}$, 
\[
M(\varphi_{i}^{-1}x)=A_{i}^{-1}M(x)\;.
\]
Iterating this for a sequence $i_{1},\ldots,i_{n},\ldots$ we have
\begin{equation}
M(\varphi_{i_{n}}^{-1}\ldots\varphi_{i_{1}}^{-1}x)=A_{i_{n}}^{-1}\ldots A_{i_{1}}^{-1}M(x)\;.\label{eq:tangent-vector-iterates}
\end{equation}

Choosing $i_{1},i_{2}\ldots$ to be i.i.d. with marginal $p$, for
fixed $x$, it is easy to see that the sequence $\varphi_{i_{n}}^{-1}\ldots\varphi_{i_{1}}^{-1}x\rightarrow\infty$
a.s., due to the expanding nature of the maps $\varphi_{i}^{-1}$
(and the fact that they do not have a common fixed point). It is also
elementary that as one escapes to infinity, the tangent vectors to
$C$ accumulate on a finite set of directions (namely, on a single
direction for a parabola or line, and a pair of directions for a hyperbola).
Thus the distribution of the left hand side of (\ref{eq:tangent-vector-iterates}),
with the indices chosen randomly, accumulates only on atomic measures.

On the other hand, the right hand side of the last equation is a random
walk on $\PR$ whose steps are chosen from $\{A_{i}\}_{i\in\Lambda}^{-1}$,
a non-conformal and totally irreducible system, and thus is attracted
to the Furstenberg measure, which under our assumptions has no atoms,
in contradiction to the previous paragraph.
\end{proof}
\begin{rem*}
It follows also from the work of Feng and K\"aenmaki \cite{FK} that
the only algebraic curves which can support a non-trivial planar self-affine
set are quadratic curves; thus, the proposition above holds for any
algebraic curve in the plane. 
\end{rem*}
\begin{rem*}
The last proposition actually holds also in the conformal case (i.e.
when $\Phi$ is conjugate to a system of similarities) using a more
direct re-scaling argument: if the measure gave positive mass to a
smooth curve, then, by re-scaling cylinder measures which are increasingly
supported on this curve, we would find that the measure is supported
on a line (the re-scaling of the tangent line to the curve), contradicting
irreducibility.
\end{rem*}

\subsection{\label{subsec:non-affinity-of L}The non-affinity of $L$}

In this section we assume again non-conformality and total irreducibility,
and also that $\dim\mu<2$, which ensures that $L$ is well defined
as a function on $X$ at $\mu$-a.e. point (Theorem \ref{thm:L-descends-to-mu}). 

We prove that the function $L:X\rightarrow\PR$ from Section \ref{subsec:L-factors-through-Pi}
does not arise from an affine map. More precisely, we show that there
does not exist an affine map $0\neq\psi\in A_{2,2}^{vec}$ such that
$L(x)=\overline{\psi}(x)$ for $\mu$-a.e. $x$. Here $\overline{\psi}:\mathbb{R}^{2}\setminus\psi^{-1}(0)\rightarrow\PR$
is the map $x\mapsto\overline{\psi(x)}$. It is defined $\mu$-a.e.
because, by total irreducibility, $\mu$ does not give mass to any
affine line.

Recall that $\varphi_{i}(x)=A_{i}x+b_{i}$ for $i\in\Lambda$ and
$x\in\mathbb{R}^{2}$, and more generally for $\psi\in A_{2,2}^{vec}$
we write $\psi(x)=A_{\psi}x+b_{\psi}$.

Given $i\in\Lambda$ and $\omega\in\Lambda^{\mathbb{N}}$ denote the
concatenation of $i$ with $\omega$ by $i\omega$. 

Also let $\nu$ denote the uniform (rotation-invariant) probability
measure on $\PR$.
\begin{lem}
\label{lem:L is equivariant}Let $i\in\Lambda$, then $L(i\omega)=A_{i}(L\omega)$
for $\xi$-a.e. $\omega\in\Lambda^{\mathbb{N}}$.
\end{lem}

\begin{proof}
By one of the characterizations of $L$ (see Section \ref{sub:furstenberg}),
for $p^{\mathbb{N}}$-a.e. $\omega$
\begin{align*}
\delta_{L(\omega)} & =\lim_{n\rightarrow\infty}A_{\omega_{1}}\ldots A_{\omega_{n}}\nu\\
 & =A_{\omega_{1}}\lim_{n\rightarrow\infty}A_{\omega_{2}}\ldots A_{\omega_{n}}\nu\\
 & =\delta_{A_{\omega_{1}}L(S\omega)}\;,
\end{align*}
where $S$ is the left shift map. This is equivalent to the statement
we are proving.
\end{proof}
Given $x,y\in\mathbb{R}^{2}$, write $x\parallel y$ to indicate that
$\spn\{x,y\}\le1$ (this allows one or both of the vectors to be $0$).
Denote the $2\times2$ identity matrix by $I$.
\begin{lem}
\label{lem:commutes in action on P(R^2)}Let $B$ be a $2\times2$
matrix such that
\begin{equation}
BA_{i}x\parallel A_{i}Bx\text{ for }x\in\mathbb{R}^{2}\text{ and }i\in\Lambda\:.\label{eq:commutes in action on P(R^2)}
\end{equation}
Then there exists $\beta\in\mathbb{R}$ such that $B=\beta I$.
\end{lem}

\begin{proof}
If $B=0$ then the lemma holds with $\beta=0$, so assume that $B\ne0$. 

We next claim that $\mathrm{\rank}(B)\neq1.$ For suppose that $\mathrm{rank}(B)=1$.
Set $W=\mathrm{\im}(B)$ and for each $i\in\Lambda$ choose $\ell\in\PR$
such that $\ell,A_{i}\ell\neq\ker B$; then by (\ref{eq:commutes in action on P(R^2)}),
\[
W=BA_{i}\ell=A_{i}B\ell=A_{i}W\;.
\]
Thus $W$ is a common fixed point of $\{A_{i}\}_{i\in\Lambda}$, contradicting
total irreducibility.

We next claim that $BL(\omega)=L(\omega)$ for $\xi$-a.e. $\omega\in\Lambda^{\mathbb{N}}$.
Indeed, choosing a typical $\omega$, we have $\delta_{L(\omega)}=\lim_{n\rightarrow\infty}A_{\omega_{1}}\ldots A_{\omega_{n}}\nu$.
Since $B$ is invertible, $B\nu$ is also a continuous measure on
$\PR$, so we have 
\begin{align*}
\delta_{B\cdot L(\omega)} & =B\cdot\lim A_{\omega_{1}}\ldots A_{\omega_{n}}\nu\\
 & =\lim(BA_{\omega_{1}}\ldots A_{\omega_{n}}\nu)\\
 & =\lim A_{\omega_{1}}\ldots A_{\omega_{n}}(B\nu)\\
 & =\delta_{L(\omega)}\;.
\end{align*}

Finally, the Furstenberg measure $\eta=L\xi$ is continuous, so there
exist infinitely many lines which are preserved by $B$. It is now
easy to see that there must exist a $\beta\in\mathbb{R}$ with $B=\beta I$,
which completes the proof of the lemma.
\end{proof}
Recall that for $\varphi\in A_{2,2}^{vec}$ we write $\varphi(x)=A_{\varphi}x+b_{\varphi}$.
\begin{lem}
\label{lem:equal translation}Let $\varphi,\psi\in A_{2,2}$ be such
that $A_{\varphi}=A_{\psi}$ and $\varphi x\parallel\psi x$ for all
$x\in\mathbb{R}^{2}$, then also $b_{\varphi}=b_{\psi}$.
\end{lem}

\begin{proof}
By assumption, $\varphi,\psi$ are invertible. By $\varphi(0)\parallel\psi(0)$
it follows that there exist $0\ne v\in\mathbb{R}^{2}$ and $t_{\varphi},t_{\psi}\in\mathbb{R}$
such that $b_{\varphi}=t_{\varphi}v$ and $b_{\psi}=t_{\psi}v$. For
$u\in\mathbb{R}^{2}$,
\[
u+t_{\varphi}v=\varphi(A_{\varphi}^{-1}u)\parallel\psi(A_{\varphi}^{-1}u)=\psi(A_{\psi}^{-1}u)=u+t_{\psi}v\;.
\]
Hence, if $u$ is independent of $v$,
\[
0=\det\left(\begin{array}{cc}
1 & t_{\varphi}\\
1 & t_{\psi}
\end{array}\right)=t_{\psi}-t_{\varphi}\:.
\]
This gives $b_{\varphi}=b_{\psi}$, which completes the proof of the
lemma. 
\end{proof}
\begin{prop}
\label{prop:non-aff of L}There does not exist $0\ne\psi\in A_{2,2}^{vec}$
with $Lx=\overline{\psi}x$ for $\mu$-a.e. $x\in\mathbb{R}^{2}$.
\end{prop}

\begin{proof}
Assume by contradiction that there exists $0\ne\psi\in A_{2,2}^{vec}$
with $Lx=\overline{\psi}x$ for $\mu$-a.e. $x\in\mathbb{R}^{2}$.
The measure $\eta=L\mu$ is continuous, hence $\overline{\psi}$ can't
be constant, which implies $A_{\psi}\ne0$. 

Let $i\in\Lambda$, then by the definition of $L:\mathbb{R}^{2}\rightarrow\PR$
(see Section \ref{sec:the function L}) and Lemma \ref{lem:L is equivariant}
it follows that for $\xi$-a.e. $\omega\in\Lambda^{\mathbb{N}}$,
\[
L(\varphi_{i}(\Pi\omega))=L(\Pi(i\omega))=L(i\omega)=A_{i}(L\omega)=A_{i}(L(\Pi\omega))\:.
\]
Hence $L(\varphi_{i}x)=A_{i}(Lx)$ for $\mu$-a.e. $x\in\mathbb{R}^{2}$,
which gives,
\begin{equation}
\overline{\psi\varphi_{i}x}=\overline{\psi}(\varphi_{i}x)=A_{i}(\overline{\psi}x)=\overline{A_{i}\psi x}\text{ for }\mu\text{-a.e }x\in\mathbb{R}^{2}\:.\label{eq:on the same line}
\end{equation}
For $x\in\mathbb{R}^{2}$ write
\[
p(x)=\det\left(\psi\varphi_{i}x\;|\;A_{i}\psi x\right),
\]
then $p\in\mathbb{R}[X,Y]$ is a quadratic polynomial. By (\ref{eq:on the same line})
we have $\mu(p^{-1}\{0\})=1$, hence $p=0$ by Proposition \ref{prop:measure 0 for alg cur}.

From $p=0$ we get,
\[
\psi\varphi_{i}x\parallel A_{i}\psi x\text{ for }x\in\mathbb{R}^{2}\:.
\]
By expanding this,
\begin{equation}
A_{\psi}A_{i}x+A_{\psi}b_{i}+b_{\psi}\parallel A_{i}A_{\psi}x+A_{i}b_{\psi}\text{ for }x\in\mathbb{R}^{2}\:.\label{eq:parallel lines}
\end{equation}
By letting $|x|\rightarrow\infty$ and dividing by $|x|$, we get
\[
A_{\psi}A_{i}x\parallel A_{i}A_{\psi}x\text{ for }x\in\mathbb{R}^{2}\:.
\]
Since this holds for all $i\in\Lambda$ and from Lemma \ref{lem:commutes in action on P(R^2)},
it follows that $A_{\psi}=\beta I$ for some $0\ne\beta\in\mathbb{R}$. 

Let $i\in\Lambda$, then by inserting $A_{\psi}=\beta I$ into (\ref{eq:parallel lines}),
\[
\beta A_{i}x+\beta b_{i}+b_{\psi}\parallel\beta A_{i}x+A_{i}b_{\psi}\text{ for }x\in\mathbb{R}^{2}\:.
\]
From this and Lemma \ref{lem:equal translation} we get that $\beta b_{i}+b_{\psi}=A_{i}b_{\psi}$
or equivalently that $b_{i}=\beta^{-1}(A_{i}-I)b_{\psi}$. Set $w=-\beta^{-1}b_{\psi}$,
then a direct computation gives $\varphi_{i}(w)=w$. As this holds
for each $i\in\Lambda$ we have found that all $\varphi_{i}$, $i\in\Lambda$,
share a common fixed point. This contradicts our basic assumptions
(see Section \ref{subsec:Statement-of-results}) and completes the
proof of the proposition.
\end{proof}
\begin{cor}
\label{cor:non-affinity-of-L}There does not exist $0\ne\psi\in A_{2,2}^{vec}$
with $Lx=\overline{\psi}x$ on a set of $x$ of positive $\mu$-measure.
\end{cor}

\begin{proof}
Suppose that $E\subseteq\mathbb{R}^{2}$, $\mu(E)>0$ and $0\ne\psi\in A_{2,2}^{vec}$
satisfies $Lx=\overline{\psi}x$ for every $x\in E$. Let $F=\Pi^{-1}E$
so $\xi(F)=\mu(E)>0$.

Let $\delta>0$. By regularity of $\xi$ we can choose a cylinder
set $C=[i_{1}\ldots i_{n}]$ such that $\xi_{C}(F)>1-\delta$. By
Lemma \ref{lem:L is equivariant} we have
\[
L(\Pi(\omega))=A_{i_{n}}^{-1}\ldots A_{i_{1}}^{-1}L(\Pi(i_{1}\ldots i_{n}\omega))\qquad\text{for }\xi\text{-a.e. }\omega\;.
\]
Now, $i_{1}\ldots i_{n}\omega\in F$ if and only if $\omega\in S^{n}(F\cap C)$
(recall that $S$ is the left shift map, and we have used the fact
that $S^{n}:C\rightarrow\Lambda^{\mathbb{N}}$ is a homeomorphism),
and this occurs with $\xi$-probability $\xi(S^{n}(F\cap C))=\xi_{C}(F)>1-\delta$.
Hence, we find that with $\xi$-probability at least $1-\delta$ over
the choice of $\omega$,
\begin{eqnarray*}
L(\Pi(\omega)) & = & A_{i_{n}}^{-1}\ldots A_{i_{1}}^{-1}L(\Pi(i_{1}\ldots i_{n}\omega))\\
 & = & A_{i_{n}}^{-1}\ldots A_{i_{1}}^{-1}\overline{\psi}(\Pi(i_{1}\ldots i_{n}\omega))\\
 & = & A_{i_{n}}^{-1}\ldots A_{i_{1}}^{-1}\overline{\psi}(\varphi_{1}\ldots\varphi_{n}\Pi(\omega))\\
 & = & \overline{A_{i_{n}}^{-1}\ldots A_{i_{1}}^{-1}\psi\varphi_{1}\ldots\varphi_{n}}(\Pi(\omega))\:.
\end{eqnarray*}

Since $A_{i_{n}}^{-1}\ldots A_{i}^{-1}\psi\varphi_{1}\ldots\varphi_{n}$
is affine, we have shown that if $L$ agrees with an affine function
on a set of positive measure, then it agrees with a (possibly different)
affine function on a set of arbitrarily large measure. Normalizing
these functions in the normed space $A_{2,2}^{vec}$ and passing to
a subsequential limit, we conclude that $L$ is a.e. affine, which
by the last proposition is impossible.
\end{proof}
Finally, we combine this with the results of Section \ref{subsec:Families-of-affine-maps-evaluating-to-line}
to obtain:
\begin{cor}
\label{cor:non-concentration-of-L}For every $\varepsilon,R>0$ there
exists a $\delta>0$ with the following property: If $\theta\in\mathcal{P}(A_{2,2}^{vec})$
is a measure supported on a set of diameter $R$, and such that
\[
\mu\left(x\,:\,\theta\conv x\text{ is not }(\{0\},\varepsilon)\text{-concentrated}\right)>\varepsilon\;,
\]
then
\[
\mu\left(x\,:\,\theta\conv x\text{ is }(L(x),\delta)\text{-concentrated}\right)\leq1-\delta\;.
\]
\end{cor}

\begin{proof}
If not, then, for some $\varepsilon,R>0$ and every $\delta>0$, we
could find a measure $\theta\in\mathcal{P}(A_{2,2}^{vec})$ with support
of diameter at most $R$, for which the first inequality is valid
and the second one is reversed. But then Corollary \ref{cor:condition-for-L-to-be-affine}
would imply that $L$ agrees with an affine map on a set of $\mu$-measure
at least $\varepsilon$, contradicting the previous corollary.
\end{proof}

\section{\label{sec:Proof-of ent growth}Entropy growth under convolution}

In this section we assume that $\Phi$ is non-conformal and totally
irreducible (but do not assume exponential separation). We also assume
that $\dim\mu<2$.

Recall that $*$ denotes convolution on $\mathbb{R}^{d}$, and that
for $\theta\in\mathcal{P}(A_{2,2})$ and $\nu\in\mathcal{P}(\mathbb{R}^{2})$
we write $\theta\conv\nu$ for the push-forward of $\theta\times\nu$
by $(g,x)\mapsto gx$. We also write $\theta\conv x=\theta\conv\delta_{x}$
etc.

Our purpose in this section is to prove Theorem \ref{thm:entropy growth-near-identity},
stating that when $\theta$ has non-negligible entropy and is supported
within bounded distance of the identity map, $\theta\conv\mu$ has
greater entropy than $\mu$ alone. 

\subsection{\label{subsec:Entropy-growth-in-Rd}Entropy growth under linear convolution
in $\mathbb{R}^{2}$}

Recall Definition \ref{def:concentration-on-subspace} of a $(V,\delta)$-concentrated
measure. Complementing this is the following notion which describes
measures whose (approximate) conditional measures on translates of
$V$ are (almost) uniform. 
\begin{defn}
\label{def:saturated-subspace}Let $V\subset\mathbb{R}^{2}$ be a
linear subspace, $\varepsilon>0$, and $m\ge1$. A measure $\nu\in\mathcal{P}(\mathbb{R}^{2})$
is said to be $(V,\varepsilon,m)$-saturated if
\[
H_{m}(\nu)\ge\dim V+H_{m}(\pi_{V^{\perp}}\nu)-\varepsilon\:.
\]
\end{defn}

It is not hard to see that if $\theta,\nu\in\mathcal{P}(\mathbb{R}^{2})$
are compactly supported, and if $\theta$ is $(V,\varepsilon)$-concentrated
and $\nu$ is $(V,\varepsilon,m)$-saturated for some subspace $V\leq\mathbb{R}^{2}$,
for some large $m$ and sufficiently small $\varepsilon>0$, then
$H(\theta*\nu,\mathcal{D}_{m})\approx H(\nu,\mathcal{D}_{m})$. The
next theorem shows that, in a local, statistical sense, this is the
only way that this can happen.

Recall from Section \ref{q-adic-components} that $\nu^{x,i}$ denotes
the re-scaled component, i.e. $\nu_{x,i}$ pushed forward by a homothety
from $\mathcal{D}_{i}(x)$ to $[0,1)^{2}$.
\begin{thm}
[{\cite[Theorem 2.8]{HO2}}]\label{thm:Hochman's entropy growth}For
every $\varepsilon>0$ and $m\ge1$ there exists $\delta=\delta(\varepsilon,m)>0$,
such that for every $n\ge N(\varepsilon,\delta,m)$ the following
holds: Let $k\ge1$ and $\theta,\nu\in\mathcal{P}(\mathbb{R}^{2})$
satisfy
\[
\diam(\supp(\theta)),\diam(\supp(\nu))=O(2^{-k})
\]
and
\[
\frac{1}{n}H(\theta*\nu,\mathcal{D}_{k+n})<\frac{1}{n}H(\nu,\mathcal{D}_{k+n})+\delta\:.
\]
Then there exist linear subspaces $V_{k},...,V_{k+n}\subset\mathbb{R}^{2}$
such that
\[
\mathbb{P}_{k\le i\le k+n}\left(\begin{array}{c}
\nu^{x,i}\mbox{ is }(V_{i},\varepsilon,m)\mbox{-saturated and}\\
\theta^{y,i}\mbox{ is }(V_{i},\varepsilon)\mbox{-concentrated}
\end{array}\right)>1-\varepsilon\:.
\]
\end{thm}

We have stated this in $\mathbb{R}^{2}$ but analogs are valid in
any dimension.

\subsection{\label{subsec:Concentration-persists}Concentration persists through
coordinate changes}

The property in Theorem \ref{thm:Hochman's entropy growth}, that
most components of a measure are $(V,\delta)$-concentrated, depends
on the coordinate system one works with. One can easily give examples
of measures with components which at some scale are predominantly
concentrated, but for another coordinate system this property is lost
(this can happen if the measure looks like a combination of measures
supported on line segments which were contained in a different neighboring
cells, but, after the coordinate change, they lie in a common cell).
However, when taken across several scales, concentration of components
is more robust, and does persists under coordinate changes, albeit
with some degradation of the parameters.

We need something slightly stronger, which allows us not only to change
coordinates in $\mathbb{R}^{2}$, but also to decompose a measure
$\theta\conv x$ for $\theta\in\mathcal{P}(A_{2,2})$ according to
the dyadic decomposition of $\theta$, and conclude that after this
decomposition, the pieces $\theta_{g,i}\conv x$ are still concentrated
if the original measure $\theta\conv x$ was.
\begin{defn}
\label{def:concentration-on-several-subspaces}Let $\nu\in\mathcal{P}(\mathbb{R}^{2})$,
$W\subset\mathbb{R}^{2}$ a linear subspace, $\delta>0$, and $m\ge1$.
We say that $\nu$ is $(W,\delta)^{m}$-concentrated if there exist
$x_{1},...,x_{m}\in\mathbb{R}^{2}$ with
\[
\nu\left(\cup_{j=1}^{m}\left(x_{j}+W\right)^{(\delta)}\right)\ge1-\delta\:.
\]
\end{defn}

Recall that $A_{2,2}$ is endowed with an invariant metric $d$ which
is derived from a Riemannian metric. It is not hard to see that for
a bounded set $Id\in B\subset A_{2,2}$ there exists a $C=C(B)>0$
such that
\begin{equation}
\diam(E\conv x)\le C(1+|x|)\cdot\diam(E)\text{ for every }E\subset B\text{ and }x\in\mathbb{R}^{2},\label{eq:bd on diam(E.x)}
\end{equation}
where $\diam(E)$ is taken with respect to $d$. We omit the proof
of the following lemma. It can be carried out by using (\ref{eq:bd on diam(E.x)})
and by imitating the proof of \cite[Lemma 5.4]{HO2}.
\begin{lem}
Let $\theta\in\mathcal{P}(A_{2,2})$, $x\in\mathbb{R}^{2}$, $k,m\ge1$,
$\delta>0$, and fix a subspace $W\subset\mathbb{R}^{2}$. Suppose
that $|x|=O(1)$, $d(\psi,Id)=O(1)$ for $\psi\in\supp(\theta)$,
$\diam(\supp(\theta))=O(2^{-k})$, and $S_{2^{k}}(\theta\conv x)$
is $(W,\delta)^{m}$-concentrated. Then for $n=\left[\frac{1}{2}\log\left(1/\delta\right)\right]$
and $\delta'=O_{m}(\frac{\log\log(1/\delta)}{\log(1/\delta)})$ we
have,
\[
\mathbb{P}_{k\le i\le k+n}\left(S_{2^{i}}(\theta_{\psi,i}\conv x)\text{ is }(W,\delta')\text{-concentrated}\right)>1-\delta'\:.
\]
\end{lem}

The proof of the following proposition is also omitted. It can be
carried out by using the previous lemma and (\ref{eq:bd on diam(E.x)}),
and by imitating the proof of \cite[Proposition 5.5]{HO2}.
\begin{prop}
\label{prop:from conc on eu to cont on A_2,2}For every $\varepsilon>0$
there exist $n=n(\varepsilon)\ge1$ and $\delta=\delta(\varepsilon)>0$,
with $n\rightarrow\infty$ and $\delta\rightarrow0$ as $\varepsilon\rightarrow0$,
such that the following holds. Let $\theta\in\mathcal{P}(A_{2,2})$,
$x\in\mathbb{R}^{2}$, $k\ge1$, and fix a subspace $W\subset\mathbb{R}^{2}$.
Suppose that $|x|=O(1)$, $d(\psi,Id)=O(1)$ for $\psi\in\supp(\theta)$,
and
\[
\mathbb{P}_{i=k}\left((\theta\conv x)^{y,i}\text{ is }(W,\delta)\text{-concentrated}\right)>1-\delta\:.
\]
Then,
\[
\mathbb{P}_{k\le i\le k+n}\left(S_{2^{i}}\left(\theta_{\psi,i}\conv x\right)\text{ is }(W,\varepsilon)\text{-concentrated}\right)>1-\varepsilon\:.
\]
\end{prop}

\subsection{\label{subsec:Linearization}Linearization }

The action operation $f:A_{2,2}\times\mathbb{R}^{2}\rightarrow\mathbb{R}^{2}$,
$f(\varphi,x)=\varphi(x)$, induces the convolution operation $\theta\conv\nu=f(\theta\times\nu)$
on measures. Because $f$ is differentiable, this action can be linearized:
if $I\subseteq A_{2,2}$ and $J\subseteq\mathbb{R}^{2}$ are small
sets of diameter $\delta$, then $f|_{I\times J}$ will be close to
linear: Specifically for $(\varphi_{0},x_{0}),(\varphi,x)\in I\times J$,
we will have 
\begin{align*}
f(\varphi,x) & =(\varphi_{0}+(\varphi-\varphi_{0}))(x_{0}+(x-x_{0}))\\
 & \approx\varphi_{0}x_{0}+(\varphi-\varphi_{0})x_{0}+\varphi_{0}(x-x_{0})+(\varphi-\varphi_{0})(x-x_{0})\\
 & =\varphi x_{0}+\varphi_{0}x-\varphi_{0}x_{0}+O(\delta^{2})\;.
\end{align*}
Letting $\theta\in\mathcal{P}(I)$ and $\nu\in\mathcal{P}(J)$ and
choosing $(\varphi,x)$ at random according to $\theta\times\nu$,
this tells us that $\theta\conv\nu=f(\theta\times\nu)$ is equal,
up to some translations and a small error term, to the distribution
of the sum of $\varphi x$ and $\varphi_{0}x$; which is nothing other
than $(\theta\conv x)*(\varphi_{0}\nu)$. This is, essentially, the
proof of the following lemma (except for verifying that the error
term is small enough to affect entropy negligibly). The formal proof
is similar to the proof of \cite[Lemma 4.2]{BHR}, and is omitted.
\begin{thm}
\label{thm:linearization}Let $Z\subset A_{2,2}\times\mathbb{R}^{2}$
be a compact set. For every $\varepsilon>0$, $k>K(\varepsilon)$,
and $0<\delta<\delta(Z,\varepsilon,k)$ the following holds. Let $(\psi_{0},x_{0})\in Z$,
$\theta\in\mathcal{P}(B_{\delta}(\psi_{0}))$, and $\tau\in\mathcal{P}(B_{\delta}(x_{0}))$,
then
\[
\left|\frac{1}{k}H(\theta\conv\tau,\mathcal{D}_{k-\log\delta})-\frac{1}{k}H((\theta\conv x)*(\psi_{0}\tau),\mathcal{D}_{k-\log\delta})\right|<\varepsilon\:.
\]
\end{thm}

The next proposition is needed to show that if $\theta\in\mathcal{P}(A_{2,2})$
has substantial entropy then so do measures $\theta\conv x$ obtained
by ``pushing it down'' to $\mathbb{R}^{2}$. This is, actually,
not true: It may be that $\theta$ is supported on the stabilizer
of $x$, a condition which still allows it to have large entropy,
but in which case $\theta\conv x=\delta_{x}$ is as concentrated as
possible. However, for a given $\theta$ this cannot happen too often,
because the stabilizers of any three non-colinear points in $\mathbb{R}^{2}$
intersect trivially (equivalently, the action on three such points
determine an affine map). One can make this more quantitative and
show that if a set of points in $\mathbb{R}^{2}$ is far enough from
being contained in an affine line, then the entropy of $\theta\conv x$
will be a constant fraction of the entropy of $\theta$ for most points
in the collection. This is the idea behind the next result; we omit
the formal proof which is very similar to the proof of \cite[Lemma 4.5]{BHR}.

In what follows we rely on the fact that $\mu$ is not supported on
a line. This follows from our assumptions that $\Phi$ is totally
irreducible and that its members don't all have the same fixed points.
\begin{prop}
\label{prop:from ent in A_2,2 to ent in R^2}For every compact $Z\subset A_{2,2}$
there exists a constant $C=C(Z,\mu)>1$ such that for every $\theta\in\mathcal{P}(A_{2,2})$
supported on $Z$ and every $k,i\ge1$,
\[
\mu\left\{ x\::\:\frac{1}{k}H\left(\theta\conv x,\mathcal{D}_{i+k}\right)\ge\frac{1}{Ck}H\left(\theta,\mathcal{D}_{i+k}\right)-\frac{C}{k}\right\} \ge C^{-1}\:.
\]
\end{prop}

We use this to prove that, roughly, if $\theta\in\mathcal{P}(A_{2,2})$
has non-trivial entropy, then a non-negligible fraction of its components
$\theta_{\psi,i}$ are not too close to being an atom, at least after
re-scaling and translation by $\psi^{-1}$.

Recall that $\lambda_{n}$ denotes the uniform measure on $\mathcal{N}_{n}=\{1,\ldots,n\}$
(Section \ref{sub:q-adic-partitions}).
\begin{lem}
\label{lem:comp not conc}For every $\varepsilon,R>0$ there exists
$\delta=\delta(\varepsilon,R)>0$ such that for $k\ge K(\varepsilon,R,\delta)\ge1$
and $n\ge N(\varepsilon,R,\delta,k)\ge1$ the following holds. Let
$\theta\in\mathcal{P}(A_{2,2})$ be such that $\diam(\supp(\theta))\leq R$
with respect to $d$ and $\frac{1}{n}H(\theta,\mathcal{D}_{n})>\varepsilon$.
Then $\lambda_{n}\times\theta(F)>\delta$, where $F=F(\theta)$ is
the set of all $(i,\psi)\in\mathcal{N}_{n}\times A_{2,2}$ such that
\[
\mathbb{P}_{i\le j\le i+k}\left(\mu\left\{ x\::\:\begin{array}{c}
S_{2^{j}}((\psi^{-1}\theta_{\psi,i})_{\varphi,j})\conv x\text{ is}\\
\mbox{not }(\{0\},\delta)\mbox{-concentrated}
\end{array}\right\} >\delta\right)>\delta\:.
\]
\end{lem}

\begin{proof}
Let $C>1$ be a large global constant, which will be determined during
the proof of the lemma. Let $\varepsilon,R>0$, let $m\ge1$ large
with respect to $\varepsilon$ and $R$, $\delta>0$ small with respect
to $m$, and let $k\ge1$ large with respect to $\delta$, and $n\ge1$
large with respect to $k$. Suppose that $m$ is so large with respect
to $\varepsilon$ and that $\delta$ is so small with respect to $\varepsilon$
and $m$, that for every $\nu\in\mathcal{P}(\mathbb{R}^{2})$ with
$\diam(\supp(\nu))\leq C$,
\begin{equation}
\nu\text{ is }(\{0\},\delta)\mbox{-concentrated implies that }\frac{1}{m}H\left(\nu,\mathcal{D}_{m}\right)<\frac{\varepsilon}{C}\:.\label{eq:conc implies small ent}
\end{equation}

Let $\theta\in\mathcal{P}(A_{2,2})$ satisfy $\diam(\supp(\theta))\leq R$
and $\frac{1}{n}H(\theta,\mathcal{D}_{n})>\varepsilon$. By $\frac{1}{n}H(\theta,\mathcal{D}_{n})>\varepsilon$
and Lemma \ref{lem:multiscale-entropy-formula},
\begin{equation}
\mathbb{E}_{0\le i\le n}\left(\frac{1}{k}H\left(\psi^{-1}\theta_{\psi,i},\mathcal{D}_{i+k}\right)\right)\ge\varepsilon-O(\frac{k}{n}+\frac{1}{k})>\frac{\varepsilon}{2}\:.\label{eq:E(ent of comp)>=00003Depsilon}
\end{equation}
By Lemma \ref{lem:Qn-has-bounded-degree}, the integrand on the left
hand side of (\ref{eq:E(ent of comp)>=00003Depsilon}) is $O(1)$.
Hence for some global constant $C_{0}>1$,
\[
\mathbb{P}_{1\le i\le n}\left(\frac{1}{k}H\left(\psi^{-1}\theta_{\psi,i},\mathcal{D}_{i+k}\right)\ge\frac{\varepsilon}{C_{0}}\right)\ge\frac{\varepsilon}{C_{0}}\:.
\]
From this and by applying Lemma \ref{lem:multiscale-entropy-formula}
once more we get that $\lambda_{n}\times\theta(F')>\frac{\varepsilon}{C_{0}}$,
where $F'$ is the set of all $(i,\psi)\in\mathcal{N}_{n}\times A_{2,2}$
such that
\[
\mathbb{E}_{i\le j\le i+k}\left(\frac{1}{m}H\left(\left(\psi^{-1}\theta_{\psi,i}\right)_{\varphi,j},\mathcal{D}_{j+m}\right)\right)\ge\frac{\varepsilon}{C_{0}}-O(\frac{m}{k})\ge\frac{\varepsilon}{2C_{0}}\:.
\]
As above, the integrand on the left hand side of the last inequality
is $O(1)$. Hence there exists a global constant $C_{1}>1$ such that
for $(i,\psi)\in F'$,
\[
\mathbb{P}_{i\le j\le i+k}\left(\frac{1}{m}H\left(\left(\psi^{-1}\theta_{\psi,i}\right)_{\varphi,j},\mathcal{D}_{j+m}\right)\ge\frac{\varepsilon}{C_{1}}\right)\ge\frac{\varepsilon}{C_{1}}\:.
\]

Now by Proposition \ref{prop:from ent in A_2,2 to ent in R^2}, by
assuming that $C$ is large enough, and by assuming that $m$ is sufficiently
large with respect to $\varepsilon$, it follows that for $(i,\psi)\in F'$,
\[
\mathbb{P}_{i\le j\le i+k}\left(\mu\left\{ x\::\:\frac{1}{m}H\left(S_{2^{j}}\left(\left(\psi^{-1}\theta_{\psi,i}\right)_{\varphi,j}\right)\conv x,\mathcal{D}_{m}\right)\ge\frac{\varepsilon}{C}\right\} >C^{-1}\right)\ge\frac{\varepsilon}{C}\:.
\]
Assume that $C$ is large enough that the supports of the measures,
appearing inside the entropy in the last expression, almost surely
have diameter at most $C$. By (\ref{eq:conc implies small ent})
and by assuming that $\delta<\frac{\varepsilon}{C}$ it now follows
that $F'\subset F$, where $F$ is the set defined in the statement
of the Lemma. Since $\lambda_{n}\times\theta(F')>\frac{\varepsilon}{C_{0}}>\delta$
this completes the proof.
\end{proof}
The following is a variant of Lemma \ref{lem:multiscale-entropy-formula}:
\begin{lem}
\label{lem:ent of comp in act conv}Let $R>0$, $\theta\in\mathcal{P}(A_{2,2})$
supported within distance $R$ of the identity, and $\nu\in\mathcal{P}(\mathbb{R}^{2})$
supported within distance $R$ of the origin. Then for every $1\le k\le n$,
\[
\frac{1}{n}H(\theta\conv\nu,\mathcal{D}_{n})\ge\mathbb{E}_{1\le i\le n}\left(\frac{1}{k}H(\theta_{\psi,i}\conv\nu_{x,i},\mathcal{D}_{i+k})\right)-O_{R}(\frac{k}{n}+\frac{1}{k})\;.
\]
\end{lem}

\begin{proof}
Let $\ell$ be the integral part of $\frac{n}{k}$. As in the proof
of \cite[Lemma 4.3]{BHR}, for each $0\le r<k$
\[
H(\theta\conv\nu,\mathcal{D}_{n})\ge\sum_{m=0}^{\ell-2}\mathbb{E}_{i=mk+r}\left(H(\theta_{\psi,i}\conv\nu_{x,i},\mathcal{D}_{k+i}\mid\mathcal{D}_{i})\right)\:.
\]
Note that
\[
\diam(\supp((\theta_{\psi,i})\conv\nu_{x,i}))=O_{R}(2^{-i})\:.
\]
Hence $\supp(\theta_{\psi,i}\conv\nu_{x,i})$ intersects $O_{R}(1)$
elements of $\mathcal{D}_{i}$, and so
\[
H(\theta\conv\nu,\mathcal{D}_{n})\ge\sum_{m=0}^{\ell-2}\mathbb{E}_{i=mk+r}\left(H(\theta_{\psi,i}\conv\nu_{x,i},\mathcal{D}_{k+i})\right)-O_{R}(\ell)\:.
\]
The rest of the proof proceeds exactly as in \cite[Lemma 4.3]{BHR}.
\end{proof}

\subsection{\label{subsec:Proof-of-ent growth thm}Entropy growth near the identity}

Out main goal in this section is to prove our main entropy growth
result, Theorem \ref{thm:entropy growth-near-identity}. We recall
the statement:
\begin{thm*}
Let $\mu$ be a self-affine measure in $\mathbb{R}^{2}$ defined by
a non-conformal, totally irreducible system $\Phi$ and satisfying
$\dim\mu<2$. Then for every $\varepsilon,R>0$ there is a $\delta=\delta(\mu,\varepsilon,R)>0$
such that for every $n>N(\mu,\varepsilon,R)$, the following holds:

If $\theta$ is a probability measure on the affine group supported
within distance $R$ of the identity, then 
\[
\frac{1}{n}H(\theta,\mathcal{D}_{n})>\varepsilon\qquad\implies\qquad\frac{1}{n}H(\theta\conv\mu,\mathcal{D}_{n})>\frac{1}{n}H(\mu,\mathcal{D}_{n})+\delta\;.
\]
\end{thm*}
We begin the proof.

Recall from Section \ref{sub:q-adic-partitions} the definition $\mathcal{N}_{n}=\{1,\ldots,n\}$
and $\mathcal{N}_{n,n+k}=\{n,n+1,\ldots,n+k\}$ with the associated
uniform measures $\lambda_{n}$ and $\lambda_{n,n+k}$ on them.

Let $0<\varepsilon<1$ and $R>0$, let $k\ge1$ be large with respect
to $\varepsilon,R$, and let $n\ge1$ be large with respect to $k$.
Let $\theta\in\mathcal{P}(A_{2,2})$ be supported within $R$ of the
identity in $A_{2,2}$, and assume that $\frac{1}{n}H(\theta,\mathcal{D}_{n})>\varepsilon$. 

By Lemma \ref{lem:comp not conc} and by replacing $\varepsilon$
with a smaller quantity without changing the notation, we may assume
that $\lambda_{n}\times\theta(F_{0})>\varepsilon$, where $F_{0}$
is the set of all $(i,\psi)\in\mathcal{N}_{n}\times A_{2,2}$ such
that
\[
\mathbb{P}_{i\le j\le i+k}\left(\mu\left\{ x\::\:\begin{array}{c}
S_{2^{j}}((\psi^{-1}\theta_{\psi,i})_{\varphi,j})\conv x\text{ is}\\
\mbox{not }(\{0\},\varepsilon)\mbox{-concentrated}
\end{array}\right\} >\varepsilon\right)>\varepsilon\:.
\]

Let $\delta>0$ be small with respect to $\varepsilon,R$ and suppose
that $k$ is large with respect to $\delta$. By Lemma \ref{lem:ent of comp in act conv},
\begin{align*}
\frac{1}{n}H(\theta\conv\mu,\mathcal{D}_{n}) & \ge\mathbb{E}_{1\le i\le n}\left(\frac{1}{k}H(\theta_{\psi,i}\conv\mu_{x,i},\mathcal{D}_{i+k})\right)-O_{R}(\frac{k}{n}+\frac{1}{k})\\
 & \ge\mathbb{E}_{1\le i\le n}\left(\frac{1}{k}H(\theta_{\psi,i}\conv\mu_{x,i},\mathcal{D}_{i+k})\right)-\frac{\delta^{2}}{5}\:.
\end{align*}
By this and Theorem \ref{thm:linearization},
\[
\frac{1}{n}H(\theta\conv\mu,\mathcal{D}_{n})\ge\mathbb{E}_{1\le i\le n}\left(\frac{1}{k}H((\theta_{\psi,i}\conv x)*\psi\mu_{x,i},\mathcal{D}_{i+k})\right)-\frac{2\delta^{2}}{5}\:.
\]
Since $\theta$ is supported on an $R$-neighborhood of the identity
the partitions $\mathcal{D}_{i+k}$ and $\psi^{-1}\mathcal{D}_{i+k}$
are $O_{R}(1)$-commensurable, so taking $k$ large relative to $R$
and $\delta$ we get
\begin{equation}
\frac{1}{n}H(\theta\conv\mu,\mathcal{D}_{n})\ge\mathbb{E}_{1\le i\le n}\left(\frac{1}{k}H((\psi^{-1}\theta_{\psi,i}\conv x)*\mu_{x,i},\mathcal{D}_{i+k})\right)-\frac{3\delta^{2}}{5}\:.\label{eq:lb by conv of comp}
\end{equation}

Write $\Gamma=\lambda_{n}\times\mu\times\theta$ and set,
\[
E_{0}=\left\{ (i,x,\psi)\in\mathcal{N}_{n}\times\mathbb{R}^{2}\times A_{2,2}\::\begin{array}{c}
\frac{1}{k}H((\psi^{-1}\theta_{\psi,i}\conv x)*\mu_{x,i},\mathcal{D}_{i+k})\\
\qquad<\frac{1}{k}H(\mu_{x,i},\mathcal{D}_{i+k})+\delta
\end{array}\right\} \:.
\]
Assuming as we are that $k$ large relative to $\delta$, we have
\begin{equation}
\frac{1}{k}H((\psi^{-1}\theta_{\psi,i}\conv x)*\mu_{x,i},\mathcal{D}_{i+k})\ge\frac{1}{k}H(\mu_{x,i},\mathcal{D}_{i+k})-\frac{\delta^{2}}{10}\:.\label{eq:by conc we always have}
\end{equation}
By $\dim\mu=\alpha$ and by Lemmas \ref{lem:entropy-dimension} and
\ref{lem:multiscale-entropy-formula}, since $n$ is large,
\begin{equation}
\mathbb{E}_{1\le i\le n}\left(\frac{1}{k}H(\mu_{x,i},\mathcal{D}_{i+k})\right)\ge\alpha-\frac{\delta^{2}}{5}\:.\label{eq:alpha is lb}
\end{equation}
Now if $\Gamma(E_{0})\le1-\delta$, then by (\ref{eq:lb by conv of comp}),
(\ref{eq:by conc we always have}), and (\ref{eq:alpha is lb}),
\[
\frac{1}{n}H(\theta\conv\mu,\mathcal{D}_{n})\ge\mathbb{E}_{1\le i\le n}\left(\frac{1}{k}H(\mu_{x,i},\mathcal{D}_{i+k})\right)+\delta\Gamma(E_{0}^{c})-\frac{7\delta^{2}}{10}\ge\alpha+\frac{\delta^{2}}{10}\:,
\]
which completes the proof of the Theorem. Hence it suffices to prove
that $\Gamma(E_{0})\le1-\delta$.

Assume by contradiction that $\Gamma(E_{0})>1-\delta$. Let $\sigma>0$
be small with respect to $\varepsilon,R$ and suppose that $\delta$
is small with respect to $\sigma$. Let $m\ge1$ be large with respect
to $\sigma$ and suppose that $\delta$ is small with respect to $m$.
By Theorem \ref{thm:Hochman's entropy growth} it follows that for
each $u=(i,x,\psi)\in E_{0}$ there exist linear subspaces $V_{i}^{u},...,V_{i+k}^{u}\subset\mathbb{R}^{2}$
such that\footnote{In (\ref{eq:cor of hochman's thm}) and later, $x,\psi$ and $i$
are fixed, and the randomness is over $y,z$ and $j$.}
\begin{equation}
\mathbb{P}_{i\le j\le i+k}\left(\begin{array}{c}
(\mu_{x,i})^{y,j}\mbox{ is }(V_{j}^{u},\sigma,m)\mbox{-saturated and}\\
(\psi^{-1}\theta_{\psi,i}\conv x)^{z,j}\mbox{ is }(V_{j}^{u},\sigma)\mbox{-concentrated}
\end{array}\right)>1-\sigma\:.\label{eq:cor of hochman's thm}
\end{equation}
\begin{lem}
We can assume that $\Gamma(E_{1})>1-\sigma$, where $E_{1}$ is the
set of all $(i,x,\psi)\in\mathcal{N}_{n}\times\mathbb{R}^{2}\times A_{2,2}$
with
\begin{equation}
\mathbb{P}_{i\le j\le i+k}\left((\psi^{-1}\theta_{\psi,i}\conv x)^{z,j}\mbox{ is }(L(x),\sigma)\mbox{-concentrated}\right)>1-\sigma\:.\label{eq:concentration near L(x)}
\end{equation}
\end{lem}

\begin{proof}
Let $Z$ be the set of all $(i,x,\psi)\in\mathcal{N}_{n}\times\mathbb{R}^{2}\times A_{2,2}$
such that,
\[
\mathbb{P}_{i\le j\le i+k}\left(\left|H_{m}\left((\mu_{x,i})^{y,j}\right)-\alpha\right|<\sigma\right)>1-\sigma/2\:.
\]
Then by Proposition \ref{prop:uniform ent dim} and Lemma \ref{lem:distribution-of-components-of-components}
it follows that $\Gamma(Z)>1-\sigma$. By Lemma \ref{lem:proj of comp of comp}
it follows that $\Gamma(Y)>1-\sigma$, where $Y$ is the set of all
$(i,x,\psi)$ with
\[
\mathbb{P}_{i\le j\le i+k}\left(\underset{W\notin B(L(x),\sigma)}{\inf}\;H_{m}(\pi_{W^{\perp}}((\mu_{x,i})^{y,j}))>\beta-\sigma\right)>1-\sigma\:.
\]
Note that $\Gamma(E_{0}\cap Z\cap Y)>1-3\sigma$, hence it suffices
to show that (\ref{eq:concentration near L(x)}) is satisfied for
$(i,x,\psi)\in E_{0}\cap Z\cap Y$ with $\sigma$ replaced by $O(\sigma)$.

Fix $u=(i,x,\psi)\in E_{0}\cap Z\cap Y$ and let $F_{u}$ be the set
of all $(j,y)\in\mathcal{N}_{i,i+k}\times\mathbb{R}^{2}$ such that,
\begin{itemize}
\item $(\mu_{x,i})^{y,j}\mbox{ is }(V_{j}^{u},\sigma,m)$-saturated;
\item $\left|H_{m}\left((\mu_{x,i})^{y,j}\right)-\alpha\right|<\sigma$;
\item $\underset{W\notin B(L(x),\sigma)}{\inf}\;H_{m}(\pi_{W^{\perp}}((\mu_{x,i})^{y,j}))>\beta-\sigma$.
\end{itemize}
Since $u\in E_{0}\cap Z\cap Y$ we have $\nu_{i,i+k}\times\mu_{x,i}(F_{u})>1-3\sigma$.
Let $(j,y)\in F_{u}$ and assume by contradiction that $\dim V_{j}^{u}=2$
or $\dim V_{j}^{u}=1$ with $V_{j}^{u}\notin B(L(x),\sigma)$, then
\begin{align}
\alpha & >H_{m}\left((\mu_{x,i})^{y,j}\right)-\sigma\nonumber \\
 & \ge\dim V_{j}^{u}+H_{m}(\pi_{(V_{j}^{u})^{\perp}}(\mu_{x,i})^{y,j})-2\sigma\label{eq:triangular-issue}\\
 & >1+\beta-3\sigma\:.\nonumber 
\end{align}
We have assumed that $0<\alpha<2$, and by\footnote{In fact here we only want $\geq\alpha/2$, not $\geq\alpha/2+\tau$,
so this is a much easier result which does not require Bourgain's
theorem.} Corollary \ref{cor:bourgain-projection-of-SA-measures} we have $\beta\geq\frac{1}{2}\alpha$,
hence, by assuming that $\sigma$ is small enough, we get a contradiction.
It follows that we must have,
\begin{equation}
\dim V_{j}^{u}=0\text{ or }\dim V_{j}^{u}=1\text{ with }V_{j}^{u}\in B(L(x),\sigma)\:.\label{eq:0 dim or 1 dim and in ball}
\end{equation}

Write
\[
S=\{j\in\mathcal{N}_{i,i+k}\::\:\mu_{x,i}\{y\::\:(j,y)\in F_{u}\}>0\},
\]
then $\nu_{i,i+k}(S)>1-3\sigma$ since $\nu_{i,i+k}\times\mu_{x,i}(F_{u})>1-3\sigma$.
Note that (\ref{eq:0 dim or 1 dim and in ball}) holds for each $j\in S$.
Let $(j,z)\in\mathcal{N}_{i,i+k}\times\mathbb{R}^{2}$ be such that
$j\in S$ and $\nu:=(\psi^{-1}\theta_{\psi,i}\conv x)^{z,j}$ is $(V_{j}^{u},\sigma)$-concentrated.
If $\dim V_{j}^{u}=0$ then $\nu$ is clearly $(L(x),\sigma)$-concentrated.
If $\dim V_{j}^{u}=1$ with $V_{j}^{u}\in B(L(x),\sigma)$ then $\nu$
is $(L(x),O(\sigma))$-concentrated. Hence in any case we have that
$\nu$ is $(L(x),O(\sigma))$-concentrated. From this, $\nu_{i,i+k}(S)>1-3\sigma$,
and (\ref{eq:cor of hochman's thm}), it follows that (\ref{eq:concentration near L(x)})
is satisfied for $u=(i,x,\psi)$ with $\sigma$ replaced by $O(\sigma)$.
This completes the proof of the lemma.
\end{proof}
\begin{lem}
We can assume that $\Gamma(E_{2})>1-\sigma$, where $E_{2}$ is the
set of all $(i,x,\psi)\in\mathcal{N}_{n}\times\mathbb{R}^{2}\times A_{2,2}$
with
\begin{equation}
\mathbb{P}_{i\le j\le i+k}\left(\begin{array}{c}
S_{2^{j}}((\psi^{-1}\theta_{\psi,i})_{\varphi,j}\conv x)\mbox{ is }\\
(L(x),\sigma)\mbox{-concentrated}
\end{array}\right)>1-\sigma\:.\label{eq:conc in A_2,2 near L(x)}
\end{equation}
\end{lem}

\begin{proof}
Fix $(i,x,\psi)\in E_{1}$ with $x\in X$, write $\tau=\psi^{-1}\theta_{\psi,i}$,
and set
\[
S=\{j\in\mathcal{N}_{i,i+k}\::\:\mathbb{P}_{l=j}\left((\tau\conv x)^{y,l}\mbox{ is }(L(x),\sigma)\mbox{-concentrated}\right)\ge1-\sqrt{\sigma}\}\:.
\]
By (\ref{eq:concentration near L(x)}) it follows that $\nu_{i,i+k}(S)\ge1-\sqrt{\sigma}$.
Let $\sigma'>0$ be small with respect to $\varepsilon>0$ and suppose
that $\sigma$ is small with respect to $\sigma'$. By Proposition
\ref{prop:from conc on eu to cont on A_2,2} there exists an integer
$q=q(\sigma')\ge1$ such that, by assuming that $\sigma$ is small
enough with respect to $\sigma'$, we have
\begin{equation}
\mathbb{P}_{j\le l\le j+q}\left(S_{2^{l}}\left(\tau_{\varphi,l}\conv x\right)\text{ is }(L(x),\sigma')\text{-concentrated}\right)\ge1-\sigma'\text{ for }j\in S\:.\label{eq:conc for j in S}
\end{equation}
Let $\sigma''>0$ be small with respect to $\varepsilon>0$ and suppose
that $\sigma'$ is small with respect to $\sigma''$. From $\nu_{i,i+k}(S)\ge1-\sqrt{\sigma}$
and (\ref{eq:conc for j in S}), by assuming that $\sigma,\sigma'$
are sufficiently small with respect to $\sigma''$, and by assuming
that $k$ is sufficiently large with respect to $q$, it follows by
a statement similar to Lemma \ref{lem:distribution-of-components-of-components}
that (\ref{eq:conc in A_2,2 near L(x)}) is satisfied with $\sigma''$
in place of $\sigma$. This completes the proof of the lemma.
\end{proof}
By the previous lemma, by Fubini's theorem, and by replacing $\sigma$
with a larger quantity which is still small with respect to $\varepsilon$
(without changing the notation), we may assume that $\lambda_{n}\times\theta(F_{1})>1-\sigma$,
where $F_{1}$ is the set of all $(i,\psi)\in\mathcal{N}_{n}\times A_{2,2}$
such that
\[
\mathbb{P}_{i\le j\le i+k}\left(\mu\left\{ x\::\:\begin{array}{c}
S_{2^{j}}((\psi^{-1}\theta_{\psi,i})_{\varphi,j})\conv x\mbox{ is }\\
(L(x),\sigma)\mbox{-concentrated}
\end{array}\right\} >1-\sigma\right)>1-\sigma\:.
\]

Recall the set $F_{0}$ from the beginning of the proof. Since $\sigma$
is small with respect to $\varepsilon$, $\lambda_{n}\times\theta(F_{0})>\varepsilon$,
and $\lambda_{n}\times\theta(F_{1})>1-\sigma$, we have $\lambda_{n}\times\theta(F_{0}\cap F_{1})>0$.
In particular there exists $(i,\psi)\in F_{0}\cap F_{1}$. Similarly,
since $\sigma$ is small with respect to $\varepsilon$, there exist
$i\le j\le i+k$ and $\varphi\in A_{2,2}$ such that for $\theta':=S_{2^{j}}((\psi^{-1}\theta_{\psi,i})_{\varphi,j})$
we have,
\begin{equation}
\mu\left\{ x\::\:\theta'\conv x\mbox{ is }(L(x),\sigma)\mbox{-concentrated}\right\} >1-\sigma\label{eq:conc on L(x)}
\end{equation}
and
\begin{equation}
\mu\left\{ x\::\:\theta'\conv x\mbox{ is not }(\{0\},\varepsilon)\mbox{-concentrated}\right\} >\varepsilon\:.\label{eq:not conc on =00007B0=00007D}
\end{equation}
Also, observe that $\theta'$ is the re-scaling by $2^{j}$ of a level-$j$
component $(\psi^{-1}\theta_{\psi,i})_{\varphi,j}$ of the measure
$\psi^{-1}\theta_{\psi,i}$, and $\psi^{-1}\theta_{\psi,i}$ is contained
in an $O(1)$-ball (with respect to the invariant metric $d$) around
the identity. On the intersection of $A_{2,2}$ with this ball, the
invariant metric and the norm metric of $A_{2,2}^{vec}$ are bi-Lipschitz
equivalent. The diameter of the support of $(\psi^{-1}\theta_{\psi,i})_{\varphi,j}$
is $O(2^{-j})$ in the invariant metric, so it also has diameter $O(2^{-j})$
in norm; hence after re-scaling by $2^{j}$, the diameter of the support
of $\theta'$ is $O(1)$ with respect to the norm metric. 

In view of the last few paragraphs, and since $\sigma$ can be taken
arbitrarily small compared to $\varepsilon$, we have a contradiction
to Corollary \ref{cor:non-concentration-of-L}. This completes the
proof of the theorem.

Finally, we prove the more basic fact that entropy does not decrease
(a special case of which is (\ref{eq:entorpy-increase-by-convolution})):
\begin{prop}
\label{prop:entropy-non-decrease-near-identity}Let $R>0$ and let
$\nu\in\mathcal{P}(\mathbb{R}^{2})$, $\theta\in\mathcal{P}(A_{2,2})$
be supported on $R$-neighborhoods of the identities of $\mathbb{R}^{2},A_{2,2}$,
respectively. Then for every $n$,
\[
H(\theta\conv\nu,\mathcal{D}_{n})\geq H(\nu,\mathcal{D}_{n})+O_{R}(1)\;.
\]
\end{prop}

\begin{proof}
Every $h\in\supp\theta$ is bi-Lipschitz with constant $O_{R}(1)$,
hence $H(h\nu,\mathcal{D}_{n})=H(\nu,\mathcal{D}_{n})+O_{R}(1)$.
Thus, using $\theta\conv\nu=\int h\nu\,d\theta(h)$ and concavity
of entropy,
\begin{align*}
H(\theta\conv\nu,\mathcal{D}_{n}) & =H(\int h\nu\,d\theta(h),\mathcal{D}_{n})\\
 & \geq\int H(h\nu,\mathcal{D}_{n})\,d\theta(h)\\
 & \geq H(\nu,\mathcal{D}_{n})+O_{R}(1)\;.\qedhere
\end{align*}
\end{proof}

\section{\label{sec:non-conformal-partitions}The non-conformal partitions
$\mathcal{D}_{n}^{g}$ and entropy growth}

In this section we assume everything: namely, that $\Phi$ is non-conformal,
totally irreducible and exponentially separated, and that $\dim\mu\geq1$. 

Our objective in this section is to prove an entropy growth result
for $\theta\conv\mu$, when $\theta$ is far from the identity, but
still of bounded diameter. It is important to notice that entropy
can even decrease under such a convolution if we do not measure it
in the right way. Indeed, consider the matrix $A=\diag(1,2^{-n})$
for some large $n$. Then at resolution $2^{-n}$ (corresponding to
$\mathcal{D}_{n}$), the measure $A\mu$ is extremely close to being
supported on a horizontal line, hence $\frac{1}{n}H(A\mu,\mathcal{D}_{n})\leq1+o(1)$.
If $\theta$ were supported on a bounded neighborhood of $A$ then,
no matter how smooth $\theta$ is, we would similarly have 
\[
\frac{1}{n}H(\theta\conv\mu,\mathcal{D}_{n})\leq1+o(1)
\]
since $\theta\conv\mu$ is still close to a horizontal line. At the
same time, if $\dim\mu>1+\delta$, then we will have 
\[
\frac{1}{n}H(\mu,\mathcal{D}_{n})=\dim\mu-o(1)>1+\delta-o(1)
\]
. Thus, for large $n$ we certainly have $\frac{1}{n}H(\theta\conv\mu,\mathcal{D}_{n})\leq\frac{1}{n}H(\mu,\mathcal{D}_{n})-\delta$,
which even gives an entropy decrease.

The problem is, of course, that we are measuring entropy in the wrong
coordinates. The right way is in the coordinates induced by $A$:
Let $Ax+a=g(x)\in A_{2,2}$ and let $VDU$ be a singular value decomposition
of $A$. Assume that $\alpha_{1}(A)>\alpha_{2}(A)$, where $\alpha_{1}(A),\alpha_{2}(A)$
are the singular values of $A$. For $n\ge0$ we set 
\begin{equation}
\mathcal{D}_{n}^{g}=VD(\mathcal{D}_{n})\;.\label{eq:D-n-g}
\end{equation}
With respect to this partition, one does not have an entropy drop
from $\mu$ to $\theta\conv\mu$. Furthermore, under our assumptions
on $\mu$, we will be able to interpolate between $\mathcal{D}_{n}^{g}$
and ordinary dyadic partitions at appropriate scales, to show that
entropy growth generally does occur.

\subsection{\label{subsec:Interpolating-between-non-conformal-partitions}Interpolating
between non-conformal and conformal partitions}

The purpose of this section is to relate the entropy of a measure
with respect to $\mathcal{D}_{n}^{g}$ to the entropy with respect
to the usual partitions $\mathcal{D}_{n}$. This relies on analysis
of projections of the measure, and therefore requires the assumptions
stated at the start of the section, which, by Theorem \ref{thm:BHR-projections},
imply that
\[
\dim\pi_{V}\mu=1\qquad\text{for }\eta^{*}\text{-all }V\in\PR\;.
\]

In this section we fix the following notation. Let $g\in A_{2,2}$
and recall that we write $g(x)=A_{g}x+b_{g}$. Let $n\in\mathbb{N}$,
and denote the singular values of $A_{g}$ by $\alpha_{1}=\alpha_{1}(A_{g})=2^{-c_{1}n}$
and $\alpha_{2}=\alpha_{2}(A_{g})=2^{-c_{2}n}$, with $0<c_{1}<c_{2}$
(we introduce $n$ because later we will consider $c_{1},c_{2}$ fixed
and $n\rightarrow\infty$; one may imagine that $c_{i}=|\chi_{i}|$).
Let $A_{g}=VDU$ be the singular value decomposition of $A_{g}$,
and recall that $\mathcal{D}_{n}^{g}=VD\mathcal{D}_{n}$, so it consists
of rectangular cells whose long edge has direction $\overline{v}=\overline{Ve_{1}}$
and length $2^{-(1+c_{1})n}$, and whose short edge has direction
$\overline{v}^{\perp}$ and length $2^{-(1+c_{2})n}$. 

As a first consequence observe that for any $M\geq0$, and up to a
translation, $\mathcal{D}_{(M+c_{2})n}^{\overline{v}\oplus\overline{v}^{\perp}}$
refines $\mathcal{D}_{Mn}^{g}$; and in fact, 
\[
\mathcal{D}_{Mn}^{g}\lor\pi_{\overline{v}}^{-1}\mathcal{D}_{(M+c_{2})n}\text{ is commensurable with }\mathcal{D}_{(M+c_{2})n}\;.
\]

It follows that for any measure $\nu\in\mathcal{P}(\mathbb{R}^{2})$,
and for $M\geq0$, 
\begin{align}
H(\nu,\mathcal{D}_{(M+c_{2})n}|\mathcal{D}_{c_{2}n}) & =H(\nu,\mathcal{D}_{(M+c_{2})n})-H(\nu,\mathcal{D}_{c_{2}n})\nonumber \\
 & =H(\nu,\mathcal{D}_{Mn}^{g}\lor\pi_{\overline{v}}^{-1}\mathcal{D}_{(M+c_{2})n})-H(\nu,\mathcal{D}_{0}^{g}\lor\pi_{\overline{v}}^{-1}\mathcal{D}_{c_{2}n})\pm O(1)\nonumber \\
 & =\left(H(\nu,\mathcal{D}_{Mn}^{g})+H(\nu,\pi_{\overline{v}}^{-1}\mathcal{D}_{(M+c_{2})n}|\mathcal{D}_{Mn}^{g})\right)\nonumber \\
 & \qquad-\left(H(\nu,\mathcal{D}_{0}^{g})+H(\nu,\pi_{\overline{v}}^{-1}\mathcal{D}_{c_{2}n}|\mathcal{D}_{0}^{g})\right)\pm O(1)\;.\label{eq:nonconformal-decomposition}
\end{align}
\begin{lem}
Let $R>1$, let $g\in A_{2,2}$ be as above, and suppose that $c_{2}-c_{1}>R^{-1}$.
Let $\theta\in\mathcal{P}(A_{2,2})$ be supported in an $R$-neighborhood
of $g$ (with respect to the invariant metric). Let $\nu=\theta\conv\mu$,
where $\mu$ is a self-affine measure generated by a non-conformal
and totally irreducible system satisfying exponential separation and
$\dim\mu\geq1$. Then
\[
H(\nu,\mathcal{D}_{0}^{g})=O_{R}(1),
\]
and for all $M\in\{0\}\cup[1,\infty)$,
\[
H(\nu,\pi_{\overline{v}}^{-1}\mathcal{D}_{(M+c_{2})n}|\mathcal{D}_{Mn}^{g})=(c_{2}-c_{1})n+o_{R}(n)\;.
\]
\end{lem}

\begin{proof}
We prove the second statement first and adopt the notation from the
previous discussion. Since $\mathcal{D}_{Mn}^{g}$ consists of rectangles
of dimensions $2^{-(M+c_{1})n}\times2^{-(M+c_{2})n}$ with long edge
in direction $\overline{v}$, and since $\pi_{\overline{v}}^{-1}\mathcal{D}_{(M+c_{2})n}$
consists of strips of width $2^{-(M+c_{2})n}$ in direction $\overline{v}^{\perp}$,
every cell of the former partition is divided by the latter partition
into $O(2^{(c_{2}-c_{1})n})$ cells. Therefore we have the trivial
bound
\[
H(\nu,\pi_{\overline{v}}^{-1}\mathcal{D}_{(M+c_{2})n}|\mathcal{D}_{Mn}^{g})\leq(c_{2}-c_{1})n+O(1)\;.
\]

To prove the reverse inequality, use $\nu=\theta\conv\mu=\int h\mu\,d\theta(h)$
and concavity of entropy to conclude that
\begin{equation}
H(\nu,\pi_{\overline{v}}^{-1}\mathcal{D}_{(M+c_{2})n}|\mathcal{D}_{Mn}^{g})\geq\int H(h\mu,\pi_{\overline{v}}^{-1}\mathcal{D}_{(M+c_{2})n}|\mathcal{D}_{Mn}^{g})\,d\theta(h)\;,\label{eq:transforming-measure-by-h}
\end{equation}
so it is enough to prove the lower bound for the integrand on the
right hand side, under the assumption that $d(h,g)=O_{R}(1)$. Recall
that $A_{g}=VDU$ is the singular value decomposition of $A_{g}$,
so that $\mathcal{D}_{Mn}^{g}=VD\mathcal{D}_{Mn}$. By assumption,
we can write $h=gh'$ with $d(h',\id)=O_{R}(1)$, and therefore $h=VDUA_{h'}+gb_{h'}=VDh''+gb_{h'}$,
where we have defined $h''=UA_{h'}$. Note that $h''$ lies in an
$O_{R}(1)$ neighborhood of the identity. Substituting this into (\ref{eq:transforming-measure-by-h}),
and eliminating the translation $gb_{h'}$ at the expense of absorbing
an additive $O(1)$ term into the $o(n)$ term, we see that it is
enough to show that
\[
H(VD(h''\mu),\pi_{\overline{v}}^{-1}\mathcal{D}_{(M+c_{2})n}|VD\mathcal{D}_{Mn})\geq(c_{2}-c_{1})n+o(n)\;.
\]
Applying $(VD)^{-1}$ to all terms, this is the same as
\[
H(h''\mu,(VD)^{-1}\pi_{\overline{v}}^{-1}\mathcal{D}_{(M+c_{2})n}|\mathcal{D}_{Mn})\geq(c_{2}-c_{1})n+o(n)\;.
\]
Now, $(VD)^{-1}\pi_{\overline{v}}^{-1}=(\pi_{\overline{v}}VD)^{-1}=(\pi_{\overline{e}_{1}}D)^{-1}=\pi_{\overline{e}_{1}}^{-1}S_{2^{c_{1}n}}$
(because $\overline{v}=\overline{Ve_{1}}$ and $D^{-1}=\diag(2^{c_{1}n},2^{c_{2}n})$),
so we must show that
\[
H(h''\mu,\pi_{e_{1}}^{-1}\mathcal{D}_{(M+c_{2}-c_{1})n}|\mathcal{D}_{Mn})\geq(c_{2}-c_{1})n+o(n)\;.
\]
For $M\ge1$ this is a consequence of Proposition \ref{prop:lb on ent of proj of comp of mu}.
For $M=0$ this follows easily from Lemma \ref{lem:uni conv of ent}
and $d(h'',\id)=O_{R}(1)$.

The first statement is proved similarly: first write $\theta=g\theta'$,
with $\theta'\in\mathcal{P}(A_{2,2})$ supported in an $O_{R}(1)$
neighborhood of the identity. Write $\mu'=\theta'\conv\mu$, so $\nu=g\mu'$.
Then, by the same reasoning as above, for some map $h''\in A_{2,2}$
within distance $O_{R}(1)$ of the identity, we have 
\[
H(\nu,\mathcal{D}_{0}^{g})=H(h''\mu',\mathcal{D}_{0})=O_{R}(1)\;,
\]
where the last bound is because $\mu'$, and hence $h''\mu'$, is
supported on a set of diameter $O_{R}(1)$.
\end{proof}
\begin{prop}
\label{prop:interpolation}Let $R>1$, let $\theta\in\mathcal{P}(A_{2,2})$
be supported on a set of diameter $R$ (in the invariant metric),
and let $g\in\supp\theta$. Let $2^{-c_{2}n}<2^{-c_{1}n}<1$ denote
the singular values of $A_{g}$ and suppose that $c_{2}-c_{1}>R^{-1}$.
Then for every $M\geq1$,
\[
H(\theta\conv\mu,\mathcal{D}_{(M+c_{2})n}|\mathcal{D}_{c_{2}n})=H(\theta\conv\mu,\mathcal{D}_{Mn}^{g})+o_{R}(n)\;.
\]
\end{prop}

\begin{proof}
By equation (\ref{eq:nonconformal-decomposition}), the claim follows
if we show that
\[
H(\theta\conv\mu,\pi_{\overline{v}}^{-1}\mathcal{D}_{(M+c_{2})n}|\mathcal{D}_{Mn}^{g})-H(\theta\conv\mu,\mathcal{D}_{0}^{g})-H(\theta\conv\mu,\pi_{\overline{v}}^{-1}\mathcal{D}_{c_{2}n}|\mathcal{D}_{0}^{g})=o(n)\;.
\]
This, in turn, follows from the previous lemma, which says that the
two extreme terms are $(c_{2}-c_{1})n+o(n)$, so these cancel up to
an $o(n)$ error, and the middle term is $O(1)$. 
\end{proof}

\subsection{\label{subsec:Entropy-growth-wrt-Dng}Entropy growth far from the
identity}

We can now prove our entropy growth results for $\theta\conv\mu$
when $\theta$ is far from the identity, but still of bounded diameter. 
\begin{thm}
\label{thm:entropy growth} Let $\mu$ be a self-affine measure in
$\mathbb{R}^{2}$ defined by a non-conformal, totally irreducible
system $\Phi$ and satisfying $\dim\mu<2$. Then for every $\varepsilon>0$
and $R>1$ there exists $\delta=\delta(\mu,\varepsilon,R)>0$ such
that for $n\ge N(\mu,\varepsilon,R)$, the following holds. 

Let $\theta\in\mathcal{P}(A_{2,2})$ be supported in an $R$-neighborhood
of a contraction $g\in A_{2,2}$. Then,
\[
\frac{1}{n}H(\theta,\mathcal{D}_{n})>\varepsilon\qquad\implies\qquad\frac{1}{n}H(\theta\conv\mu,\mathcal{D}_{n}^{g})>\dim\mu+\delta\;.
\]
Furthermore, if we also assume exponential separation and $\dim\mu\geq1$,
then for any $M\ge1$, writing $a_{i}=\frac{1}{n}\log\alpha_{i}(A_{g})$
for $i=1,2$ and assuming $a_{1}-a_{2}>R^{-1}$,
\[
\frac{1}{Mn}H(\theta,\mathcal{D}_{Mn})>\varepsilon\qquad\implies\qquad\frac{1}{Mn}H(\theta\conv\mu,\mathcal{D}_{(M+|a_{2}|)n}|\mathcal{D}_{|a_{2}|n})>\dim\mu+\delta\;.
\]
\end{thm}

\begin{proof}
The argument is identical to the previous proposition except that
instead of concavity we apply Theorem \ref{thm:entropy growth-near-identity}.
In detail, let $g(x)=Ax+b$ and $A=VDU$ be the singular value decomposition.
Let $B=VD$ so that $\mathcal{D}_{n}^{g}=B\mathcal{D}_{n}$. We claim
that the statement follows from Theorem \ref{thm:entropy growth-near-identity}
applied to $\mu$ and the measure $\theta'$ obtained by translating
$B^{-1}\theta$ by $-B^{-1}b$. Indeed, by left-invariance of $d$,
\[
|H(\theta',\mathcal{D}_{n})-H(\theta,\mathcal{D}_{n})|=O(1)\;.
\]
Also, again by left-invariance, $\theta'$ is supported on an $R$-neighborhood
of $B^{-1}g-B^{-1}b=U$, and since $U$ lies in the compact (and hence
bounded) group of orthogonal matrices, $\theta'$ is supported in
a $(R+c)$-neighborhood of the identity in $A_{2,2}$, where the constant
$c$ is the diameter of the orthogonal group of $\mathbb{R}^{2}$.
By Theorem \ref{thm:entropy growth-near-identity} we obtain that
for some $\delta>0$, for $n$ large enough,
\[
\frac{1}{n}H(\theta'\conv\mu,\mathcal{D}_{n})\geq\dim\mu+\delta\;.
\]
Finally, we have
\[
H(\theta\conv\mu,\mathcal{D}_{n}^{g})=H(\theta\conv\mu,B\mathcal{D}_{n})=H(B^{-1}(\theta\conv\mu),\mathcal{D}_{n})=H(\theta'\conv\mu,\mathcal{D}_{n})+O(1)\:,
\]
which completes our proof of the first part. The second part follows
from Proposition \ref{prop:interpolation} and from the first part
of the present theorem (using $Mn$ in place of $n$).
\end{proof}
Finally, we have the softer fact that entropy can never substantially
decrease under convolution (if measured at appropriate scales). 
\begin{prop}
\label{prop:entropy-non-decrease}Let $\mu$ be a self-affine measure
in $\mathbb{R}^{2}$ defined by a non-conformal, totally irreducible
system $\Phi$. For every $R>1$, if $n>N(R)$, the following holds.

Let $\theta\in\mathcal{P}(A_{2,2})$ be supported in an $R$-neighborhood
of a contraction $g\in A_{2,2}$. Then, as $n\rightarrow\infty$,
\[
\frac{1}{n}H(\theta\conv\mu,\mathcal{D}_{n}^{g})\geq\dim\mu-o_{R}(1)\;.
\]
Furthermore, if we also assume exponential separation and $\dim\mu\geq1$,
then for any $M\geq1$, writing $a_{i}=\frac{1}{n}\log\alpha_{i}(A_{g})$
for $i=1,2$ and assuming $a_{1}-a_{2}>R^{-1}$, as $n\rightarrow\infty$,
\[
\frac{1}{Mn}H(\theta\conv\mu,\mathcal{D}_{(M+|a_{2}|)n}|\mathcal{D}_{|a_{2}|n})\geq\dim\mu-o_{R}(1)\;.
\]
\end{prop}

\begin{proof}
We observe $g^{-1}\theta$ is supported on a an $R$-neighborhood
of the identity and apply Proposition \ref{prop:entropy-non-decrease-near-identity}
to get,
\begin{align*}
\frac{1}{n}H(\theta\conv\mu,\mathcal{D}_{n}^{g}) & =\frac{1}{n}H(g^{-1}\theta\conv\mu,\mathcal{D}_{n})+O(\frac{1}{n})\\
 & \geq\frac{1}{n}H(\mu,\mathcal{D}_{n})+O_{R}(\frac{1}{n})\\
 & =\dim\mu+o_{R}(1)\;.
\end{align*}
The second statement is immediate from Proposition \ref{prop:interpolation}.
\end{proof}

\section{\label{sec:decomposition-of-p*n}Surplus entropy of $p^{*n}$ at
small scales}

In this section we shall assume that $\Phi$ is non-conformal, totally
irreducible and satisfies exponential separation. We also assume that
$\dim\mu<2$.

As in the introduction, we identify the probability vector $p=(p_{i})_{i\in\Lambda}$
with the measure $\sum_{i\in\Lambda}p_{i}\cdot\delta_{\varphi_{i}}\in\mathcal{P}(A_{2,2})$
and write $p^{*n}$ for the $n$-fold self-convolution of $p$ in
$A_{2,2}$. 

Our goal is to show that the level-$0$ components of $p^{*n}\in\mathcal{P}(A_{2,2})$
has substantial entropy at small scales, assuming $p^{*n}$ has non-negligible
entropy when conditioned on the fibers of the symbolic coding map
$\Pi$.

\subsection{\label{subsec:Distances-in-affine-group}Distances in the affine
group}

Write $G=Gl_{3}(\mathbb{R})$. Recall that $d$ is a left-invariant
metric on $A_{2,2}$. Identifying $A_{2,2}$ in the usual way as a
subgroup of $G$, we may assume that $d$ is the restriction to $A_{2,2}$
of a left-invariant metric on $G$, also denoted by $d$, which is
derived from a Riemannian metric.

Given $\beta_{1},\beta_{2},\beta_{3}\in\mathbb{R}\setminus\{0\}$,
write $\diag(\beta_{1},\beta_{2},\beta_{3})\in G$ for the diagonal
matrix with entries $\beta_{1},\beta_{2},\beta_{3}$ on the diagonal.
Given $E\in G$ write $\Vert E\Vert$ for the operator norm of $E$.
\begin{lem}
\label{lem:dist of diag from 1}Let $\beta_{1},\beta_{2},\beta_{3}>0$
and set $D=diag(\beta_{1},\beta_{2},\beta_{3})$. Then,
\[
d(D,1_{G})=O(1+\max\left\{ \log\Vert D\Vert,\log\Vert D^{-1}\Vert\right\} )\:.
\]
\end{lem}

\begin{proof}
Clearly we can assume that $\beta_{i}\ne1$ for some $1\le i\le3$.
Write
\[
M=\left\lceil \max\left\{ \left|\log\beta_{i}\right|\::\:1\le i\le3\right\} \right\rceil \;,
\]
and set
\[
E=diag(\beta_{1}^{1/M},\beta_{2}^{1/M},\beta_{3}^{1/M})\:.
\]
 Since $\beta_{i}^{1/M}\in[\frac{1}{2},2]$ for $1\le i\le3$, it
holds that $d(E,1_{G})=O(1)$. Hence,
\[
d(D,1_{G})=d(E^{M},1_{G})\le\sum_{j=1}^{M}d(E^{j},E^{j-1})=M\cdot d(E,1_{G})=O(M)\:.
\]
Now since
\[
M\le1+\max\left\{ \log\Vert D\Vert,\log\Vert D^{-1}\Vert\right\} \;,
\]
the lemma follows.
\end{proof}
\begin{lem}
\label{lem:dist from 1_G}For any $E\in G$,
\[
d(E,1_{G})=O(1+\max\left\{ \log\Vert E\Vert,\log\Vert E^{-1}\Vert\right\} )\:.
\]
\end{lem}

\begin{proof}
Let $E=VDU$ be a singular value decomposition of $E$. Since $V,U$
are orthogonal,
\[
d(V,1_{G}),d(U,1_{G})=O(1)\:.
\]
Therefore,
\begin{multline*}
d(E,1_{G})\le d(VDU,V)+d(V,1_{G})=d(DU,1_{G})+O(1)\\
\le d(DU,D)+d(D,1_{G})+O(1)=d(D,1_{G})+O(1)\:.
\end{multline*}
Now since $\Vert E\Vert=\Vert D\Vert$ and $\Vert E^{-1}\Vert=\Vert D^{-1}\Vert$,
the lemma follows by Lemma \ref{lem:dist of diag from 1}.
\end{proof}
Recall that for $W_{1},W_{2}\in\PR$ we write $d_{\PR}(W_{1},W_{2})$
for the operator norm $\left\Vert \pi_{W_{1}}-\pi_{W_{2}}\right\Vert _{op}$
of the difference between the orthogonal projections onto $W_{1}$
and $W_{2}$. Given $A\in Gl_{2}(\mathbb{R})$, with $\alpha_{1}(A)>\alpha_{2}(A)$
and singular value decomposition $A=VDU$, recall that we write $L(A)=\overline{Ve_{1}}\in\PR$.
\begin{lem}
\label{lem:est of d}Let $g_{1},g_{2}\in A_{2,2}$ satisfy $g_{i}(x)=B_{i}x+b_{i}$
and $\alpha_{1}(B_{i})>\alpha_{2}(B_{i})$ for $i=1,2$. Assume that,

\begin{align}
d_{\PR}(L(B_{1}),L(B_{2})) & =O\left(\frac{\alpha_{2}(B_{1})}{\alpha_{1}(B_{1})}\right)\label{enu:est of d a1}\\
\alpha_{i}(B_{2}) & =\Theta\left(\alpha_{i}(B_{1})\right)\text{ for }i=1,2\label{enu:est of d a2}\\
|b_{1}-b_{2}| & =O(\alpha_{1}(B_{1}))\label{enu:est of d a3}\\
|\pi_{L(B_{1})^{\perp}}(b_{1}-b_{2})| & =O(\alpha_{2}(B_{1}))\;.\label{enu:est of d a4}
\end{align}

Then $d(g_{1},g_{2})=O(1)$.
\end{lem}

\begin{proof}
Note that $d(g_{1},g_{2})=d(g_{2}^{-1}g_{1},1_{G})$ and
\[
g_{2}^{-1}g_{1}(x)=B_{2}^{-1}B_{1}x+B_{2}^{-1}(b_{1}-b_{2})\text{ for \ensuremath{x\in\mathbb{R}^{2}}\;.}
\]
Set
\[
E=\left(\begin{array}{cc}
B_{2}^{-1}B_{1} & B_{2}^{-1}(b_{1}-b_{2})\\
0 & 1
\end{array}\right)\in G\;,
\]
then by Lemma \ref{lem:dist from 1_G} it suffices to show that $\Vert E\Vert,\Vert E^{-1}\Vert=O(1).$
We shall show that $\Vert E\Vert=O(1)$. In an analogues manner it
can be shown that $\Vert E^{-1}\Vert=O(1)$. Note that
\begin{equation}
\Vert E\Vert=O(1+\Vert B_{2}^{-1}B_{1}\Vert+|B_{2}^{-1}(b_{1}-b_{2})|)\:.\label{eq:bd on E}
\end{equation}

For $i=1,2$ let $V_{i}D_{i}U_{i}$ be a singular value decomposition
of $B_{i}$. It holds that,
\begin{multline*}
|B_{2}^{-1}(b_{1}-b_{2})|=\left|D_{2}^{-1}V_{2}^{-1}\left(\left\langle b_{1}-b_{2},V_{2}e_{1}\right\rangle V_{2}e_{1}+\left\langle b_{1}-b_{2},V_{2}e_{2}\right\rangle V_{2}e_{2}\right)\right|\\
\le\alpha_{1}(B_{2})^{-1}|b_{1}-b_{2}|+\alpha_{2}(B_{2})^{-1}\left|\left\langle b_{1}-b_{2},V_{2}e_{2}\right\rangle \right|\:.
\end{multline*}
By assumptions (\ref{enu:est of d a2}) and (\ref{enu:est of d a3})
of the lemma,
\[
\alpha_{1}(B_{2})^{-1}|b_{1}-b_{2}|=O(1)\:.
\]
Additionally,
\begin{multline*}
\left|\left\langle b_{1}-b_{2},V_{2}e_{2}\right\rangle \right|=|\pi_{L(B_{2})^{\perp}}(b_{1}-b_{2})|\\
\le d_{\PR}(L(B_{1})^{\perp},L(B_{2})^{\perp})\cdot|b_{1}-b_{2}|+|\pi_{L(B_{1})^{\perp}}(b_{1}-b_{2})|\:.
\end{multline*}
By this and assumptions (\ref{enu:est of d a1}) to (\ref{enu:est of d a4})
\[
\alpha_{2}(B_{2})^{-1}\left|\left\langle b_{1}-b_{2},V_{2}e_{2}\right\rangle \right|=O(1)\;,
\]
which shows that
\begin{equation}
|B_{2}^{-1}(b_{1}-b_{2})|=O(1)\:.\label{eq:bd on translation}
\end{equation}

For $i=1,2$,
\begin{eqnarray*}
|B_{2}^{-1}B_{1}U_{1}^{-1}e_{i}| & = & |D_{2}^{-1}V_{2}^{-1}V_{1}D_{1}e_{i}|=\alpha_{i}(B_{1})\cdot|D_{2}^{-1}V_{2}^{-1}V_{1}e_{i}|\\
 & = & \alpha_{i}(B_{1})\cdot\left|D_{2}^{-1}V_{2}^{-1}\left(\left\langle V_{1}e_{i},V_{2}e_{1}\right\rangle V_{2}e_{1}+\left\langle V_{1}e_{i},V_{2}e_{2}\right\rangle V_{2}e_{2}\right)\right|\\
 & \le & \frac{\alpha_{i}(B_{1})}{\alpha_{1}(B_{2})}\left|\left\langle V_{1}e_{i},V_{2}e_{1}\right\rangle \right|+\frac{\alpha_{i}(B_{1})}{\alpha_{2}(B_{2})}\left|\left\langle V_{1}e_{i},V_{2}e_{2}\right\rangle \right|\\
 & = & O(1)+\frac{\alpha_{i}(B_{1})}{\alpha_{2}(B_{2})}\left|\left\langle V_{1}e_{i},V_{2}e_{2}\right\rangle \right|\:.
\end{eqnarray*}
From this and assumption (\ref{enu:est of d a2}) we get $|B_{2}^{-1}B_{1}U_{1}^{-1}e_{2}|=O(1)$.
Additionally,
\[
\left|\left\langle V_{1}e_{1},V_{2}e_{2}\right\rangle \right|=\left|\pi_{L(B_{1})}(V_{2}e_{2})\right|\le d_{\PR}(L(B_{1}),L(B_{2}))+\left|\pi_{L(B_{2})}(V_{2}e_{2})\right|\:.
\]
From this, $\pi_{L(B_{2})}(V_{2}e_{2})=0$, and assumptions (\ref{enu:est of d a1})
and (\ref{enu:est of d a2}),
\[
\frac{\alpha_{1}(B_{1})}{\alpha_{2}(B_{2})}\left|\left\langle V_{1}e_{1},V_{2}e_{2}\right\rangle \right|=O(1)\:.
\]
It follows 
\[
|B_{2}^{-1}B_{1}U_{1}^{-1}e_{i}|=O(1)\text{ for \ensuremath{i=1,2}}\;,
\]
which shows that $\Vert B_{2}^{-1}B_{1}\Vert=O(1)$. From this, (\ref{eq:bd on translation})
and (\ref{eq:bd on E}) we get $\Vert E\Vert=O(1)$, which completes
the proof of the lemma.
\end{proof}

\subsection{\label{subsec:A-decomposition-of-p*n}Surplus entropy of components
of $p^{*n}$}

Recall that $\xi=p^{\mathbb{N}}\in\mathcal{P}(\Lambda^{\mathbb{N}})$
and $\Pi:\Lambda^{\mathbb{N}}\rightarrow\mathbb{R}^{2}$ is the coding
map associated with $\Phi$. 

Let $\{\xi_{\omega}\}_{\omega\in\Lambda^{\mathbb{N}}}\subset\mathcal{P}(\Lambda^{\mathbb{N}})$
be the disintegration of $\xi$ with respect to $\Pi^{-1}(\mathcal{B})$,
where $\mathcal{B}$ is the Borel $\sigma$-algebra of $\mathbb{R}^{2}$.
The function $\omega\rightarrow\xi_{\omega}$ is measurable and defined
$\xi$-a.e. We also write this as $\{\xi_{x}\}_{x\in X},$ since the
map $\omega\rightarrow\xi_{\omega}$ is measurable with respect to
$\Pi^{-1}\mathcal{B}$. This is defined $\mu$-a.e. since $\mu=\Pi\xi$.

Given $\nu\in\mathcal{P}(\Lambda^{\mathbb{N}})$ and $n\ge1$ write
\[
[\nu]_{n}=\sum_{w\in\Lambda^{n}}\nu[w]\cdot\delta_{\varphi_{w}}\in\mathcal{P}(A_{2,2})\:.
\]
\begin{lem}
\label{lem:first decomp of nu^*n}For every $n\ge1$,
\[
p^{*n}=\int[\xi_{\omega}]_{n}\:d\xi(\omega)\:.
\]
\end{lem}

\begin{proof}
It holds that
\begin{align*}
p^{*n} & =\sum_{w\in\Lambda^{n}}\xi[w]\cdot\delta_{\varphi_{w}}\\
 & =\sum_{w\in\Lambda^{n}}\int\xi_{\omega}[w]\:d\xi(\omega)\cdot\delta_{\varphi_{w}}\\
 & =\int\sum_{w\in\Lambda^{n}}\xi_{\omega}[w]\cdot\delta_{\varphi_{w}}\:d\xi(\omega)\\
 & =\int[\xi_{\omega}]_{n}\:d\xi(\omega)\;,
\end{align*}
which proves the lemma.
\end{proof}
Let $0>\chi_{1}>\chi_{2}>-\infty$ be the Lyapunov exponents corresponding
to $\sum_{i\in\Lambda}p_{i}\cdot\delta_{A_{i}}\in\mathcal{P}(Gl_{2}(\mathbb{R}))$
(see Theorem \ref{thm:Furstenberg-Oseledets} (\ref{enu:Lyapunov exponents})).
For $g\in A_{2,2}$ recall that $A_{g}\in Gl_{2}(\mathbb{R})$ and
$b_{g}\in\mathbb{R}^{2}$ are the linear and translation parts of
$g$ respectively. Also recall that $\mathcal{P}_{n}$ is the partition
of $\Lambda^{\mathbb{N}}$ into $n$-cylinders: $\mathcal{P}_{n}=\{[w]\subset\Lambda^{\mathbb{N}}\::\:w\in\Lambda^{n}\}$. 
\begin{prop}
\label{prop:decomp of desc of slices}Let $\mu$ be a self-affine
measure defined by a non-conformal, totally irreducible and exponentially
separated system $\Phi$. Suppose that $\dim\mu<2$ and
\[
H(\xi,\mathcal{P}_{1}|\Pi^{-1}\mathcal{B})>0\:.
\]
Then there exist $\varepsilon>0$ and $M\ge1$ so that for $\xi$-a.e.
$\omega\in\Lambda^{\mathbb{N}}$ and $n>N(\omega)$,

\[
\frac{1}{Mn}H([\xi_{\omega}]_{n},\mathcal{D}_{Mn}|\mathcal{D}_{0})>\epsilon\;.
\]
Furthermore, writing $\widetilde{\theta}^{\omega,n}$ for a random
level-$0$ component of $[\xi_{\omega}]_{n}$, 
\begin{equation}
\liminf_{n\rightarrow\infty}\mathbb{P}\left(\frac{1}{Mn}H(\widetilde{\theta}^{\omega,n},\mathcal{D}_{Mn})>\varepsilon\right)>\varepsilon\;,\label{eq:fibers-cor-1}
\end{equation}
and there exists a sequence $\delta_{n}\searrow0$ (depending on $\omega$)
such that, for $i=1,2$,
\begin{equation}
\lim_{n\rightarrow\infty}\mathbb{P}\left(|\chi_{i}-\frac{1}{n}\log\alpha_{i}(A_{g})|<\delta_{n}\text{ for all }g\in\supp\widetilde{\theta}^{\omega,n}\right)=1\;.\label{eq:fibers-cor-2}
\end{equation}
\end{prop}

\begin{proof}
By $H(\xi,\mathcal{P}_{1}|\Pi^{-1}\mathcal{B})>0$ and \cite[Theorem 2.2, part (iii)]{FH},
there exists $\varepsilon'>0$ such that $\xi_{\omega}$ has exact
dimension $>\varepsilon'$ for $\xi$-a.e. $\omega\in\Lambda^{\mathbb{N}}$.
Hence
\begin{equation}
\lim_{n\rightarrow\infty}\:\frac{1}{n}H(\xi_{\omega},\mathcal{P}_{n})>\varepsilon'\text{ for }\xi\text{-a.e. }\omega\in\Lambda^{\mathbb{N}}\;.\label{eq:ent dim of slices-1}
\end{equation}
Since $\Phi$ satisfies exponential separation, there exists $M\ge1$
such that
\[
\mathcal{D}_{Mn}(\varphi_{w_{1}})\ne\mathcal{D}_{Mn}(\varphi_{w_{2}})\qquad\text{for every \ensuremath{n\ge1} and distinct \ensuremath{w_{1},w_{2}\in\Lambda^{n}}\;.}
\]
By this and (\ref{eq:ent dim of slices-1}),
\[
\underset{n\rightarrow\infty}{\lim}\:\frac{1}{n}H([\xi_{\omega}]_{n},\mathcal{D}_{Mn})>\varepsilon'\text{ for }\xi\text{-a.e. }\omega\in\Lambda^{n}\:.
\]
Setting $\varepsilon=\varepsilon'/M$ we have, equivalently,
\begin{equation}
\underset{n\rightarrow\infty}{\lim}\:\frac{1}{Mn}H([\xi_{\omega}]_{n},\mathcal{D}_{Mn})>\varepsilon\text{ for }\xi\text{-a.e. }\omega\in\Lambda^{n}\:.\label{eq:ent of desc slices-1}
\end{equation}
We wish show that this continues to hold when we condition on $\mathcal{D}_{0}$.
For this, it suffices to show that there are sets $E_{n}=E_{\omega,n}\subseteq A_{2,2}$
such that
\begin{enumerate}
\item $\lim_{n\rightarrow\infty}[\xi_{\omega}]_{n}(E_{\omega,n})=1$ for
$\xi$-a.e. $\omega$;
\item $E_{\omega,n}$ can be covered by $2^{o(n)}$ cells from $\mathcal{D}_{0}$. 
\end{enumerate}
This is sufficient because, by (1) and by concavity and almost convexity
of entropy, we have that the entropies
\[
\frac{1}{Mn}H([\xi_{\omega}]_{n},\mathcal{D}_{Mn}|\mathcal{D}_{0})\text{ and }\frac{1}{Mn}H(([\xi_{\omega}]_{n})_{E_{\omega,n}},\mathcal{D}_{Mn}|\mathcal{D}_{0})
\]
are asymptotic as $n\rightarrow\infty$; and by (2), the second of
these entropies is asymptotic to $\frac{1}{Mn}H(([\xi_{\omega}]_{n})_{E_{\omega,n}},\mathcal{D}_{Mn})$,
because (2) easily implies that
\[
\frac{1}{Mn}H(([\xi_{\omega}]_{n})_{E_{\omega,n}},\mathcal{D}_{0})=o(1)\:.
\]

For the remainder of the proof we fix a $\xi$-typical $\omega\in\Lambda^{\mathbb{N}}$,
which we will assume satisfies several full-measure conditions which
arise in the course of the proof.

By Theorem \ref{thm:Furstenberg-Oseledets} (and the identity $\xi=\int\xi_{\omega}d\xi(\omega)$),
for $i=1,2$,
\[
\alpha_{i}(A_{\sigma|_{n}})=2^{n(\chi_{i}+o_{\sigma}(1))}\text{ for \ensuremath{\xi_{\omega}}-a.e. \ensuremath{\sigma\in\Lambda^{\mathbb{N}}}.\;}
\]
Furthermore, as a by-product of the proof of the Oseledets theorem
(see. e.g. \cite{Ruelle1979}),
\[
d_{\PR}(L(A_{\sigma|_{n}}),L(\sigma))=2^{n(\chi_{2}-\chi_{1}+o_{\sigma}(1))}\text{ for \ensuremath{\xi}-a.e. \ensuremath{\sigma\in\Lambda^{\mathbb{N}}}.\;}
\]
Hence, by Proposition \ref{prop:const directions} and the assumption
$\dim\mu<2$,
\[
d_{\PR}(L(A_{\sigma|_{n}}),L(\omega))=2^{n(\chi_{2}-\chi_{1}+o_{\sigma}(1))}\text{ for \ensuremath{\xi_{\omega}}-a.e. \ensuremath{\sigma\in\Lambda^{\mathbb{N}}}\;.}
\]
It follows that there exists a sequence $\delta_{n}\searrow0$ (which
implicitly depends on $\omega$) such that the sets $F_{n}=F_{\omega,n}$
defined by
\begin{equation}
F_{n}=\left\{ \sigma\in\Lambda^{\mathbb{N}}\;:\;\begin{array}{c}
d_{\PR}(L(A_{\sigma|_{n}}),L(\omega))\leq2^{n(\chi_{2}-\chi_{1}+\delta_{n})},\\
2^{n(\chi_{i}-\delta_{n})}\leq\alpha_{i}(A_{\sigma|_{n}})\leq2^{n(\chi_{i}+\delta_{n})}\text{ for }i=1,2,\\
\text{ and }[\sigma|_{n}]\cap\Pi^{-1}(\Pi\omega)\ne\emptyset
\end{array}\right\} \label{eq:sing vals of eta in F_=00007Bomega,n=00007D}
\end{equation}
satisfy
\[
\xi_{\omega}(F_{n})\rightarrow1\;.
\]
Note that $F_{n}$ is a union of $n$-cylinders (since $\sigma\in F_{n}$
depends on $\sigma|_{n}$). We define $E_{n}=E_{\omega,n}\subseteq A_{2,2}$
by
\[
E_{n}=\{\varphi_{\sigma|_{n}}\;:\;\sigma\in F_{n}\}\:.
\]
Then, by definition of $[\xi_{\omega}]_{n}$, we have
\[
[\xi_{\omega}]_{n}(E_{n})=\xi_{\omega}(F_{n})\rightarrow1\;,
\]
giving the first required property of $E_{n}$. 

It remains to be shown that we can cover $E_{n}$ by $2^{o(n)}$ level-$0$
dyadic cells, or equivalently, $2^{o(n)}$ sets of diameter $O(1)$.
To begin, observe that by (\ref{eq:sing vals of eta in F_=00007Bomega,n=00007D}),
for each $n\ge1$ and $\sigma\in F_{n}$, 
\[
d_{\PR}(L(A_{\sigma|_{n}}),L(\omega))\le2^{3\delta_{n}n}\cdot\inf_{\zeta\in F_{n}}\:\frac{\alpha_{2}(A_{\zeta|_{n}})}{\alpha_{1}(A_{\zeta|_{n}})}
\]
and
\[
0<\alpha_{i}(A_{\sigma|_{n}})\le2^{2\delta_{n}n}\cdot\inf_{\zeta\in F_{n}}\:\alpha_{i}(A_{\zeta|_{n}})\text{ for }i=1,2\:.
\]
Hence we can partition $F_{n}$ into $2^{o(n)}$ Borel sets in such
a way that on each cell the values of $L(A_{\sigma|_{n}})$ lie in
an interval of diameter $\inf_{\zeta\in F_{n}}\alpha_{2}(A_{\zeta|_{n}})/\alpha_{1}(A_{\zeta|_{n}})$
and the values of $\alpha_{i}(A_{\sigma|n})$ lie in an interval of
length $\frac{1}{2}\inf_{\zeta\in F_{n}}\alpha_{i}(A_{\zeta|_{n}})$.
We obtain a finite Borel partition $\mathcal{F}_{n}=\mathcal{F}_{\omega,n}$
of $F_{n}$ such that $|\mathcal{F}_{n}|=2^{O(\delta_{n}n)}=2^{o(n)}$,
and such that
\begin{equation}
d_{\PR}(L(A_{\sigma|_{n}}),L(A_{\zeta|_{n}}))\le\frac{\alpha_{2}(A_{\zeta|_{n}})}{\alpha_{1}(A_{\zeta|_{n}})}\text{ for \ensuremath{F\in\mathcal{F}_{n}} and \ensuremath{\sigma,\zeta\in F}\;,}\label{eq:dist of directions-1}
\end{equation}
and for $i=1,2$,
\begin{equation}
|\alpha_{i}(A_{\sigma|_{n}})-\alpha_{i}(A_{\zeta|_{n}})|\le\frac{1}{2}\alpha_{i}(A_{\zeta|_{n}})\text{ for \ensuremath{F\in\mathcal{F}_{n},} and \ensuremath{\sigma,\zeta\in F}\;.}\label{eq:dist of sing vals-1}
\end{equation}
Every $F\in\mathcal{F}_{n}$ is defined by conditions on $n$-cylinders
so $F$ is again a union of $n$-cylinders, hence the collection $\mathcal{\mathcal{E}}_{n}$
of corresponding sets
\[
E=E(F)=\{\varphi_{\sigma|_{n}}\;:\;\sigma\in F\}
\]
is a partition of $E_{n}$, and has the same size as $\mathcal{F}_{n}$. 

Therefore, it is sufficient for us to show that $\diam E(F)=O(1)$
for all $F\in\mathcal{F}_{n}$. For this we will use Lemma \ref{lem:est of d}.
Equations (\ref{eq:dist of directions-1}) and (\ref{eq:dist of sing vals-1})
establish the first two hypotheses of that lemma, so it remains to
establish the last two.

Let $B\subset\mathbb{R}^{2}$ be a ball with center $0$ and $\supp(\mu)\subset B$.
Let $n\ge1$ and $\sigma\in\Lambda^{\mathbb{N}}$ with $[\sigma|_{n}]\cap\Pi^{-1}(\Pi\omega)\ne\emptyset$.
For $\zeta\in[\sigma|_{n}]\cap\Pi^{-1}(\Pi\omega)$ we have,
\[
\{\Pi(\omega)\}=\cap_{k\ge1}\varphi_{\zeta|_{k}}(B)\:.
\]
Hence $\varphi_{\zeta|_{n}}(0),\Pi(\omega)\in\varphi_{\zeta|_{n}}(B)$,
which gives $\varphi_{\sigma|_{n}}(0),\Pi(\omega)\in\varphi_{\sigma|_{n}}(B)$.
It follows that,
\begin{equation}
|\varphi_{\sigma|_{n}}(0)-\Pi(\omega)|=O(\alpha_{1}(A_{\sigma|_{n}}))\label{eq:dist of trans-1}
\end{equation}
and
\begin{equation}
|\pi_{L(A_{\sigma|_{n}})^{\perp}}(\varphi_{\sigma|_{n}}(0)-\Pi(\omega))|=O(\alpha_{2}(A_{\sigma|_{n}}))\:.\label{eq:dist of proj of trans-1}
\end{equation}

Let $n\ge1$, $F\in\mathcal{F}_{n}$ and $\sigma,\zeta\in F$. Set
$a_{\sigma}=\varphi_{\sigma|_{n}}(0)$, $a_{\zeta}=\varphi_{\zeta|_{n}}(0)$,
$\pi_{\sigma}=\pi_{L(A_{\sigma|_{n}})^{\perp}}$ and $\pi_{\zeta}=\pi_{L(A_{\zeta|_{n}})^{\perp}}$.
By (\ref{eq:dist of trans-1}) and (\ref{eq:dist of sing vals-1}),
\begin{align*}
|a_{\sigma}-a_{\zeta}| & \le|a_{\sigma}-\Pi(\omega)|+|\Pi(\omega)-a_{\zeta}|\\
 & \le O(\alpha_{1}(A_{\sigma|_{n}})+\alpha_{1}(A_{\zeta|_{n}}))\\
 & =O(\alpha_{1}(A_{\zeta|_{n}}))\:.
\end{align*}
This is the third hypothesis of Lemma \ref{lem:est of d}.

Finally, By (\ref{eq:dist of proj of trans-1}),
\begin{align*}
|\pi_{\zeta}(a_{\sigma}-a_{\zeta})| & \le|\pi_{\zeta}(a_{\sigma}-\Pi(\omega))|+|\pi_{\zeta}(\Pi(\omega)-a_{\zeta})|\\
 & =|\pi_{\zeta}(a_{\sigma}-\Pi(\omega))|+O(\alpha_{2}(A_{\zeta|_{n}}))\:.
\end{align*}
Since $d_{\PR}$ is defined via the operator norm,
\[
|\pi_{\zeta}(a_{\sigma}-\Pi(\omega))|\le|\pi_{\sigma}(a_{\sigma}-\Pi(\omega))|+d_{\PR}(L(A_{\sigma|_{n}})^{\perp},L(A_{\zeta|_{n}})^{\perp})\cdot|a_{\zeta}-\Pi(\omega)|\:.
\]
Hence by (\ref{eq:dist of proj of trans-1}), (\ref{eq:dist of trans-1}),
(\ref{eq:dist of directions-1}) and (\ref{eq:dist of sing vals-1}),
\[
|\pi_{\zeta}(a_{\sigma}-\Pi(\omega))|=O(\alpha_{2}(A_{\sigma|_{n}})+\frac{\alpha_{2}(A_{\zeta|_{n}})}{\alpha_{1}(A_{\zeta|_{n}})}\cdot\alpha_{1}(A_{\sigma|_{n}}))=O(\alpha_{2}(A_{\zeta|_{n}}))\;,
\]
which gives $|\pi_{\zeta}(a_{\sigma}-a_{\zeta})|=O(\alpha_{2}(A_{\zeta|_{n}}))$,
the last hypothesis of Lemma \ref{lem:est of d}. Thus we have shown
that $\varphi_{\sigma|_{n}}$ and $\varphi_{\zeta|_{n}}$ satisfy
all of the hypotheses of Lemma \ref{lem:est of d}, and hence that
$d(\varphi_{\sigma|_{n}},\varphi_{\zeta|_{n}})=O(1)$ for all $\sigma,\zeta\in F$.
This precisely means that $\diam E(F)=O(1)$, as needed.

To prove Equation (\ref{eq:fibers-cor-1}), we use the trivial identity
\[
\frac{1}{Mn}H([\xi_{\omega}]_{n},\mathcal{D}_{Mn}|\mathcal{D}_{0})=\frac{1}{Mn}\mathbb{E}\left(H(\widetilde{\theta}^{\omega,n},\mathcal{D}_{Mn})\right)
\]
(which is just a consequence of the definition of conditional entropy
and the component distribution), and the elementary fact that if random
variable $H\in[0,1]$ satisfies $\mathbb{E}(H)>\varepsilon$ then
$\mathbb{P}(H>\varepsilon/2)>\varepsilon/2$. So (\ref{eq:fibers-cor-1})
follows from what was already proved upon replacing $\varepsilon$
by $\varepsilon/C$ for some universal constant $C>1$.

As for (\ref{eq:fibers-cor-2}), from our construction it is clear
that
\begin{equation}
\lim_{n\rightarrow\infty}\mathbb{P}\left(|\chi_{i}-\frac{1}{n}\log\alpha_{i}(A_{g})|<\delta_{n}\text{ for some }g\in\supp\widetilde{\theta}^{\omega,n}\right)=1\;.\label{eq:some}
\end{equation}
If $|\chi_{i}-\frac{1}{n}\log\alpha_{i}(A_{g})|<\delta_{n}$ for some
$g\in\supp\widetilde{\theta}^{\omega,n}$ and if $h\in\supp\widetilde{\theta}^{\omega,n}$
then, since $d(g,h)\leq R$ for some global $R>0$ (because $\widetilde{\theta}^{\omega,n}$
is supported on a level-$0$ dyadic cell), we have $|\chi_{i}-\frac{1}{n}\log\alpha_{i}(A_{h})|<\delta_{n}+O_{R}(1/n)$
(because we can write $h=gg'$ with $d(g',\id)\leq R$, and so clearly
$\alpha_{i}(h)/\alpha_{i}(g)=\Theta_{R}(1)$ for $i=1,2$, from which
the claim follows). Thus in the Equation (\ref{eq:some}) we can replace
``some'' by ``all'' at the expense of replacing $\delta_{n}$
by $C\max\{\delta_{n},1/n\}$ for some universal constant $C>1$.
\end{proof}

\section{\label{sec:Proof-of-main-result}Proof of main results}

\subsection{Strongly irreducible case: Proof of Theorem \ref{thm:main}}

As explained in the introduction, our main result (Theorem \ref{thm:main})
follows from Theorem \ref{thm:main-reformulated}, which is the following
statement:
\begin{thm*}
\label{thm:dim 0 slices}If $\mu$ is a self-affine measure defined
by a non-conformal, totally irreducible and exponentially separated
system, and if $H(\xi,\mathcal{P}_{1}|\Pi^{-1}\mathcal{B})>0$ and
$\dim\mu\geq1$, then $\dim\mu=2$.
\end{thm*}
\begin{proof}
Assume for the sake of contradiction that $\dim\mu<2$. 

Let $\varepsilon>0$ and $M\ge1$ be as in Proposition \ref{prop:decomp of desc of slices}.
For $n\ge1$ we have $\mu=p^{*n}\conv\mu$. By Lemma \ref{lem:first decomp of nu^*n},
$\mu=p^{*n}\conv\mu=\int[\xi_{\omega}]_{n}\conv\mu\,d\xi(\omega)$,
so by concavity of conditional entropy,
\begin{gather}
\frac{1}{Mn}H(\mu,\mathcal{D}_{(M+|\chi_{2}|)n}\mid\mathcal{D}_{|\chi_{2}|n})\;\;\;\;\;\;\;\;\;\;\;\;\;\;\;\;\;\;\;\;\;\;\;\;\;\;\;\;\;\;\;\;\;\;\;\;\nonumber \\
\;\;\;\;\;\;\ge\int\frac{1}{Mn}H([\xi_{\omega}]_{n}\conv\mu,\mathcal{D}_{(M+|\chi_{2}|)n}\mid\mathcal{D}_{|\chi_{2}|n})\:d\xi(\omega)\;.\label{eq:ent first decomposition}
\end{gather}

Let us write $\widetilde{\theta}^{\omega,n}$ for a random level-$0$
component of the measure $[\xi_{\omega}]_{n}$, so that for each $\omega$,
\[
[\xi_{\omega}]_{n}=\mathbb{E}\left(\widetilde{\theta}^{\omega,n}\right)
\]
Inserting this into (\ref{eq:ent first decomposition}) and using
concavity again,
\begin{gather}
\frac{1}{Mn}H(\mu,\mathcal{D}_{(M+|\chi_{2}|)n}\mid\mathcal{D}_{|\chi_{2}|n})\;\;\;\;\;\;\;\;\;\;\;\;\;\;\;\;\;\;\;\;\;\;\;\;\;\;\;\;\;\;\;\;\;\;\;\;\;\;\;\;\;\;\nonumber \\
\;\;\;\;\;\;\geq\int\mathbb{E}\left(\frac{1}{Mn}H(\widetilde{\theta}^{\omega,n}\conv\mu,\mathcal{D}_{(M+|\chi_{2}|)n}|\mathcal{D}_{|\chi_{2}|n})\right)\;d\xi(\omega)\;.\label{eq:ent second decomposition}
\end{gather}
Our goal is to get a lower bound for the integrand on the right hand
side. Specifically we will show that for $\xi$-a.e. $\omega$, with
probability tending to one (over the choice of the component), the
entropy in the expectation is bounded below by $\alpha-o(1)$, and,
when $n$ is large enough, with some definite probability $q>0$ it
is greater than $\alpha+\delta$ (for another parameter $\delta>0$).
This will imply that for large $n$ the right hand side is $\geq\alpha+q\delta-o(1)$,
giving a contradiction.

Let $R>1$ be a global constant which is larger than the diameter
of any level-$0$ dyadic component of $A_{2,2}$. Suppose also that
$R^{-1}<\frac{\chi_{1}-\chi_{2}}{2}$. From now on fix a $\xi$-typical
$\omega\in\Lambda^{\mathbb{N}}$. Terms of the form $o(1)$ etc. are
asymptotic as $n\rightarrow\infty$ (but may depend on $\omega$ as
indicated).

Since $\varepsilon$ and $M$ were chosen as in Proposition \ref{prop:decomp of desc of slices},
\begin{equation}
\liminf_{n\rightarrow\infty}\mathbb{P}\left(\frac{1}{Mn}H(\widetilde{\theta}^{\omega,n},\mathcal{D}_{Mn})>\varepsilon\right)>\varepsilon\;,\label{eq:good-fiber-entropy}
\end{equation}
and, for some $\delta_{n}\searrow0$ (depending on $\omega$), for
$i=1,2$,
\begin{equation}
\lim_{n\rightarrow\infty}\mathbb{P}\left(|\chi_{i}-\frac{1}{n}\log\alpha_{i}(A_{g})|<\delta_{n}\text{ for all }g\in\supp\widetilde{\theta}^{\omega,n}\right)=1\;.\label{eq:good-fiber-exponents}
\end{equation}

Fix a large $n\ge1$, a component $\widetilde{\theta}^{\omega,n}$
in the event in (\ref{eq:good-fiber-exponents}), and some $g\in\supp\widetilde{\theta}^{\omega,n}$.
Note that $R$ bounds the diameter of $\supp\widetilde{\theta}^{\omega,n}$.
Write
\[
a_{i}=\frac{1}{n}\log\alpha_{i}(A_{g})\quad\text{for }i=1,2
\]
so that 
\[
|a_{i}-\chi_{i}|<\delta_{n}\qquad\text{for }i=1,2\;.
\]
Since $\delta_{n}\searrow0$ we may assume that $a_{1}-a_{2}>\frac{\chi_{1}-\chi_{2}}{2}>R^{-1}$.
Then, by Proposition \ref{prop:entropy-non-decrease},

\[
\frac{1}{Mn}H(\widetilde{\theta}^{\omega,n}\conv\mu,\mathcal{D}_{(M+|a_{2}|)n}|\mathcal{D}_{|a_{2}|n})\ge\alpha-o(1)\;,
\]
which, in view of $|a_{i}-\chi_{i}|<\delta_{n}$ is the same as
\begin{equation}
\frac{1}{Mn}H(\widetilde{\theta}^{\omega,n}\conv\mu,\mathcal{D}_{(M+|\chi_{2}|)n}|\mathcal{D}_{|\chi_{2}|n})\ge\alpha-o_{\omega}(1)\;.\label{eq:ent no increase}
\end{equation}
This is the general lower bound we wanted for the integrand in (\ref{eq:ent second decomposition}).

Next, assume that $\widetilde{\theta}^{\omega,n}$ is in the event
in (\ref{eq:good-fiber-entropy}) and let $\delta=\delta(\varepsilon,R)>0$
be as in Theorem \ref{thm:entropy growth}. Then, by\footnote{This is where the assumption $\dim\mu<2$ is used.}
Theorem \ref{thm:entropy growth}, 
\[
\frac{1}{Mn}H(\widetilde{\theta}^{\omega,n}\conv\mu,\mathcal{D}_{(M+|a_{2}|)n}|\mathcal{D}_{|a_{2}|n})\ge\alpha+\delta\;.
\]
Using again the fact that $|\chi_{i}-a_{i}|<\delta_{n}$, this is
equivalent to
\begin{equation}
\frac{1}{Mn}H\left(\widetilde{\theta}^{\omega,n}\conv\mu,\mathcal{D}_{(M+|\chi_{2}|)n}\mid\mathcal{D}_{|\chi_{2}|n}\right)\ge\alpha+\delta-o_{\omega}(1)\;.\label{eq:ent increase}
\end{equation}
Combining (\ref{eq:ent no increase}), (\ref{eq:ent increase}) with
(\ref{eq:ent second decomposition}), (\ref{eq:good-fiber-entropy})
and (\ref{eq:good-fiber-exponents}), we find that 
\[
\frac{1}{Mn}H(\mu,\mathcal{D}_{(M+|\chi_{2}|)n}|\mathcal{D}_{|\chi_{2}|n})\geq\alpha+\delta\cdot\varepsilon-o_{\omega}(1)\:.
\]
But since $\mu$ has exact dimension $\alpha$,
\[
\frac{1}{Mn}H(\mu,\mathcal{D}_{(M+|\chi_{2}|)n}\mid\mathcal{D}_{|\chi_{2}|n})=\alpha+o(1)\:.
\]
This contradiction completes the proof of the theorem.
\end{proof}

\subsection{Triangular case: proof of Theorem \ref{thm:main-triangular}}

As in the introduction, let $\pi_{1}$ denote projection to the $x$-axis
$\overline{e}_{1}\in\PR$, and also write $\overline{e}_{2}\in\PR$
for the vertical direction. We recall the statement of Theorem \ref{thm:main-triangular}:
\begin{thm*}
Let $\mu$ be a self-affine measure defined by a system $\Phi=\{\varphi_{i}(x)=A_{i}x+v_{i}\}_{i\in\Lambda}$
as in (\ref{eq:triangular-system}), i.e. $\{A_{i}\}$ are invertible
and lower-triangular. Suppose that,
\begin{itemize}
\item $\{A_{i}\}$ are not simultaneously conjugated to a diagonal system;
\item $\Phi$ satisfies exponential separation;
\item The Lyapunov exponents are distinct: $-\infty<\chi_{2}<\chi_{1}<0$
and $\overline{e}_{2}$ is contracted at rate $2^{\chi_{2}}$ (for
example, this holds if $|c_{i}|<|a_{i}|$ for all $i\in\Lambda$);
\item $\mu$ is not supported on a quadratic curve;
\item The projection $\pi_{1}\mu$ has the maximal possible dimension, i.e.
\begin{equation}
\dim\pi_{1}\mu=\min\{1,\dim\mu\}\;.\label{eq:pi-1-lower-bound-1}
\end{equation}
\end{itemize}
Then
\[
\dim\mu=\min\{2,\dim_{L}\mu\}\;.
\]
\end{thm*}
Let us discuss what changes relative to the proof of the irreducible
case.

\subsubsection*{Furstenberg measures and Ledrappier-Young}

Most of Theorem \ref{thm:Furstenberg-Oseledets} continues to hold
for system which are non-conformal and have distinct Lyapunov exponents,
with the exception of the uniqueness of the limiting distribution
(part 4), and the pointwise convergence in the last equation of part
(5), which no longer holds for all initial lines. Nevertheless, the
measures $\eta,\eta^{*}$ are well-defined as the limiting distributions
of $L(\zeta_{n}\ldots\zeta_{1})$ and $L(\zeta_{n}^{*}\ldots\zeta_{1}^{*})$,
respectively, where $(\zeta_{i})$ are i.i.d. variables with distribution
$\sum_{i\in\Lambda}p_{i}\cdot\delta_{A_{i}}$. Under our assumptions
that $\overline{e}_{2}$ is contracted asymptotically at rate $2^{\chi_{2}}$,
and the matrices are not jointly diagonalizable, one can show that
\begin{enumerate}
\item $\eta$ is continuous and has positive dimension, and it is the limiting
distribution of $\zeta_{n}\ldots\zeta_{1}W$ for every $W\in\PR\setminus\{\overline{e}_{2}\}$.
\item $\eta^{*}=\delta_{\overline{e}_{1}}$, and it is the limiting distribution
of $\zeta_{n}^{*}\ldots\zeta_{1}^{*}W$ for every $W\in\PR$.
\end{enumerate}
The Ledrappier-Young formula is valid, but since $\eta^{*}=\delta_{\overline{e}_{1}}$,
it simply states that
\[
\dim\mu=\dim\pi_{1}\mu+\dim\mu_{x}^{\overline{e}_{2}}\qquad\text{for }\mu\text{-a.e. }x\;.
\]
Recall from the introduction that $\pi_{1}\mu$ is self-similar. Also
it is not supported on a point, since then $\mu$ would be supported
on a translate of $\overline{e}_{2}$, contradicting our assumption
that $\mu$ is not supported on a quadratic curve. Thus, we know at
least that $\dim\pi_{1}\mu>0$. This is still far from (\ref{eq:pi-1-lower-bound-1}),
but one cannot in general do better without further information (see
discussion after Theorem \ref{thm:main-triangular}).

\subsubsection*{Projections and slices}

Due to the fact that $\eta^{*}=\delta_{\overline{e}_{1}}$ has dimension
zero, Theorem \ref{thm:BHR-projections} no longer holds. But $\eta^{*}=\delta_{\overline{e}_{1}}$
still attracts the random walks started from all initial lines. This,
and the inequality $\chi_{2}<\chi_{1}$ which we have assumed, mean
that the results in Section \ref{sec:Entropy} continue to hold as
stated. 

Note that in the case considered in \cite{BHR} (where $\overline{e}_{2}$
is contracted at asymptotic rate $2^{\chi_{1}}$ instead of $2^{\chi_{2}}$),
the situation was different: there, we did not have convergence to
$\eta^{*}$ from all initial lines, and so many analogous results
about projections needed to be modified to non-uniform variants.

\subsubsection*{The function $L$}

Because $\dim\eta^{*}=0$, Corollary \ref{cor:bourgain-projection-of-SA-measures}
is no longer valid. Nevertheless, we have added the assumption
\[
\beta=\dim\pi_{1}\mu\ge\min\{1,\dim\mu\},
\]
hence Propositions \ref{prop:singular pulbacks} and Proposition \ref{prop:const directions}
continue to hold. 

\subsubsection*{Algebraic arguments}

As noted in the introduction, in the triangular matrix case, the attractor
could be supported on a quadratic curve, and in such cases the dimension
can be smaller than the expected one even if the other hypotheses
hold. We have therefore added the condition that $\mu$ is not supported
on quadratic curve as one of the hypotheses of Theorem \ref{thm:main-triangular},
so Section \ref{subsec:The-mu-measure-of-quadratic-curves} is no
longer needed, except for the easy observation that if $\mu$ gave
positive mass to a quadratic curve, it would be supported on one.

For the non-affinity of $L$ that is proved in Section \ref{subsec:non-affinity-of L},
a few modifications are necessary:

In Lemma \ref{lem:commutes in action on P(R^2)}, the conclusion is
not as stated, but rather, that either $B$ is scalar, or else it
has rank $1$ and its image is the common eigenvector of the $A_{i}$,
namely, $\overline{e}_{2}$.

In Proposition \ref{prop:non-aff of L}, several modifications are
needed. First, as noted above, the fact that $\mu$ does not give
mass to quadratic curves follows from our assumptions, rather than
from Proposition \ref{prop:measure 0 for alg cur}. Second, when invoking
Lemma \ref{lem:commutes in action on P(R^2)}, one must deal with
the possibility that $\im(A_{\psi})=\overline{e_{2}}$. Supposing
that this is the case, it follows from equation (\ref{eq:parallel lines})
that $b_{\psi}\in\overline{e}_{2}$, but then $\overline{e_{2}}$
is an invariant line under all $\varphi_{i}$ and we conclude that
$\mu$ is supported on this line, contradicting again the assumption
that it is not supported on a quadratic curve.

\subsubsection*{Entropy growth}

The entropy growth result, Theorem \ref{thm:entropy growth-near-identity},
requires no change.

\subsubsection*{Bottom line}

The remainder of the proof can now proceed as it did for Theorem \ref{thm:main}.

\bibliographystyle{plain}
\bibliography{bib}

\noindent Einstein Institute of Mathematics, Edmond J. Safra campus,
The Hebrew University of Jerusalem, Israel; 

and

\noindent Institute for Advanced Study, Princeton, 1 Einstein Drive,
Princeton, NJ 08540, USA. 

michael.hochman@mail.huji.ac.il

\bigskip

\noindent Einstein Institute of Mathematics, Edmond J. Safra campus,
The Hebrew University of Jerusalem, Israel;

and

\noindent Centre for Mathematical Sciences, Wilberforce Road, Cambridge
CB3 0WA, UK. 

ariel.rapaport@mail.huji.ac.il
\end{document}